\documentclass[11pt]{amsart}

\usepackage{tikz-cd,amsmath,amssymb,amstext,amsgen,amsbsy,amsopn,amsfonts,graphicx,amsthm,semtkzX,mathdots,multicol}
\usepackage[alphabetic]{amsrefs}
\usepackage{titletoc}
\usepackage[linktocpage=true]{hyperref}
\usepackage{xcolor}
\definecolor{Burgundy1}{RGB}{128,0,32}
\hypersetup{
    colorlinks,
    linkcolor={Burgundy1!90!black},
    citecolor={teal!70!black},
    urlcolor={blue!80!black}
}
\usepackage{cleveref}

\usepackage[left=2.5cm, right=2.5cm, top=3.5cm, bottom=3cm,footskip=1cm,headsep=1.5cm]{geometry}
\usepackage[all]{xy}
\usepackage[OT2,T1]{fontenc}

\usepackage{float}

\crefname{equation}{}{}
\numberwithin{equation}{section}

\newtheorem{theorem}[equation]{Theorem}
\newtheorem{lemma}[equation]{Lemma}
\newtheorem{cor}[equation]{Corollary}
\Crefname{cor}{Corollary}{Corollaries}
\newtheorem{conjecture}[equation]{Conjecture}
\newtheorem{proposition}[equation]{Proposition}

\theoremstyle{remark}
\newtheorem{remark}[equation]{Remark}
\theoremstyle{definition}
\newtheorem{example}[equation]{Example}
\theoremstyle{definition}

\theoremstyle{definition}
\newtheorem{defi}[equation]{Definition}
\theoremstyle{definition}
\newtheorem{notation}[equation]{Notation}
\theoremstyle{definition}
\newtheorem{convention}[equation]{Convention}
\theoremstyle{definition}
\newtheorem{assumption}[equation]{Assumption}
\theoremstyle{definition}
\newtheorem{caution}[equation]{Caution}

\DeclareSymbolFont{cyrletters}{OT2}{wncyr}{m}{n}
\DeclareMathSymbol{\Sha}{\mathalpha}{cyrletters}{"58}
\def\K{\ensuremath K}
\def\L{\ensuremath L}

\def\W{\ensuremath\mathcal{W}}
\def\F{\ensuremath\mathbb{F}}
\def\Q{\ensuremath\mathbb{Q}}
\def\N{\ensuremath N_{L/K}}
\def\c{\ensuremath \mathfrak{s}}
\def\s{\ensuremath \mathfrak{s}}
\def\cR{\ensuremath \mathcal{R}}
\def\Z{\ensuremath \mathbb{Z}}
\def\ub{\"ubereven}

\address{School of Mathematics and Statistics, University of Glasgow, University Place, Glasgow, G12 8QQ.}
\email{ajmorgan44@gmail.com}

\begin{document}

\title{2-Selmer parity for hyperelliptic curves in quadratic extensions}

\author{Adam Morgan}
\setcounter{tocdepth}{1}
\maketitle

\begin{abstract}
We study the 2-parity conjecture for Jacobians of hyperelliptic curves over number fields. Under some mild assumptions on their reduction, we prove the conjecture over quadratic extensions of the base field. The proof proceeds via a generalisation of a formula of Kramer and Tunnell relating local invariants of the curve,  which may be of independent interest. A new feature of this generalisation is the appearance of terms which govern whether or not the Cassels--Tate pairing on the Jacobian is alternating, which first appeared in work of Poonen--Stoll. We establish the local formula in many instances and show that in  remaining cases it follows from standard global conjectures.
\end{abstract}

\tableofcontents
\vspace{-10pt}
\section{Introduction} 
Let $K$ be a number field and $A/K$ an abelian variety. Conjecturally, the corresponding completed $L$-function of $A/K$, $L^\star(A/K,s)$,  has an analytic continuation to the whole of the complex plane and satisfies a functional equation  
\[L^\star(A/K,s)=w(A/K)L^\star(A/K,2-s),\]
where $w(A/K)\in\{\pm 1\}$ is the global root number of $A/K$. The Birch and Swinnerton-Dyer conjecture asserts that the Mordell--Weil rank of $A/K$ agrees with the order of vanishing at $s=1$ of $L^\star(A/K,s)$:
\[\text{ord}_{s=1} L^\star(A/K,s){=}\text{rk}(A/K).\]
If $w(A/K)=1$ (resp. $-1$), then $L^\star(A/K,s)$ is an even (resp. odd) function around $s=1$ and as such its order of vanishing there is even (resp. odd). Thus a consequence of the Birch and Swinnerton-Dyer conjecture is the parity conjecture:
\[w(A/K){=}(-1)^{\text{rk}(A/K)}.\]

Essentially all progress towards the parity conjecture has proceeded via the $p$-parity conjecture. For a fixed prime $p$, denote by $\text{rk}_p(A/K)$ the $p$-infinity Selmer rank of $A/K$. Under the conjectural finiteness of the Shafarevich--Tate group (or indeed, under the weaker assumption that its $p$-primary part is finite), $\text{rk}_p(A/K)$ agrees with $\text{rk}(A/K)$. The $p$-parity conjecture is the assertion that \[w(A/K)=(-1)^{\text{rk}_p(A/K)}.\] Note that without knowing finiteness of the Shafarevich--Tate group, these conjectures are inequivalent for different primes $p$.

\subsection{Known cases of the $p$-parity conjecture}  
For  elliptic curves over $\mathbb{Q}$, Dokchitser--Dokchitser \cite{MR2680426} have shown that the $p$-parity conjecture holds for all primes $p$. Subsequently, Nekov{\'a}{\v{r}} \cite{MR3101073}  extended this result to all totally real number fields, excluding some elliptic curves with potential complex multiplication; these exceptional cases have recently been treated by Green--Maistret \cite{GM21}. For a general number field $K$, \v{C}esnavi\v{c}ius \cite{MR3552491} has shown that the $p$-parity conjecture holds for elliptic curves over $K$ possessing a $p$-isogeny, whilst work of Kramer--Tunnell \cite{MR664648} and Dokchitser--Dokchitser \cite{MR2831512} proves that the $2$-parity conjecture holds for an arbitrary elliptic curve $E/K$, not over $K$ itself, but over any quadratic extension of $K$.

For higher dimensional abelian varieties much less is known. The most general result at present is due to Coates, Fukaya, Kato and Sujatha, who prove in \cite{MR2551757} that for odd primes $p$, the $p$-parity conjecture holds for any abelian variety possessing a suitable $p$-power degree isogeny, subject to some further technical conditions. For $p=2$ the main result is due to  Dokchitser--Maistret \cite{DM19}, who prove the $2$-parity conjecture for quite general semistable abelian surfaces. 

\subsection{Main result}
Following on from the work of Kramer--Tunnell and Dokchitser--Dokchitser for elliptic curves, we consider the $2$-parity conjecture for Jacobians of hyperelliptic curves over quadratic extensions of their field of definition. Our main result is the following: 
   
\begin{theorem} \label{cases of the parity conjecture}
Let $K$ be a number field and $L/K$ a quadratic extension. Let $C/K$ be a hyperelliptic curve of genus $g\geq 2$ and let $J/K$ be the Jacobian of $C$. Suppose that $J$ has semistable reduction at each prime $\mathfrak{p}\nmid 2$ of $K$ which ramifies in $L/K$, and assume moreover that:

\begin{itemize}
\item for each prime $\mathfrak{p}\mid 2$ of $K$ which is inert in $L/K$, $J$ has good reduction at $\mathfrak{p}$,
\item for each prime $\mathfrak{p}\mid 2$ of $K$ which ramifies in $L/K$, $J$ has good ordinary reduction at $\mathfrak{p}$ and  $K_\mathfrak{p}(J[2])/K_\mathfrak{p}$ has odd degree.  
\end{itemize}
Then the $2$-parity conjecture holds for $J/L$. 
\end{theorem}

\begin{remark}
 \Cref{cases of the parity conjecture} gives a large supply of hyperelliptic curves satisfying the $2$-parity conjecture over every quadratic extension of their field of definition; see \Cref{explicit ordinary cor} for explicit conditions on a Weierstrass equation defining $C$ that ensure the conditions of \Cref{cases of the parity conjecture} at primes dividing $2$  are satisfied.
 \end{remark}
 
 \begin{remark}
If the genus of $C$ is $2$ then one can weaken the assumption that $J$ has good reduction at each inert prime dividing 2 to assume only that $J$ has semistable reduction at such primes; see \Cref{prop:unram_cases_of_conjecture}. 
\end{remark}
 
 \subsection{Reduction to a local question} \label{subsec:red_to_local}
 The proof of \Cref{cases of the parity conjecture} proceeds by reducing to a purely local question, as we now explain. 
 
In the notation of \Cref{cases of the parity conjecture},  for each place $v$ of $K$ which is non-split in $L$, denote by $\mathfrak{v}$ the unique place of $L$ extending $v$. Since $J$ is defined over $K$, the root number $w(J/L)$ decomposes as a product of local terms indexed by places of $K$ which are non-split in $L/K$:
\begin{equation} \label{root_number_decomp_intro}
w(J/L)=\prod_{\substack{v\textup{ place of }K\\ v\textup{ non-split in }L}} w(J/L_\mathfrak{v}),
\end{equation}
where $w(J/L_\mathfrak{v})\in \{\pm 1\}$ is the local root number of $J/L_\mathfrak{v}$.
The strategy to prove \Cref{cases of the parity conjecture} is to similarly decompose the parity of the $2$-infinity Selmer rank of $J$ over $L$ into local terms, and compare these place by place. Specifically, results of \cite{MR3951582} combined with work of Poonen--Stoll \cite{MR1740984} give a decomposition of the parity of $\text{rk}_{2}(J/L)$ into local terms as detailed below, generalising a theorem of Kramer \cite[Theorem 1]{MR597871} for elliptic curves. Before stating this decomposition we need to introduce some notation.

\begin{notation} \label{intro_notat_nrm}
 For each place $v$ of $K$ which does not split in $L$, define the \textit{local norm map}  
\[N_{L_{\mathfrak{v}}/K_v}:J(L_{\mathfrak{v}})\rightarrow J(K_{v})\] sending $P\in J(L_\mathfrak{v})$ to 
\[
  N_{L_{\mathfrak{v}}/K_v}(P)=\sum_{\sigma\in\text{Gal}(L_{\mathfrak{v}}/K_{v})}\sigma(P).
\]
Note that, as a quotient of $J(K_v)/2J(K_v)$, the cokernel of this map is a finite dimensional $\mathbb{F}_2$-vector space. 

Define also the invariant $\epsilon(C/K_v)\in \{0,1\}$  by setting
\[\epsilon(C/K_v)=\begin{cases}1~~&~~C/K_v \textup{ is deficient},\\ 0~~&~~\textup{otherwise.} \end{cases}\]
Here, following \cite[Section 8]{MR1740984}, we say that $C/K_v$ is deficient if $C$ has no $K_v$-rational divisor of degree $g-1$.
\end{notation}

The relevance of the invariant $\epsilon(C/K_v)$ comes from a result of Poonen and Stoll \cite[Theorem 8]{MR1740984} characterising the failure of the Shafarevich--Tate group of $J/K$ to have square order (if finite) in terms of the $\epsilon(C/K_v)$. Denoting by $C^L/K$ the quadratic twist of $C$ by $L$, we define $\epsilon(C^L/K_v)$ similarly. We then have the following decomposition of the parity of $\textup{rk}_2(J/L)$ into local terms:

\begin{theorem}[=\Cref{selmer decomposition}] \label{selmer decomposition_intro}
We have
\[
 (-1)^{\textup{rk}_{2}(J/L)}=\prod_{\substack{v\textup{ place of }K\\v\textup{ non-split in }L}}(-1)^{\epsilon(C/K_v)+\epsilon(C^L/K_v)+\dim J(K_{v})/N_{L_\mathfrak{v}/K_v}J(L_{\mathfrak{v}})}.\]
\end{theorem}

Ideally, one might hope that the local terms contributing to $w(J/L)$ and $(-1)^{\text{rk}_{2}(J/L)}$ simply agree place by place. However, this turns out not to be the case, and so the strategy hinges on identifying the discrepancy between these local terms as a quantity which vanishes globally. To this end we conjecture the following, generalising a formula of  Kramer--Tunnell \cite{MR664648} for elliptic curves:

\begin{conjecture} \label{Kramer Tunnell}  
Let $K$ be a local field of characteristic different from $2$. Let $L/K$ be a quadratic extension, let $C/K$ be a hyperelliptic curve, and denote by $J/K$ the Jacobian of $C$. Then we have
\[
w(J/L)=(\Delta_C,L/K)(-1)^{\epsilon(C/K)+\epsilon(C^L/K)+\dim J(K)/\N J(L)}.
\]
\end{conjecture} 
Here the quantity $\Delta_C$ is the discriminant of $f(x)$ for any Weierstrass equation $y^2=f(x)$ defining $C$, and $ (\Delta_C,L/K)\in \{\pm 1\}$ is the Hilbert/Artin symbol of $\Delta_C$ with respect to the extension $L/K$.\footnote{Given another Weierstrass equation $y^2=h(x)$ for $C$, the discriminants of $f(x)$ and $h(x)$   differ by a square in $K$, hence the term $(\Delta_C,L/K)$ is independent of the choice of Weierstrass equation.}
 
Returning now to the case where $L/K$ is a quadratic extension of number fields and $C$ is a hyperelliptic curve defined over $K$, by the product formula for Hilbert symbols we have 
\[\prod_{\substack{v\textup{ place of }K\\v\textup{ non-split in }L}}(\Delta_C,L_\mathfrak{v}/K_v)=1.\]
 In particular, we see from \eqref{root_number_decomp_intro} and \Cref{selmer decomposition_intro} that \Cref{Kramer Tunnell} implies the 2-parity conjecture for $J/L$. We will prove \Cref{Kramer Tunnell} under the assumptions on the reduction of $C$ appearing in the statement of \Cref{cases of the parity conjecture}, hence proving that result. Specifically, our second main result is the following:

\begin{theorem} \label{thm:cases_of_kramer_tunnell}
\Cref{Kramer Tunnell} holds in the following cases:
\begin{itemize}
\item $K=\mathbb{R}$,
\item $K$ has odd residue characteristic, and either $L/K$ is unramified or $J/K$ has semistable reduction,
\item $K$ is a finite extension of $\mathbb{Q}_2$, $L/K$ is unramified, and either $J/K$ has good reduction or $g=2$ and $J/K$ has semistable reduction,
\item $K$ is a finite extension of $\mathbb{Q}_2$,  $J/K$ has good ordinary reduction, and $K(J[2])/K$ has odd degree. 
\end{itemize}
\end{theorem}

\begin{remark}
More generally,  \Cref{Kramer Tunnell} holds if there is an odd degree Galois extension $F/K$ over which $C$ satisfies the conditions of \Cref{thm:cases_of_kramer_tunnell} with $L/K$ replaced by $FL/F$; see \Cref{compatibility results}.
\end{remark}

As further evidence for \Cref{Kramer Tunnell}, we show that the  cases above (and in fact substantially fewer) are sufficient to deduce \Cref{Kramer Tunnell} from the 2-parity conjecture via a global-to-local argument, at least for curves arising via base-change from a number field. 

\begin{theorem}[=\Cref{global to local}] \label{global to local_intro}
Let $K$ be a number field, $C/K$ a hyperelliptic curve, $J/K$ its Jacobian, and $v_0$ a place of $K$. If the $2$-parity conjecture holds for $J$ over every quadratic extension of $K$,   then \Cref{Kramer Tunnell} holds for $J/K_{v_0}$ and every quadratic extension $L/K_{v_0}$. 
\end{theorem}      

\begin{remark} \label{genus_0_KT}
We remark that \Cref{Kramer Tunnell} makes sense (and, surprisingly, is not entirely vacuous) in genus 0. Indeed, for a quadratic extension $L/K$ of local fields of characteristic different from $2$, consider a hyperelliptic curve $C:y^2=f(x)$ where $f(x)\in K[x]$ is a squarefree polynomial of degree $1$ or $2$. The Jacobian of $C$ is trivial, so the root number  and cokernel of the local norm map are trivial also. Further, $C/K$ (resp. $C^L/K$) is deficient if and only if it has no $K$-point. It is then easy to check that $(\Delta_C,L/K)=(-1)^{\epsilon(C/K)+\epsilon(C^L/K)}$ for any quadratic extension $L/K$. 
\end{remark}

\subsection{Comparison with work of Kramer--Tunnell}
\Cref{Kramer Tunnell}   has its origins in work of Kramer--Tunnell. Specifically, for a local field $K$, a separable quadratic extension $L/K$, and an elliptic curve $E/K$, Kramer--Tunnell \cite{MR664648} conjectured  the formula
\begin{equation}\label{eq:original_kramer_tunnell}
w(E/K)w(E^L/K)=(-\Delta_E,L/K)(-1)^{\dim E(K)/N_{L/K}E(L)},
\end{equation}
and proved it in many cases, including in every instance when $K$ has odd residue characteristic. This conjecture is now known in all cases thanks to subsequent work of Dokchitser--Dokchitser \cite{MR2831512} and \v{C}esnavi\v{c}ius--Imai \cite{MR3577895}.

By \cite[Proposition 3.11]{MR3552491} we have 
\[w(E/L)=w(E/K)w(E^L/K)(-1,L/K),\]
whilst $\epsilon(E/K)=0$ for every local field $K$ and elliptic curve $E/K$. Thus \Cref{Kramer Tunnell} specialises to \Cref{eq:original_kramer_tunnell} when $C/K$ is an elliptic curve. 

The presence of the new terms $\epsilon(C/K)$ and $\epsilon(C^L/K)$ in the purely local \Cref{Kramer Tunnell},  which are `forced' by global considerations concerning the possible failure of the Shafarevich--Tate group of a   principally polarised abelian variety   to have square order (see \Cref{2-Selmer Groups in Quadratic Extensions}), is a key new feature of this work. These terms also place constraints on possible proofs of \Cref{Kramer Tunnell}. Indeed,   $\epsilon(C/K)$ is not a function purely of the Jacobian of $C$ (as in \Cref{genus_0_KT}, $\epsilon(C/K)$ can be non-trivial even for curves of genus $0$!).  
  A lot of the technical difficulty in this work is involved in relating invariants defined in terms of the Jacobian of $C$,  such as the cokernel of the local norm map, to the invariants $\epsilon(C/K)$, $\epsilon(C^L/K)$ and $(\Delta_C,L/K)$, which have no obvious meaning for general abelian varieties. 
 
As above, the Kramer--Tunnell formula \eqref{eq:original_kramer_tunnell} is known to hold for local fields of characteristic $2$ and separable quadratic extensions $L/K$. It is thus tempting to extend the scope of \Cref{Kramer Tunnell} to include such extensions (especially in light of the  work of  \v{C}esnavi\v{c}ius--Imai \cite{MR3577895} who reduce \eqref{eq:original_kramer_tunnell} over local fields of characteristic $2$ to the corresponding conjecture for finite extensions of $\mathbb{Q}_2$). However, since we prove no instances of \Cref{Kramer Tunnell} over local fields of characteristic $2$ in this work, we have elected not to do this. 

\subsection{Overview of the paper}
 In \Cref{2-Selmer Groups in Quadratic Extensions} we explain how to deduce \Cref{selmer decomposition_intro} by combining results of \cite{MR3951582} with work of Poonen--Stoll \cite{MR1740984}. 

In \Cref{local norm section} we recall and prove some basic properties of the local norm map for general abelian varieties. Of particular use later is \Cref{norm map as Tamagawa numbers} which, for nonarchimedean local fields of odd residue characteristic, expresses the order of the cokernel of the local norm map in terms of Tamagawa numbers, generalising a result of Kramer--Tunnell  \cite[Corollary 7.6]{MR664648} for elliptic curves.

In \Cref{compatibility results} we prove some compatibility results concerning the behaviour of  \Cref{Kramer Tunnell} under quadratic twist, and under odd-degree Galois extension of the base field. 

Across  \Cref{2-tors sect,deficiency section} we collect and prove some basic results concerning, respectively, $2$-torsion in Jacobians of hyperelliptic curves, and criteria for determining when a hyperelliptic curve over a local field $K$ is deficient. Whilst much of this material is standard, \Cref{deficiency lemma}, which characterises deficiency for a particular class of hyperelliptic curves (essentially those with a $K$-rational theta characteristic) may be of independent interest.  

In \Cref{first cases} we combine the results of  \Cref{2-tors sect,deficiency section} to deduce some simple cases of \Cref{Kramer Tunnell}. Namely, we establish \Cref{Kramer Tunnell} when $K$ is nonarchimedean, and when  $K$ has odd residue characteristic and $J/K$ has good reduction. Then in \Cref{global to local section} we show that these special cases are already enough to deduce  \Cref{global to local_intro}.

With the exception of the short Sections \ref{residue characteristic 2} and \ref{sec:main_thm_proofs} (which, respectively, consider \Cref{Kramer Tunnell} for finite extensions of $\mathbb{Q}_2$, and tie together results from previous sections to prove \Cref{cases of the parity conjecture,thm:cases_of_kramer_tunnell}), the remainder of the paper splits into $2$ parts. Firstly, in \Cref{main unramified section,proof of compatibility} we consider \Cref{Kramer Tunnell}  when the extension $L/K$ is unramified,  proving it completely in this case when $K$ has odd residue characteristic. We do this by analysing the minimal proper regular model of $C$. 
The key fact making \Cref{Kramer Tunnell}  accessible here is that the formation of the minimal regular model commutes with unramified base change; this enables a comparison between invariants of $C$ and those of its unramified quadratic twist. 
 The central technical result of these sections is \Cref{homomorphism theorem}, which we formulate for general curves, and which shows that the quantity 
\[2^{\epsilon(C/K)}\frac{|\Phi(\bar{k})|}{|\Phi(k)|},\]
viewed as an element of $\mathbb{Q}^{\times}/\mathbb{Q}^{\times 2}$, 
behaves well under quite general twisting. Here $k$ is the residue field of $K$ and $\Phi$ is the N\'{e}ron component group  of the Jacobian of $C$. 
We would also like to advertise \Cref{tam computations}, which is a by-product of the proof of \Cref{homomorphism theorem}, and which gives a relatively simple way of computing the Tamagawa number of the Jacobian of an arbitrary curve, modulo rational squares,  as a function of its minimal regular model. This result plays a prominent role  in simplifying computations in \Cref{sec:ram_twist_hyp_curve}.

 Finally, across  Sections \ref{min reg model ramified} to \ref{completion_ram_quad_odd_res}, we prove  \Cref{Kramer Tunnell} when  $C/K$ has semistable reduction and when $L/K$ is a ramified quadratic extension of local fields with odd residue characteristic. Roughly speaking, once again our strategy is to encode each of the invariants appearing in  \Cref{Kramer Tunnell} in terms of the minimal proper regular models of both $C$ and $C^L$. However, since now $L/K$ is ramified,  the minimal regular model of $C^L$ can  be significantly different to that of $C$, making it hard to relate the relevant invariants. We overcome this by fixing a Weierstrass equation $y^2=f(x)$ for $C$ and drawing on the explicit description of the minimal regular models of $C$ and $C^L$ in terms of clusters (certain combinatorial objects encoding the distances between the roots of $f(x)$) afforded by the works  \cite{DDMM18} and \cite{MR4201122}. This essentially reduces \Cref{Kramer Tunnell}  to a purely combinatorial question about clusters, though one that still seems far from straightforward. We split the resulting analysis into two parts. First, in \Cref{h1 pi computation} we give an explicit description in terms of clusters of the group $\mathfrak{B}_{C/K}$ introduced by Betts--Dokchitser in \cite{MR3933907}; this group packages together information about the Tamagawa number of the Jacobian of $C$ over both $K$ and $L$, but seems simpler to describe than each of these quantities. Then in \Cref{sec:ram_twist_hyp_curve} we study the minimal regular model of $C^L$, describing in terms of clusters the Tamagawa number of the Jacobian of $C^L$ modulo rational squares; see \Cref{main_ram_quad_twist_cor}. Finally, in \Cref{completion_ram_quad_odd_res} we combine these results to establish the sought case of  \Cref{Kramer Tunnell}.

\subsection*{Notation and conventions}

For a field $K$ we denote by $\bar{K}$ a (fixed once and for all) algebraic closure of $K$, and denote by $K^s\subseteq \bar{K}$ the separable closure of $K$. We denote by $G_K=\textup{Gal}(K^s/K)$ the absolute Galois group of $K$.

\subsubsection{Hyperelliptic curves} \label{convention}
By  a hyperelliptic curve $C$ over a field $K$ we mean a smooth, proper, geometrically connected curve of genus $g \geq 2$, defined over $K$, and admitting a finite separable morphism  $C\rightarrow \mathbb{P}^1_K$ of degree $2$. When $K$ has characteristic different from $2$, one can always find a separable polynomial $f(x)\in K[x]$ of degree $2g+1$ or $2g+2$ such that $C$ is isomorphic to the curve given by gluing the affine schemes
\[U_1=\text{Spec}\frac{K[x,y]}{y^2-f(x)}\quad \textup{ and }\quad U_2=\text{Spec}\frac{K[u,v]}{v^2-u^{2g+2}f(1/u)},\]
  via the relations $x=1/u$ and $y=x^{g+1}v$. 
 By an abuse of notation we say that $C$ is given by the Weierstrass equation $y^2=f(x)$, and refer to elements of $U_2(\bar{K}) \setminus U_1(\bar{K})$ as the points at infinity. There are 2 such points if $\text{deg}(f)$ is even, and 1 if $\text{deg}(f)$ is odd. We denote by $\iota$ the hyperelliptic involution of $C$. For $C:y^2=f(x)$ this is the automorphism $(x,y)\mapsto (x,-y)$.  
 
 When $\textup{char}(K)\neq 2$, we define the discriminant $\Delta_C\in K^\times$ of a hyperelliptic curve given by a Weierstrass equation $C:y^2=f(x)$  by the formula given in   \cite[Seciton 2]{MR1363944}. One sees from that work that, up to squares in $K^\times$, this both agrees with the polynomial discriminant of $f(x)$ and is independent of the choice of Weierstrass equation for $C/K$. In particular, we will often consider $\Delta_C\in K^\times/K^{\times 2}$ without reference to a Weierstrass equation for $C$. Further, if we write $f(x)=c_ff_0(x)$ where $c_f$ is the leading coefficent of $f(x)$ and $f_0(x)$ is monic, then the discriminants of $f(x)$ and $f_0(x)$ differ by $c_f^{2\textup{deg}
(f)-2}$, hence agree modulo squares in $K$. In particular, the class $\Delta_C\in K^\times/K^{\times 2}$ does not feel the leading coefficient of $f(x)$.

\subsubsection{Quadratic twists} \label{quad_twist_hyp_convention}
Let $K$ be a field of characteristic different from $2$, and $L/K$ a quadratic extension. For a hyperelliptic curve $C/K$ we denote by $C^L/K$ the quadratic twist of $C$ by $L/K$. This is the twist of $C/K$ corresponding to the $1$-cocycle 
\[\textup{Gal}(L/K)\stackrel{\sim}{\longrightarrow}\{1,\iota\}\leq \textup{Aut}_L(C).\]
Suppose that  $C/K$ is given by a Weierstrass equation $y^2=f(x)$, and that $L=K(\sqrt{d})$ for some $d\in K^\times$. Then $C^L/K$ is given by the Weierstrass equation $y^2=df(x)$. In particular, it follows from the discussion on hyperelliptic discriminants above that, as elements  of $K^{\times}/K^{\times 2}$, we have $\Delta_C=\Delta_{C^L}$.

For an abelian variety $A/K$ we similarly denote by $A^L/K$ the quadratic twist of $A$ by $L/K$, which corresponds to the $1$-cocycle 
\[\textup{Gal}(L/K)\stackrel{\sim}{\longrightarrow}\{\pm 1\}\leq \textup{Aut}_L(A).\]
Denote by $\chi:G_K\rightarrow \{\pm 1\}$ the quadratic character corresponding to $L/K$. Then there is a $K^s$-isomorphism $\psi:A\stackrel{\sim}{\longrightarrow}A^L$ such that,  for all $\sigma \in G_K$, the composition $\psi^{-1}\circ {}^\sigma\psi$ is multiplication by $\chi(\sigma)$ on $A$, where ${}^\sigma \psi$ denotes the unique isomorphism $A\rightarrow A^L$ acting as $\sigma \circ \psi \circ \sigma^{-1}$ on $K^s$-points.  In particular, $\psi$ restricts to an isomorphism of $G_K$-modules $A[2]\cong A^L[2]$.

Since the hyperelliptic involution on $C$ induces multiplication by $-1$ on its Jacobian $J/K$,   the Jacobian of $C^L/K$ coincides with $J^L/K$.

\subsubsection{Galois cohomology}

For a profinite group $G$, a discrete $G$-module $M$, and integer $i\geq 0$, we denote by $H^i(G,M)$ the $i$-th cohomology group of $G$ with coefficients in $M$,  as defined in e.g. \cite{MR0225922}. We denote by $M^G$ the subgroup of elements of $M$ fixed by $G$. For $g\in G$ we denote by $M^g$ the subgroup of elements fixed by $g$. 

 When $G=G_K$ for a field $K$ we will often write $H^i(K,M)$ in place of $H^i(G_K,M)$. Similarly, for a Galois extension $L/K$ and a discrete $\textup{Gal}(L/K)$-module $M$, we often write $H^i(L/K,M)$ in place of $H^i(\textup{Gal}(L/K),M)$. 

\subsubsection{Notation for number fields and local fields}
For a number field $K$ we denote by $\mathcal{O}_K$ the ring of integers of $K$. For a  place $v$ of $K$, $K_v$ will denote the corresponding completion.    

By a local field $K$ we mean a locally compact valued field. Thus $K$ is isomorphic (as a valued field) to one of $\mathbb{R}$, $\mathbb{C}$, or a finite extension of either $\mathbb{Q}_p$ or $\mathbb{F}_p((t))$ for a prime $p$. For a nonarchimedean local field $K$ we take the following notation:

$\begin{array}{ll}
\mathcal{O}_\K &  \text{ring of integers of}~\K, \\
k &  \text{residue field of}~\K, \\
\pi & \text{a choice of uniformiser of }K,\\
v:\bar{K}^{\times}\rightarrow \mathbb{Q}  & \text{valuation on }\bar{K}\text{ normalised with respect to }K, \text{ so that }v(\pi)=1, \\
\mathfrak{m} &\textup{maximal ideal of the ring of integers of }\bar{K},\\
\K^\text{nr} & \text{maximal unramified extension of}~\K, \\
(a,\L/\K) & \text{Artin symbol of}~ a\in \K^\times~\text{in a Galois extension}~\L/\K\text{. We will usually take}~\L/\K \\   
&  \text{quadratic, in which case we regard this symbol as being valued in}~\{\pm 1\}.\\
\end{array}$
\bigskip

\subsubsection{Notation for  curves  and abelian varieties} \label{Ab_var_notat_intro}
For a smooth, proper, geometrically connected  curve $X$ over a  local field $K$, we define $\epsilon(X/K)\in \{0,1 \}$ to be equal to $1$ if $X$ is deficient over $K$, and equal to $0$ else. Thus $\epsilon(X/K)=1$ if and only if $X$ has a $K$-rational divisor of degree $g-1$, where $g$ is the genus of $X$.

Throughout the paper, for a field $K$, $C/K$ will almost always denote a hyperelliptic curve over $K$, $g$ will denote the genus of $C$, and $J/K$ will denote the Jacobian of $C$. 

For an abelian variety $A$ over a field $K$ (usually the Jacobian of a hyperelliptic curve $C$), we take the following notation.

\smallskip

\textit{For $K$ a number field:}
\smallskip

$\begin{array}{ll}
\textup{rk}_2(A/K) & \text{the }2\text{-infinity Selmer rank of }A/K,\\
\textup{Sel}^2(A/K) &\textup{the }2\text{-Selmer group of }A/K,\\ 
\Sha(A/K) & \text{the Shafarevich--Tate group of}~A/K,\\
\Sha_{\textup{nd}}(A/K) & \text{the quotient of}~\Sha(A/K)~ \text{by its maximal divisible subgroup},\\
w(A/K) & \text{the global root number of}~A/K.\\
\end{array}$
\smallskip

\textit{For $K$ a nonarchimedean local field:}
\smallskip

$\begin{array}{ll}
\Phi &   \text{the component group of the special fibre of the N\'{e}ron model of }A/K;\\
& \text{we often refer to this as the N\'{e}ron component group of $A$.}\\
c(A/\K) &  \text{the Tamagawa number of }A/K.\text{ By definition this is the order of the group }\Phi(k)\\ 
&\text{ of }k\text{-rational points of }\Phi,\\
w(A/K) & \text{the local root number of }A/K,\\
N_{L/K} & \text{for }L/K\text{ separable quadratic, denotes the norm map }A(L)\rightarrow A(K)\\
& \text{ sending }P\in A(L)\text{ to }N_{L/K}:=\sum_{\sigma \in \textup{Gal}(L/K)}\sigma(P).
\end{array}$

\subsection*{Acknowledgements}
A significant part of this work appears in the author's PhD thesis \cite{AMo15} at the University of Bristol. I would like to thank my advisor Tim Dokchitser for suggesting the problem, for constant encouragement, and for many helpful suggestions and comments. I would also like to thank K\k{e}stutis \v{C}esnavi\v{c}ius for several important comments and correspondence during that time. 

Compared to \cite{AMo15}, the main results are strengthened by drawing on the works  \cite{DDMM18} and \cite{MR4201122}, and I am grateful to Omri Farragi and Sarah Nowell for answering my questions about their work \cite{MR4201122}. I would also like to thank Alex Bartel, L. Alexander Betts, Vladimir Dokchitser, Qing Liu and C\'{e}line Maistret for helpful conversations.

Parts of this work were completed while the author was supported by the Engineering and Physical Sciences Research Council (EPSRC) grants EP/M016846/1 `Arithmetic of hyperelliptic curves', and  EP/V006541/1 `Selmer groups, Arithmetic Statistics and Parity Conjectures'. 

\section{2-Selmer groups in quadratic extensions} \label{2-Selmer Groups in Quadratic Extensions}

 In this section we combine results of \cite{MR3951582} and \cite{MR1740984} to deduce \Cref{selmer decomposition_intro}. Let $L/K$ be a quadratic extension of number fields, let $C/K$ be a hyperelliptic curve, and let $J/K$ denote the Jacobian of $C$.  Further, denote by $\textup{rk}_2(J/K)$ the $2$-infinity Selmer rank of $J/K$, and recall from \Cref{intro_notat_nrm} the definitions of the local norm map and the invariant $\epsilon(C/K_v)$ for a place $v$ of $K$. Let $C^L/K$ (resp. $J^L/K$) denote the quadratic twist of $C$ (resp. $J$) by $L/K$.

\begin{theorem}[=\Cref{selmer decomposition_intro}] \label{selmer decomposition}
We have
\[
 \textup{rk}_{2}(J/L)\equiv \sum_{\substack{v\textup{ place of }K\\v\textup{ non-split in }L/K}}\Big(\epsilon(C/K_v)+\epsilon(C^L/K_v)+\dim J(K_{v})/N_{L_\mathfrak{v}/K_v}J(L_{\mathfrak{v}}) \Big)~~\quad(\textup{mod }2).\]
\end{theorem}

\begin{proof}
 By \cite[Theorem 10.12]{MR3951582} we have 
\[\dim\textup{Sel}^2(J/K)+\dim\textup{Sel}^2(J^L/K)\equiv \sum_{\substack{v\textup{ place of }K\\v\textup{ non-split in }L/K}}\dim J(K_{v})/N_{L_\mathfrak{v}/K_v}J(L_{\mathfrak{v}})~~(\textup{mod }2)\]
where here $\textup{Sel}^2(J/K)$ denotes the $2$-Selmer group of $J/K$ (and similarly for  $J^L/K$).  
Consequently (cf. \cite[proof of  Theorem 10.20]{MR3951582}) we have 
\[\textup{rk}_2(J/L)\equiv \dim\Sha_{\textup{nd}}(J/K)[2]+\dim\Sha_{\textup{nd}}(J^L/K)[2]+\sum_{\substack{v\textup{ place of }K\\v\textup{ non-split in}L}}\dim J(K_{v})/N_{L_\mathfrak{v}/K_v}J(L_{\mathfrak{v}})~\quad~(\textup{mod }2),\]
where  $\Sha_{\textup{nd}}(J/K)$ denotes the quotient of the Shafarevich--Tate group of $J/K$ by its maximal divisible subgroup.  
 It follows from \cite[Theorem 11]{MR1740984} that 
\[ \dim\Sha_{\textup{nd}}(J/K)[2] \equiv  \sum_{v\textup{ place of }K}\epsilon(C/K_v)\quad\quad \textup{ and }\quad\quad \dim\Sha_{\textup{nd}}(J^L/K)[2] \equiv \sum_{v\textup{ place of }K}\epsilon(C^L/K_v),\]
where both congruences are modulo $2$. For  the second equality we are using that the Jacobian of $C^L$ coincides with the quadratic twist $J^L$.
Since $C$ and $C^L$ are isomorphic over $K_v$ for each place $v$ that splits in $L/K$, the result follows.
\end{proof} 

\begin{remark}
One of the key reasons for working with Jacobians of hyperelliptic curves in this paper is that the quadratic twist $J^L$ is again the Jacobian of an explicit curve: the quadratic twist $C^L$. This allows us to give an explicit description of the parity of both $ \dim\Sha_{\textup{nd}}(J/K)[2] $ and $\dim\Sha_{\textup{nd}}(J^L/K)[2] $ in terms of deficiency. 
\end{remark}

\section{Basic properties of the local norm map} \label{local norm section}

In this section we prove some basic properties of the cokernel of the local norm map. Take $K$ to be a local field of characteristic different from $2$, and let $L/K$ be a quadratic extension. We work with arbitrary principally polarised abelian varieties since everything goes through in this setting. Thus for now we fix a principally polarised abelian variety $A/K$, and denote by $A^L$ the quadratic twist of $A$ by $L/K$. Denote by $N_{L/K}:A(L)\rightarrow A(K)$ the local norm map, sending $P\in A(L)$ to $\sum_{\sigma \in \textup{Gal}(L/K)}\sigma(P)$. Further, denote by $\chi:G_K\rightarrow \{\pm 1\}$ the quadratic character corresponding to $L/K$.

\begin{lemma} \label{cokernel of mult by 2}
We have 
\[\dim A(K)/2A(K) = \dim A^{L}(K)/2A^{L}(K) .\]
\end{lemma}

\begin{proof}
Let $\delta:A(K)/2A(K)\hookrightarrow H^1(K,A[2])$ be the connecting map associated to the multiplication-by-$2$ Kummer sequence for $A$. By \cite[Proposition 4.10]{MR2833483} the image of $\delta$ is a maximal isotropic subspace of $H^1(K,A[2])$ with respect to the pairing coming from cup-product and the local invariant map. Since $A[2]\cong A^L[2]$ as $G_K$-modules this gives
\[\dim A(K)/2A(K)=\frac{1}{2}\dim H^1(K,A[2])=\dim A^L(K)/2A^L(K),\]
as desired.
\end{proof}

From the definition of the quadratic twist $A^L$ we have an $L$-isomorphism $\psi:A\stackrel{\sim}{\longrightarrow} A^L$ such that, for all $\sigma \in G_K$, the composition $\psi^{-1}\circ {}^\sigma\psi$ is multiplication by $\chi(\sigma)$ on $A$. The map $\psi^{-1}$ identifies $A^L(L)$ with $A(L)$, and identifies $A^L(K)$ with $\text{ker}\left(\N:A(L)\rightarrow A(K)\right)$.  
The local norm map $A^{L}(L)\rightarrow A^{L}(K)$ then  identifies with the map sending $P\in A(L)$ to $P-\sigma(P)$. To avoid confusion, we denote this map by $\N^{L}$. 

\begin{lemma} \label{basic properties of norm map} 
The group $A(K)/\N A(L)$ is a finite dimensional $\mathbb{F}_{2}$-vector space and  
\[
\dim A(K)/\N A(L) =\dim A^{L}(K)/\N^L A(L).
\]
\end{lemma}

\begin{proof}
That $A(K)/\N A(L)$ is a finite dimensional $\mathbb{F}_2$-vector space follows from the fact that $2A(K)\subseteq \N A(L)$  along with the well-known finiteness of $A(K)/2A(K)$. Next, consider the map
\[\theta:  \N A(L) /2A(K)\longrightarrow \N^{L}A(L)/2A^{L}(K)\]
sending $\N(P)$ to $\N^{L}(P)$. This is readily checked to be a (well defined)  isomorphism. The result now follows from \Cref{cokernel of mult by 2}.
\end{proof}

Now let $\textup{Res}_{L/K}A$ denote the Weil restriction of scalars of $A$ from $L$ to $K$. This is an abelian variety over $K$ of dimension $2\textup{dim}A$ which represents the functor $T\mapsto A(T\times_K L)$. As explained in \cite[Section 2]{MR0330174} (see also \cite[Proposition 4.1]{MR2331769}), denoting by $\gamma$ the involution of $A\times A$ swapping the factors,   $\textup{Res}_{L/K}A$ can be described as the twist of $A\times A$ corresponding to the $1$-cocycle $G_K\rightarrow \textup{Aut}_{K^s}(A\times A)$ defined by
\[\sigma \mapsto \begin{cases}\textup{id}~~&~~\chi(\sigma)=1,\\ \gamma~~&~~\chi(\sigma)=-1. \end{cases}\]
This  identifies $(\textup{Res}_{L/K}A)(K)$ with the $L$-points of $A$ diagonally embedded in $A(\bar{K})\times A(\bar{K})$, realising the functor of points description for $T=K$.

As above, both $\textup{Res}_{L/K}A$ and $A\times A^L$ are twists of $A\times A$, and one checks that the endomorphism of $A\times A$  given by $(P,Q)\mapsto (P+Q,P-Q)$ descends to an isogeny $\phi: \textup{Res}_{L/K}A \rightarrow A \times A^L$. On $K$-points this is just the map 
\begin{equation} \label{eq:phi_on_K_points}
(N_{L/K},N^L_{L/K}):(\textup{Res}_{L/K}A)(K)=A(L)\longrightarrow A(K)\times A^L(K).\end{equation}
(See \cite[Sections 4 and 5]{MR2331769} for generalisations of this isogeny when $L/K$ is replaced by a  general finite Galois extension.)

We exploit the  isogeny $\phi$ to prove the final lemma of this section, which expresses the cokernel of the local norm map in terms of Tamagawa numbers. The special case of this for elliptic curves is due to Kramer and Tunnell \cite[Corollary 7.6]{MR664648}, although the proof is different.   Recall from \Cref{Ab_var_notat_intro} that $c(A/K)$ denotes the Tamagawa number of $A/K$. 

\begin{lemma} \label{norm map as Tamagawa numbers}
Assume that the residue characteristic of $K$ is odd. Then

\begin{equation*}
\dim A(K)/\N A(L)=\textup{ord}_2 \frac{c(A/K)c(A^L/K)}{c(A/L)}.
\end{equation*}
\end{lemma}

\begin{proof}
To ease notation write $X=\text{Res}_{L/K}A$ and $Y=A\times A^{L}$. With $\phi$ as above, since $K$ has odd residue characteristic  it follows from  a formula of Schaefer \cite[Lemma 3.8]{MR1370197} that  
\begin{equation} \label{eq:tam_cong_iden_1}
\textup{ord}_2\frac{\big|Y(K)/\phi X(K)\big|}{\big|X(K)[\phi]\big|}=\textup{ord}_2 \frac{c(Y/K)}{c(X/K)}=\textup{ord}_2 \frac{c(A/K)c(A^L/K)}{c(A/L)},
\end{equation}
the last equality following from  \cite[Proposition 3.19]{MR2961846} (see also \cite[Proof of Proposition 2]{MR0330174}).  From the description \eqref{eq:phi_on_K_points} of the map $\phi$ on $K$-points, one sees that $X(K)[\phi]\cong A(K)[2]$ and that we have a short exact sequence
\begin{equation} \label{tam_congruence_seq_2}
0\longrightarrow A^L(K)/2A^L(K)\longrightarrow Y(K)/\phi X(K)\longrightarrow A(K)/N_{L/K}A(L)\longrightarrow 0,
\end{equation}
the first map induced by inclusion into the second factor, and the second map being the projection onto the first factor. Since $K$ has odd residue characteristic  we have  (cf. \cite[Proposition 3.9]{MR1370197} for example)
\[\dim A^L(K)/2A^L(K)=\dim A^L(K)[2] \quad (=\dim A(K)[2]) .\]
 The result now follows by combining this last observation with \eqref{eq:tam_cong_iden_1} and \eqref{tam_congruence_seq_2}. 
\end{proof}

\section{Compatibility results} \label{compatibility results}

In this section we prove several compatibility results for \Cref{Kramer Tunnell}. These provide some evidence in favour of the conjecture and will also be used to make some reductions as part of the proof of \Cref{thm:cases_of_kramer_tunnell}. 

In what follows, $K$ denotes a local field of characteristic different from $2$. Let $L/K$ be a quadratic extension and let $C/K$ be a hyperelliptic curve. 

\subsection{Odd degree Galois extensions}

\begin{lemma}\label{odd degree extension}
Every individual term in \Cref{Kramer Tunnell} is unchanged under odd degree Galois extension of the base field. In particular, if $F/K$ is an odd degree Galois extension, then \Cref{Kramer Tunnell} holds for $C/K$ and the extension $L/K$  if and only if it holds for $C/F$ and the extension $LF/F$.  
\end{lemma}

\begin{proof}
That the term $(\Delta_C,L/K)$ is invariant under odd degree extensions (not necessarily Galois) is standard. Similarly, it's not hard to show that the terms involving deficiency of $C$ and its twist are also individually invariant under arbitrary odd degree extensions (cf. \Cref{deficiency in extensions} for a more general result which implies this). The statement for each of the root numbers is also standard; see  for example  \cite[Lemma A.1~and~Proposition A.2]{MR2534092} or \cite[Proposition 3.4]{MR664648}. For the cokernel of the local norm map, the statement for elliptic curves is \cite[Proposition 3.5]{MR664648} and the argument for general abelian varieties is identical.
\end{proof}  

\subsection{First compatibility with quadratic twist}

\begin{lemma} \label{twist 1}
  \Cref{Kramer Tunnell} holds for $C/K$ and the extension $L/K$ if and only if it holds for $C^L/K$ and the same extension. 
\end{lemma}

\begin{proof}
Since the root number and terms involving deficiency appear symmetrically between $J$ and $J^L$ in \Cref{Kramer Tunnell}, it suffices to show that 
\[(\Delta_C,L/K)=(\Delta_{C^L},L/K)\]
 and
\[\dim J(K)/\N J(L)\equiv \dim J^L(K)/\N J(L) ~ (\text{mod 2}).\]
For the first equality one checks readily that $\Delta_C$ and $\Delta_{C^L}$ lie in the same class in $K^\times/K^{\times 2}$ (cf. \Cref{quad_twist_hyp_convention}). The second statement follows from \Cref{basic properties of norm map}.
\end{proof}       

\subsection{Second compatibility with quadratic twist}

The second compatibility result involving quadratic twist is more subtle. That such a compatibility result should exist for elliptic curves was  discussed in the original paper  of Kramer and Tunnell \cite[remark following Proposition 3.3]{MR664648} and the result was later proven (again for elliptic curves) by Klagsbrun, Mazur and Rubin \cite[Lemma 5.6]{MR3043582}. The key step in that proof is to establish the following congruence. In order to state it, fix distinct quadratic extensions $L_1/K$ and $L_2/K$, and denote by $L_3/K$   the third quadratic subextension of $L_1L_2/K$. 

\begin{lemma} \label{norm_max_iso_congruence}
Let  $A/K$ be a principally polarised abelian variety.  Then we have
\[\dim A(K)/N_{L_1/K}A(L_1)+\dim A(K)/N_{L_2/K}A(L_2) \equiv  A^{L_1}(K)/N_{L_3/K}A^{L_1}(L_3)~~\textup{(mod 2)}.\]
\end{lemma}

\begin{proof}
The case where $A/K$ is an elliptic curve is \cite[Lemma 5.6]{MR3043582} and the argument is essentially the same. Let $L_0=K$ and for each $i=1,2,3$ identify $A^{L_i}[2]$ with $A[2]$ as $G_K$-modules in the usual way. For each $i$ let   $X_i$ denote the image of $A^{L_i}(K)/2A^{L_i}(K)$ under the map 
\[A^{L_i}(K)/2A^{L_i}(K)\stackrel{\delta^{L_i}}{\longrightarrow} H^1(K,A^{L_i}[2])=H^1(K,A[2]),\]
where $\delta^{L_i}$ is the connecting homomorphism associated to the multiplication-by-$2$ Kummer sequence for $A^{L_i}$. By \cite[Proposition 5.2]{MR2373150}, for $i=1,2,3$ we have 
\[A(K)/N_{L_i/K}A(L_i)\cong X_0/(X_0\cap X_i).\]
Similarly, we have
\[A^{L_1}(K)/N_{L_3/K}A^{L_1}(L_3)\cong X_1/(X_1\cap X_2).\]
In the elliptic curve case treated in \cite{MR3043582} it is shown that each $X_i$ is a maximal isotropic subspace with respect to a certain quadratic form on $H^1(K,A[2])$. The result is then deduced from \cite[Corollary 2.5]{MR3043582} which is a general result concerning the parity of the dimension of  intersections of maximal isotropic subspaces. For general principally polarised abelian varieties, the fact that each $X_i$ is a maximal isotropic subspace for the natural generalisation of this quadratic form   is detailed in \cite[Section 10.1]{MR3951582}.
The one difference from the case of elliptic curves is that now the quadratic form (in general) takes values in $\mathbb{Z}/4\mathbb{Z}$, rather than just   $\mathbb{Z}/2\mathbb{Z}$ as is assumed in \cite[Corollary 2.5]{MR3043582}. However, one readily verifies that this assumption is not used in the proof of  \cite[Corollary 2.5]{MR3043582}.
\end{proof}
 
We now return to the case where $C/K$ is a hyperelliptic curve and $J/K$ is its Jacobian. 

\begin{cor} \label{twist 2}
 \Cref{Kramer Tunnell} for $J/K$ and the extensions $L_1/K$ and $L_2/K$  implies \Cref{Kramer Tunnell} for $J^{L_1}/K$ and the extension $L_3/K$.
\end{cor}

\begin{proof}
By \cite[Proposition 3.11]{MR3552491}, for any quadratic extension $L/K$ we have 
\begin{equation} \label{rooty_root_formula}
w(J/L)=\left((-1)^g,L/K\right)w(J/K)w(J^L/K),
\end{equation}
where $g$ is the genus of $C$ (the cited result is only stated for elliptic curves, but the proof generalises verbatim to give the claimed formula). 
From \eqref{rooty_root_formula} it follows that
\[w(J/L_1)w(J/L_2)=w(J^{L_1}/L_3).\]
Further, by standard properties of Hilbert symbols and the fact that the discriminants of $C$ and any quadratic twist of $C$ differ by squares, we have 
\[(\Delta_C,L_1/K) (\Delta_C,L_2/K) =(\Delta_{C^{L_1
}},L_3/K).\]
Since we also have 
\begin{eqnarray*}
 \epsilon(C/K)+\epsilon(C^L_1/K) + \epsilon(C/K)+\epsilon(C^{L_2}/K) &\equiv&\epsilon(C^{L_1}/K)+\epsilon(C^L_2/K) ~~(\textup{mod }2) \\&=&\epsilon(C^{L_1}/K)+\epsilon((C^{L_1})^{L_3}/K),
 \end{eqnarray*}
the result follows from \Cref{norm_max_iso_congruence}.
\end{proof}

\begin{remark}
For a local field $K$ and hyperelliptic curve $C/K$, it follows from \Cref{twist 1} and \Cref{twist 2} that, in order to prove \Cref{Kramer Tunnell} for $C/K$ and all quadratic extensions of $K$, it suffices to prove the same result but with $C$ replaced by an arbitrary quadratic twist.
\end{remark}

\section{Two torsion in the Jacobian of a hyperelliptic curve} \label{2-tors sect}

For this section let $K$ be a field of characteristic different from $2$. Let  $C/K:y^2=f(x)$ be a hyperelliptic curve of genus $g$ and let $J/K$ be its Jacobian. Denote by $\mathcal{W}$  the $G_K$-set of ramification points of the $x$-coordinate morphism $C\rightarrow \mathbb{P}^1$. Thus $\mathcal{W}$ consists of the points $(r,0)$ for $r$ a root of $f(x)$, along with the unique point at infinity on $C$ if $\textup{deg}(f)$ is odd. As $G_K$-modules we then have 
\begin{equation}\label{2_tors_g_mod_desc}
J[2]\cong \ker\left(\mathbb{F}_2^{\mathcal{W}}\stackrel{\Sigma}{\longrightarrow}\mathbb{F}_2\right)/ \mathbb{F}_2  D.
\end{equation}
 Here $\mathbb{F}_2^{\mathcal{W}}$ denotes the permutation representation over $\mathbb{F}_2$ on the elements of $\mathcal{W}$, $\Sigma:\mathbb{F}_2^{\mathcal{W}}\rightarrow \mathbb{F}_2$ denotes the sum-of-coefficients map, and $D=\sum_{w\in\W} w$. See \cite[Section 6]{MR1465369} for more details.  
Noting that $g\geq 2$, hence $|\mathcal{W}|> 4$,  we see from the above description that $K(J[2])/K$ is the splitting field of $f(x)$. 

We now compute the dimension of the rational $2$-torsion $J(K)[2]$. The case where $K(J[2])/K$ is cyclic is treated  in \cite[Theorem 1.4]{MR1865865} (but note the erratum \cite{MR2169307}) whilst the case where $f(x)$ has an odd degree factor over $K$ is \cite[Lemma 12.9]{MR1465369}. We will require a slightly more general statement. 

In what follows we write $f(x)=c_ff_0(x)$, where $c_f\in K^\times$ is the leading coefficient of $f(x)$, and $f_0(x)\in K[x]$ is monic. To clean up the statement we also make the following convention.

\begin{convention} \label{convention_hyp_factor}
In what follows, if $\text{deg}(f)$ is odd, the rational point at infinity on $C$ is to be interpreted as an odd degree irreducible factor of $f(x)$ over $K$.
\end{convention}

\begin{lemma} \label{two torsion}
Let $n$ be the number of irreducible factors of $f(x)$ over $K$ (see \Cref{convention_hyp_factor} above). If $f(x)$ has an odd degree factor over $K$ then \[\dim J(K)[2]=n-2.\] Otherwise, if each irreducible factor of $f(x)$  has even degree, let $F/K$ be the splitting field of $f(x)$ and let $m$ be the number of quadratic subextensions of $F/K$ over which $f_0(x)$ factors as a product of $2$ distinct conjugate polynomials. Then
\[\dim J(K)[2]= \begin{cases}
n-1 & ~~g~\text{even},\\
n-1+\textup{ord}_2(1+m) & ~~g~\text{odd}.
\end{cases}
\]
\end{lemma}

\begin{proof}
Denote by $G$ the Galois group of $F/K$ and let $M$ be the $G$-module $M=\ker\big(\mathbb{F}_2^\mathcal{W}\stackrel{\Sigma}{\longrightarrow}\mathbb{F}_2\big)$. Then by \eqref{2_tors_g_mod_desc} we have an exact sequence 
\begin{equation}\label{exact sequence}
0\longrightarrow \mathbb{F}_2D \longrightarrow M^G \longrightarrow J[2]^G\longrightarrow \ker\left(H^1(G,\mathbb{F}_2D)\rightarrow H^1(G,M)\right)\longrightarrow 0.
\end{equation}
Now
\[\dim M^G= \ker\big((\F_2^\mathcal{W})^G\stackrel{\Sigma}{\longrightarrow}\F_2\big)=\begin{cases} n-1 & ~~f(x) ~ \text{has an odd degree factor over} ~K,\\  n & ~~\text{else.}\end{cases}\]
Consequently, we must show that $\ker\left(H^1(G,\mathbb{F}_2D)\rightarrow H^1(G,M)\right)$ has dimension $0$ or $\textup{ord}_2(1+m)$ according, respectively, to whether $g$ is even or odd  (note that if $f(x)$ has an odd degree factor over $K$  then $m=0$). 

Now $H^1(G,\F_2D)=\text{Hom}(G,\F_2D)$, and the non-trivial homomorphisms from $G$ into $\F_2D$ correspond to the quadratic subextensions of $F/K$. Let $\phi$ be such a homomorphism, corresponding to a quadratic subextension $E/K$. Then $\phi$ maps to $0$ in $H^1(G,M)$ if and only if there is $\eta\in M$ with $\sigma(\eta)+\eta=\phi(\sigma)D$ for each $\sigma\in G$. Now an element $\eta \in \mathbb{F}_2^{\W}$ satisfying this equation corresponds to a factor of $f_0(x)$ over $E$, $h(x)$ say, for which $f_0(x)=h(x)\cdot \tau h(x)$, where $\tau$ denotes the generator of $\textup{Gal}(E/K)$. Since $\frac{1}{2}|\W|=g+1$,  such an $\eta$ is in the sum-zero part of $\F_2^\W$ if and only if $g$ is odd. We conclude from this that the number of non-identity elements in $\ker\left(H^1(G,\mathbb{F}_2D)\rightarrow H^1(G,M)\right)$ is equal to $0$ if $g$ is even, and $m$ if $g$ is odd. This gives the result.
\end{proof}

Now let $\Delta_{f}$ be the discriminant of $f(x)$.  It is a square in $K$ if and only if the Galois group of $f(x)$ is a subgroup of the alternating group $A_{n}$ where $n=\deg f$. As a corollary of \Cref{two torsion}  we observe that, if $K(J[2])/K$ is cyclic, then whether or not the discriminant of $f(x)$ is a square in $K$ can essentially be detected from the rational $2$-torsion in $J$. In the statement, we continue to impose \Cref{convention_hyp_factor}. 

\begin{cor}\label{two torsion cor}

Suppose $K\left(J[2]\right)/K$ is cyclic. Then $\Delta_{f}$ is a square in $K$ if and only if one of the following holds:

\begin{enumerate}
	\item[\textit{(i)}] $(-1)^{\dim{J(K)[2]}}=1$ and either $g$ is odd or $f(x)$ has an odd degree factor over $K$,
	\item[\textit{(ii)}]  $(-1)^{\dim{J(K)[2]}}=-1$, $g$ is even, and all factors of $f(x)$ over $K$ have even degree. 
\end{enumerate}
\end{cor}

\begin{proof} 
Let $\sigma$ be a generator of $\text{Gal}\left(K(J[2])/K\right)$.  Then $\Delta_{f}$ is a square in $K$ if and only if $\epsilon(\sigma)=1$, where $\epsilon(\sigma)$ is the sign of $\sigma$ as a permutation on the roots of $f(x)$. Suppose $\sigma$ has cycle type $(d_{1},...,d_{n})$, so that the $d_{i}$ are   the degrees of the irreducible factors of $f(x)$ over $K$. Then we have $\epsilon(\sigma)=(-1)^{\sum_{i=1}^{n}(d_{i}-1)}=(-1)^{\text{deg}f-n}$. Now $J[2]^{\sigma}=J(K)[2]$. Moreover, $K(J[2])/K$ contains at most one quadratic subextension, which yields a factorisation of $f_0(x)$ into $2$ distinct conjugate polynomials if and only if each $d_i$ is even. The result now follows from \Cref{two torsion}.
\end{proof}

\section{Deficiency} \label{deficiency section}

Let $K$ be a local field. Recall from \Cref{Ab_var_notat_intro} that  a (smooth, proper, geometrically connected) curve $X/K$ of genus $g$   is said to be deficient  if $X$ has no $K$-rational divisor of degree $g-1$ . 
In this section we collect some results on deficiency which will be of use later. Firstly, we determine the behaviour of deficiency in field extensions. Next, we give some criteria for determining when a hyperelliptic curve is deficient, which apply in particular when $K(J[2])/K$ is cyclic. Finally, for nonarchimedean base fields we recall a criteria due to Poonen and Stoll which describes deficiency of a general curve in terms of its minimal proper regular model. 

The first $2$ results mentioned above are a consequence of the following description of deficiency, which  arises as part of the proof of \cite[Theorem 11]{MR1740984}. Consider the short exact sequences of $G_K$-modules 
\begin{equation} \label{deficiency_exact_seq_1}
0 \longrightarrow K^s(X)^{\times}/K^{s \times}\stackrel{\text{div}}{\longrightarrow}\text{Div}(X_{K^s})\longrightarrow \text{Pic}(X_{K^s}) \longrightarrow 0
\end{equation} 
and 
\begin{equation} \label{deficiency_exact_seq_2}
0 \longrightarrow K^{s \times}\longrightarrow K^s(X)^{\times}\longrightarrow K^s(X)^{\times}/K^{s \times} \longrightarrow 0.
\end{equation}
Here $K^s(X)$ is the function field of $X$ over the separable closure of $K$, $\textup{Div}(X_{K^s})$ is the group of divisors on the base-change of $X$ to $K^s$, $\textup{Pic}(X_{K^s})$ is the Picard group of $X_{K^s}$, and the map $\textup{div}$ sends a rational function on $X_{K^s}$ to its associated divisor. 

As explained in the proof of \cite[Theorem 11]{MR1740984}, combining the associated long exact sequences for Galois cohomology we obtain an exact sequence
\begin{equation*}
0\longrightarrow \text{Pic}(X) \longrightarrow \text{Pic}(X_{K^s})^{G_K} \longrightarrow \text{Br}(K),
\end{equation*} 
where $\textup{Br}(K)=H^2(K,K^{s \times})$ denotes the Brauer group of $K$.
\begin{notation}
We denote by $\phi_K$ the composition  
\[\phi_K:\text{Pic}(X_{K^s})^{G_K} \longrightarrow \text{Br}(K)\stackrel{\textup{inv}_K}{\longrightarrow}\mathbb{Q}/\mathbb{Z}, \]
where the first map is the one constructed above, and the second is the local invariant map.
\end{notation}
By a result of Lichtenbaum  \cite{MR242831} (see also \cite[Section 4]{MR1740984}) $X$  has a $K$-rational divisor class   of degree $g-1$. Fix such a class $\mathcal{L} \in \text{Pic}^{g-1}(X_{K^s})^{G_K}$. In the proof of \cite[Theorem 11]{MR1740984}, Poonen--Stoll show that  $(g-1)\phi_K (\mathcal{L})\in \{0,1/2\}$, and that  $X$ is deficient over $K$  if and only if $(g-1)\phi_K(\mathcal{L})=1/2$.

\subsection{Deficiency in field extensions}
 
Recall that $\epsilon(X/K)\in \{0,1\}$ is defined to be equal to $1$ if $X$ is deficient over $K$, and equal to $0$ otherwise.

\begin{lemma} \label{deficiency in extensions}
For any finite extension $L/K$ we have
\[ \epsilon(X/L)\equiv [L:K]\epsilon(X/K) ~~\textup{(mod 2)}.\]
\end{lemma}

\begin{proof}
Fix  a rational divisor class $\mathcal{L}$ of degree $g-1$ in $\text{Pic}(X_{K^s})^{G_K}$, so that  $(g-1)\phi_K(\mathcal{L}) \in \mathbb{Q}/\mathbb{Z}$ is equal to $\frac{1}{2}$ (resp.  $0$) if $X$ is deficient over $K$ (resp. is not deficient over $K$). Then $\mathcal{L}$ also gives a rational divisor class of degree $g-1$ in $\text{Pic}(X_{K^s})^{G_L}$, and commutativity of the diagram   
\begin{equation} \label{brauer diagram}
\xymatrix{\text{Br}(K)\ar[d]\ar[r]^{\text{inv}_{K}}\ar[d]^{\text{res}} & \mathbb{Q}/\mathbb{Z}\ar[d]^{[L:K]}\\
\text{Br}(L)\ar[r]^{\text{inv}_{L}} & \mathbb{Q}/\mathbb{Z}
}
\end{equation}
(see e.g. \cite[Proposition XIII.3.7]{MR554237}) shows that $(g-1)\phi_L(\mathcal{L})=[L:K](g-1)\phi_K(\mathcal{L})$.  
\end{proof}

\subsection{Deficiency for hyperelliptic curves}
Now suppose that the characteristic of $K$ is different from $2$. Take $C/K$ to be a hyperelliptic curve of genus $g$ and fix a Weierstrass equation $y^2=f(x)$ for $C$. Since  $C$ has $K$-rational divisors of degree 2 (arising as the pull-back of rational points on $\mathbb{P}^1_K$),  if  $g$ is odd then $C$ is not deficient. Consequently, we impose the following assumption.

\begin{assumption}
For the rest of this subsection, suppose that $g$ is even.
\end{assumption}  

Again using that $C$ has $K$-rational divisors of degree 2, we see under this assumption that having a $K$-rational divisor of degree $g-1$ is equivalent to having a $K$-rational divisor of any odd degree, which is in turn equivalent to having a rational point over some odd degree extension of $K$. In particular, if $f(x)$ has an odd degree factor over $K$ then $C$ is not deficient. 

Now write $f(x)=c_ff_0(x)$, where $c_f$ is the leading coefficient of $f(x)$ and $f_0(x)$ is monic. The following proposition gives a convenient criterion for testing deficiency in the special case that $f_0(x)$ factors as a product of $2$ conjugate polynomials over some quadratic extension of $K$.

\begin{proposition}\label{deficiency lemma} 
Suppose  $f_0(x)$ factors over a quadratic extension $F/K$ as a product of $2$ polynomials conjugate under the action of $\textup{Gal}(F/K)$. Then $C$ is deficient if and only if $(c_f,F/K)=-1$. 
\end{proposition}

\begin{proof}
Over $F$, write  $f_0(x)= f_{0a}(x)f_{0b}(x)$ where $f_{0a}(x)$ and $f_{0b}(x)$ are monic and conjugate under the action of $\textup{Gal}(F/K)$. As $g$ is assumed even, both $f_{0a}(x)$ and $f_{0b}(x)$ necessarily have odd degree $g+1$. For each root $r$ of $f(x)$, write $P_r=(r,0)\in C(K^s)$. For $\star\in \{a,b\}$, consider the degree $g+1$ divisor  
\[D_\star=\sum_{r\textup{ root of }f_{0\star}(x)}P_r\quad \quad \in \textup{Div}(C_{K^s}).\]
Denote by $\chi :G_K\rightarrow \{\pm 1\}$ the quadratic character corresponding to $F/K$. Then for all $\sigma \in G_K$ we have
\[\sigma(D_a)=\begin{cases} D_a~~&~~\chi(\sigma)=1,\\D_b~~&~~\chi(\sigma)=-1.\end{cases}\]
Since $\textup{div}(y/f_{0b}(x))=D_a-D_b$ we see that the class of $D_a$ in $\textup{Pic}(C_{K^s})$, which we denote $[D_a]$, is invariant under $G_K$. Further, we see that  under the connecting map $\textup{Pic}(C_{K^s})^{G_K}\rightarrow H^1(K,K^s(C)^{\times}/K^{s \times})$ associated to \eqref{deficiency_exact_seq_1}, the class $[D_a]$ maps to the class of the $1$-cocycle $\rho$ defined by 
 \[\rho(\sigma)=\begin{cases}1~~&~~\chi(\sigma)=1,\\ y/f_{0b}(x)~~&~~\chi(\sigma)=-1. \end{cases}\]
 This lifts via the same formula to a $1$-cochain valued in $K^s(C)^{\times}$. The image of  $\rho$ under the connecting map $H^1(K,K^s(C)^{\times}/K^{s \times})\rightarrow H^2(K,K^{s \times})=\textup{Br}(K)$   associated to $\eqref{deficiency_exact_seq_2}$ is thus represented by the class of the $2$-cocycle $\alpha$ defined by  $\alpha(\sigma,\tau)=\rho(\sigma) \cdot {}^\sigma \rho(\tau) \cdot \rho(\sigma \tau)^{-1}$. A straightforward computation shows that $\alpha(\sigma,\tau)=1$ unless $\chi(\sigma)=-1=\chi(\tau)$, in which case it is equal to 
\begin{equation*}
\frac{y}{f_{0b}(x)}\cdot \frac{y}{f_{0a}(x)}=\frac{y^2}{f_0(x)}=c_f.
\end{equation*}
Under $\text{inv}_K:\text{Br}(K)\rightarrow\ \mathbb{Q}/\mathbb{Z}$, the class of this $2$-cocycle is mapped to  0 if $c_f$ is a norm from $F^{\times}$, and to $1/2$ otherwise. Thus  $(g+1)\phi([D_a])=1/2$ if and only if  $(c_f,F/K)=-1$.
\end{proof}

\begin{remark}
As in \Cref{2-tors sect}, let $\mathcal{W}$ denote the $G_K$-set of ramification points of the $x$-coordinate morphism $C\rightarrow \mathbb{P}^1_K$. Further, let $\mathcal{T}$ denote the quotient of the sum-$1$ part of the permutation module $\mathbb{F}_2^\mathcal{W}$ by the diagonal action of $\mathbb{F}_2$. We see from \eqref{2_tors_g_mod_desc} that $\mathcal{T}$ is a torsor under $J[2]$. In fact, $\mathcal{T}$ can be identified as a $G_K$-set with the collection of theta characteristic on $C_{K^s}$ (see \cite[Section 4]{Mum71} and recall that $g$ is assumed even). From this we see that $C$ has a $K$-rational theta characteristic if and only if either $f(x)$ has an odd degree factor over $K$, or $f_0(x)$  factors as a product of $2$  conjugate polynomials over some quadratic extension of $K$.  Thus the results above concerning deficiency apply precisely when $C$ has a $K$-rational theta characteristic.
\end{remark}

\begin{cor}\label{cyclic deficiency lemma}
Suppose that $K(J[2])/K$ is cyclic. If $C$ is deficient over $K$ then ($g$ is even and) every irreducible factor of $f(x)$ over $K$ has even degree. When this is the case, denote by   $F/K$ the unique quadratic subextension of $K(J[2])/K$. Then $C$ is deficient if and only if $(c_f,F/K)=-1$.
\end{cor}

\begin{proof}
As noted above we may assume that each irreducible factor of $f(x)$ over $K$ has even degree, in which case $f(x)$ has degree $2g+2$. The assumption that $K(J[2])/K$ is cyclic now ensures that there is indeed a unique quadratic subextension of $K(J[2])/K$, $F/K$ say, and that $f_0(x)$ factors into $2$ conjugate  polynomials over $F$. The  result now follows from \Cref{deficiency lemma}.  
\end{proof}

\begin{remark} \label{using_cyclic_def_rmk}
Suppose we are in the situation of \Cref{cyclic deficiency lemma}, so that ($g$ is even and) $K(J[2])/K$ is cyclic. Let $L/K$ be a quadratic extension, say $L=K(\sqrt{d})$ for some $d\in K^\times$. Then the quadratic twist of $C$ by $L/K$ has Weierstrass equation $C^L:y^2=df(x)$. As above, we have $\epsilon(C/K)=0=\epsilon(C^L/K)$ unless every irreducible factor of $f(x)$ over $K$ has even degree. In this latter case, with $F/K$ as in the statement of \Cref{cyclic deficiency lemma}, we see from that result that 
\[(-1)^{\epsilon(C/K)+\epsilon(C^L/K)}=(c_f,F/K)(dc_f,F/K)=(d,F/K).\]
\end{remark}
 
\subsection{Deficiency in terms of the minimal proper regular model}
Now suppose that $K$ is a  nonarchimedean local field, possibly of characteristic $2$,  and that $X/K$ is a  (not necessarily hyperelliptic) curve  of genus $g$. To conclude the section we recall a characterisation of deficiency in terms of the minimal proper regular model of $X$. We will make extensive use of this criterion later. 

Let $\mathcal{X}/\mathcal{O}_{K}$ denote the minimal proper regular model of $X$, and let $\mathcal{X}_{\bar{k}}$ denote the base-change to $\bar{k}$ of its special fibre. Let $\{\Gamma_i\}_{i\in I}$ denote the set of irreducible components of $\mathcal{X}_{\bar{k}}$. For each $i\in I$, let $d_i$ denote the multiplicity of $\Gamma_i$ in $\mathcal{X}_{\bar{k}}$, and let $\textup{Orb}_{G_k}(\Gamma_i)$ denote the $G_k$-orbit of $\Gamma_i$.  


\begin{lemma} \label{min reg def remark} \label{min reg def lemma}
The curve $X$ is deficient over $K$ if and only if 
\[\textup{gcd}_{i\in I}\{d_i\cdot |\textup{Orb}_{G_k}(\Gamma_i)|\}\]
 does not divide $g-1$.
\end{lemma}

\begin{proof}  
This is observed by  Poonen--Stoll; see \cite{MR1740984}, remark following the proof of Lemma 16.
\end{proof}

\begin{remark}
When $X/K$ is a hyperelliptic curve we have $\textup{gcd}_{i\in I}\{d_i\cdot |\textup{Orb}_{G_k}(\Gamma_i)|\}\in \{1,2\}$. This follows from \cite[Corollary 1.5]{MR1717533} and the fact that all hyperelliptic curves have closed points of degree dividing $2$.
\end{remark}


\section{First cases of Conjecture 1.7} \label{first cases}

In this section we prove \Cref{Kramer Tunnell} in two cases: when $K$ is archimedean, and when $K$ has odd residue characteristic and $J/K$ has good reduction. It will turn out that these are the only cases needed to prove \Cref{global to local_intro} (in fact, even the case of archimedean $K$ is not necessary for this). 

\subsection{Archimedean local fields}   

Here we consider \Cref{Kramer Tunnell}  for archimedean local fields. Clearly the only case of interest is the extension $\mathbb{C}/\mathbb{R}$.   Let $C/\mathbb{R}$ be a hyperelliptic curve of genus $g$ and let $J/\mathbb{R}$ be its Jacobian. 

\begin{proposition} \label{kramer-tunnell for reals places}
\Cref{Kramer Tunnell} holds for  $C/\mathbb{R}$ and the extension $\mathbb{C}/\mathbb{R}$.
\end{proposition}

\begin{proof}
We have $w(J/\mathbb{C})=(-1)^{g}$ (see, for example, \cite[Lemma 2.1]{MR2309184}). Further, by \cite[Lemma 10.9 (ii)]{MR3951582} we have $|J(\mathbb{R})/N_{\mathbb{C}/\mathbb{R}} J(\mathbb{C}) |=2^{-g}|J(\mathbb{R})[2]|$.  Denote by $J_{-1}$ the quadratic twist of $J$ by $\mathbb{C}/\mathbb{R}$, and and denote by $C_{-1}$ the quadratic twist of $C$ by $\mathbb{C}/\mathbb{R}$ similarly.  To verify \Cref{Kramer Tunnell} we must show that 
\begin{equation*}
(-1)^{\dim J(\mathbb{R})[2]}=(\Delta_C,\mathbb{C}/\mathbb{R})(-1)^{\epsilon(C/K)+\epsilon(C_{-1}/K)}.
\end{equation*}

Now  $\mathbb{R}(J[2])/\mathbb{R}$ is cyclic, and  $(\Delta_C,\mathbb{C}/\mathbb{R})=1$ if and only if $\Delta_C$ is a square in $\mathbb{R}$. Consequently, \Cref{two torsion cor} gives $(-1)^{\dim J(\mathbb{R})[2]}=(\Delta_C,\mathbb{C}/\mathbb{R})$, except when $g$ is even and all irreducible factors of $f(x)$ over $\mathbb{R}$ have even degree. In this latter case the two expressions differ by a sign. Since by \Cref{cyclic deficiency lemma} (cf. also \Cref{using_cyclic_def_rmk}) this is exactly the case where $\epsilon(C/K)+\epsilon(C_{-1}/K)=1$, we have the result.      
\end{proof}

\subsection{Good reduction in odd residue characteristic} \label{odd good}

Suppose now that $L/K$ is a quadratic extension of nonarchimedean local fields of odd residue characteristic. Let $C/K$ be a hyperelliptic curve and $J/K$ its Jacobian.  We denote by $v$  the normalised valuation on $K$.

\begin{proposition} \label{kramer-tunnell for good reduction in odd residue char}
Suppose that $J$ has good reduction over $K$. Then \Cref{Kramer Tunnell} holds for $C$ and the extension $L/K$. 
\end{proposition}

\begin{proof}
Since $J$ has good reduction over $K$ we have $w(J/L)=1$, so we are reduced to showing that
\begin{equation*}
(-1)^{\dim J(K)/\N J(L)}=(\Delta_C,L/K)(-1)^{\epsilon(C/K)+\epsilon(C^L/K)}.
\end{equation*}
Suppose first that $L/K$ is unramified. Then \cite[Lemma 10.9 (i)]{MR3951582}  gives $(-1)^{\dim J(K)/\N J(L)}=1$ (this goes back to a result of Mazur \cite[Corollary 4.4]{MR0444670}). Moreover, the assumption  on the reduction of $J$ implies that $K(J[2])/K$ is unramified. Thus adjoining a square root of $\Delta_C$ to $K$ produces an unramified extension. In particular, $v(\Delta_C)$ is even and  $(\Delta_C,L/K)=1$. Finally, \Cref{cyclic deficiency lemma} gives $(-1)^{\epsilon(C/K)+\epsilon(C^L/K)}=1$ also.

Now suppose $L/K$ is ramified. This time, \cite[Lemma 10.9 (i)]{MR3951582}  gives \[(-1)^{\dim J(K)/\N J(L)}=(-1)^{\dim J(K)[2]}.\] Moreover, as $v(\Delta_C)$ is even we have $(\Delta_C,L/K)=1$ if and only if $\Delta_C$ is a square in $K$. We now conclude by \Cref{two torsion cor} and \Cref{cyclic deficiency lemma}.
\end{proof}


\section{Global conjectures imply instances of Conjecture 1.7} \label{global to local section}

We have already proven enough cases of \Cref{Kramer Tunnell} to prove \Cref{global to local_intro}.

\begin{theorem}[=\Cref{global to local_intro}] \label{global to local}
Let $K$ be a number field, $C/K$ a hyperelliptic curve, $J/K$ its Jacobian and $v_0$ a place of $K$. If the $2$-parity conjecture holds for $J$ over every quadratic extension $F/K$, then \Cref{Kramer Tunnell} holds for $J/K_{v_0}$ and every quadratic extension $L/K_{v_0}$. 
\end{theorem}      

\begin{proof}
Let $L/K_{v_0}$ be a quadratic extension, and write $L=K_{v_0}(\sqrt{\alpha})$. Let $S$ be a finite set of places of $K$ containing all places where $J$ has bad reduction, all places dividing 2, and all archimedean places. Set $T=S-\{v_0\}$. 

Now let $F/K$ be a quadratic extension such that:
\begin{itemize}
\item each place $v\in T$ splits in $F/K$, 
\item there is exactly one place $\mathfrak{v}_0$ of $F$ extending $v_0$,
\item  we have $F_{\mathfrak{v}_0}=L$.
\end{itemize}
 Explicitly, we may take $F=K(\sqrt{\beta})$ where $\beta \in K$ is chosen, by weak approximation, to be sufficiently close to $\alpha$ $v_0$-adically, and sufficiently close to $1$ $v$-adically for all $v\in T$. 
 
With such an extension $F/K$ chosen, for a place $v$ of $K$ which is non-split in $F/K$, denote by $\mathfrak{v}$ the unique place of $F$ extending $v$. Then (cf. \Cref{selmer decomposition}) the products
\[ 
\prod_{\substack{v\textup{ place of }K\\\textup{non-split in }F/K}} \hspace{-12pt}w(J/F_\mathfrak{v}) \quad\text{ and }~~ \quad \prod_{\substack{v\textup{ place of }K\\\textup{non-split in }F/K}} \hspace{-12pt} (\Delta_C,F_\mathfrak{v}/K_v)(-1)^{\dim J(K_v)/N_{F_\mathfrak{v}/K_v}J(F_\mathfrak{v}) +\epsilon(C/K_v)+\epsilon(C^F/K_v)}
\]
are equal to $w(J/F)$ and $(-1)^{\text{rk}_2(J/F)}$ respectively, hence agree under the assumption that the $2$-parity conjecture holds for $J$ over $F$. 

 On the other hand, by \Cref{kramer-tunnell for good reduction in odd residue char} and our assumptions on $F/K$, the individual contributions to these products at a  place $v$ agree, save possibly at $v=v_0$. Thus the contributions at $v=v_0$ must agree also.
\end{proof}

\section{Unramified extensions} \label{main unramified section}

Let $K$ be a nonarchimedean local field of characteristic different from $2$. In this section we begin the study of \Cref{Kramer Tunnell} for  unramified quadratic extensions. Thus we fix a hyperelliptic curve $C/K$, and denote by $L/K$ the unique  unramified quadratic extension of $K$. As usual, we denote by $J/K$ the Jacobian of $C$. Across Sections  \ref{main unramified section} and \ref{proof of compatibility} we will prove the following:

\begin{proposition} \label{prop:unram_cases_of_conjecture}
\Cref{Kramer Tunnell}  for $C/K$ and the extension $L/K$ holds in each of the following cases:
\begin{itemize}
\item[(i)] the residue characteristic of $K$ is odd,
\item[(ii)] the residue characteristic of $K$ is $2$ and  $J/K$  has good reduction,
\item[(iii)] the residue characteristic of $K$ is $2$, $C$ has genus $2$, and $J/K$ has semistable reduction.
\end{itemize}
\end{proposition}

To prove this, the key fact we will exploit is that the formation of N\'{e}ron models and minimal regular models commutes with unramified base-change. As $L/K$ is assumed unramified this makes the  quantities appearing in  \Cref{Kramer Tunnell} comparatively easy to describe and relate to one another (in particular, it allows us to readily relate invariants of $C$ to invariants of the quadratic twist of $C$ by $L/K$). This enables us to reduce \Cref{Kramer Tunnell} to a statement which  depends only on the curve $C$ considered over the maximal unramified extension of $K$; see \Cref{unramifed reduction}. We then prove this statement under the conditions of  \Cref{prop:unram_cases_of_conjecture}.

 Denote by $k$ the residue field of $K$, and denote by $k_L$ the residue field of $L$.   Further, denote by $\mathfrak{f}(J/K)$   the conductor of $J$, and denote by $\Phi$  the component group of the special fibre of the N\'{e}ron model of $J$.

\begin{lemma} \label{unram root numbers and norm}
We have
\[w(J/L)=(-1)^{\mathfrak{f}(J/K)}\]
and 
\[\dim J(K)/\N J(L)= \dim H^1(k_L/k,\Phi(k_L)).\] 
\end{lemma}

\begin{proof}
For the statement about root numbers see \cite[Corollary 4.6]{MR3860395}. The statement concerning the norm map follows from \Cref{basic properties of norm map} and \cite[Proposition 4.3]{MR0444670}. 
\end{proof}  

\Cref{unram root numbers and norm} describes two of the terms appearing in \Cref{Kramer Tunnell}. Moreover, as $L/K$ is unramified  we have
\begin{equation}
\left(\Delta_C,L/K\right)=(-1)^{v(\Delta_C)},
\end{equation}
where $v$ denotes the normalised valuation on $K$.
We thus see that  \Cref{Kramer Tunnell} for $C$ and $L/K$ is equivalent to the assertion  
\begin{equation} \label{unramified kramer tunnell_1}
\mathfrak{f}(J/K)\stackrel{?}{\equiv} v(\Delta_C)+\dim  H^1(k_L/k,\Phi(k_L))+\epsilon(C/K)+\epsilon(C^L/K)~~~\text{ (mod 2)}.
\end{equation}
Since $\mathfrak{f}(J/K)$ and $v$ are unchanged under unramified extension, this predicts that the quantity 
\begin{equation} \label{unramified norm quantity}
\dim  H^1(k_L/k,\Phi(k_L))+\epsilon(C/K)+\epsilon(C^L/K)~~~\text{ (mod 2)}
\end{equation}
is unchanged   upon replacing $K$ by a finite unramified extension $F$, and replacing $L$ by the unique quadratic unramified extension $F'/F$. In fact, we can use this observation to predict a simpler expression for \eqref{unramified norm quantity}. We begin with some notation. 

\begin{notation} \label{notat:epsilon_C}
Denote by $\mathcal{C}$ the minimal regular model of $C$ over $\mathcal{O}_K$. For each irreducible component $\Gamma$ of $\mathcal{C}_{\bar{k}}$, write $d(\Gamma)$ for its multiplicity. Further, denote by $\iota$ the automorphism of $\mathcal{C}_{\bar{k}}$ induced from the hyperelliptic involution on $C$ (which extends to an automorphim of $\mathcal{C}$ by uniqueness of the minimal regular model). We then define
\[\eta(C)=\begin{cases}1~~&~~g\textup{ even and }  |\textup{orb}_\iota(\Gamma)|\cdot d(\Gamma)~\equiv 0~~(\textup{mod }2) \textup{ for each irred. cmpt. }\Gamma\textup{ of }\mathcal{C}_{\bar{k}},\\0~~&~~\textup{otherwise,} \end{cases}\]
where here $\textup{orb}_\iota(\Gamma)$ denotes the orbit of an irreducible component $\Gamma$ under the action of $\iota$.
\end{notation}
Now suppose that $F/K$ is a sufficiently large even-degree unramified extension so that  $G_{k_F}$ acts trivially on both $\Phi(\bar{k})$ and on the set of irreducible components of $\mathcal{C}_{\bar{k}}$. Let $F'/F$ be the unique quadratic unramified extension. Then we have 
\begin{equation} \label{invariants_after_large_extension}
H^1(k_{F'}/k_F,\Phi(k_{F'}))\cong\Phi(\bar{k})[2],\quad  \quad \epsilon(C/F)=0 \quad \textup{ and } \quad \epsilon(C^{F'}/F)=\eta(C).
\end{equation}
 The first equality follows from our assumptions on the $G_{k_F}$-action on $\Phi(\bar{k})$, along with the description of the cohomology of cyclic groups given in \cite[Section 8]{MR0219512}. The second equality follows from \Cref{deficiency in extensions} since $F/K$ is assumed to have even degree. The third equality follows from   \Cref{min reg def remark}, our assumptions on the $G_{k_F}$-action on the irreducible components of $\mathcal{C}_{\bar{k}}$, and the fact that the formation of the minimal regular model commutes with unramified base-change. Indeed, this last fact allows us to identify the geometric special fibre of the minimal regular model of $C^L/K$ with that of $C/K$, save with $G_k$-action twisted by $\iota$.
 
 From \eqref{invariants_after_large_extension} we find
\[\dim H^1(k_{F'}/k_F,\Phi(k_{F'}))+\epsilon(C/F)+\epsilon(C^{F'}/F)=\Phi(\bar{k})[2]+ \eta(C).\]
 Consequently, the discussion preceding \Cref{notat:epsilon_C} predicts the following identity, which we will give an unconditional proof of.
  
\begin{proposition} \label{unchangedness of norm}
With the notation above, we have
\[\dim  H^1(L/K,\Phi(k_L))+\epsilon(C/K)+\epsilon(C^L/K) \equiv \dim \Phi(\bar{k})[2]+ \eta(C)~~\textup{ (mod 2)}.\] 
\end{proposition}

The proof of \Cref{unchangedness of norm} that we will give is somewhat lengthy and we postpone it to the next section.   

An immediate corollary of \Cref{unchangedness of norm} is the following:

\begin{cor} \label{unramifed reduction}
 \Cref{Kramer Tunnell} holds for $C/K$ and the extension $L/K$  if and only if
\begin{equation} \label{unramified claim}
\mathfrak{f}(J/K) \equiv v(\Delta_C)  +\dim \Phi(\bar{k})[2]+ \eta(C)~~\textup{ (mod 2)}.
\end{equation}
\end{cor} 

\begin{proof}
Combine (the discussion surrounding) \eqref{unramified kramer tunnell_1} with \Cref{unchangedness of norm}.
\end{proof}

\begin{remark} \label{is_the_congruence_easier}
In the statement of \Cref{unchangedness of norm} it is not simply true that \[\dim H^1(L/K,\Phi(k_L)) \equiv \dim\Phi(\bar{k})[2] ~~\textup{ (mod 2)}\]
and \[\epsilon(C/K)+\epsilon(C^L/K) \equiv \eta(C)~~\textup{ (mod 2)},\]
as the following example shows. 
\end{remark}

\begin{example} \label{unram_genus_2_ub_example}
Consider the genus $2$  hyperelliptic curve 
\[C/\mathbb{Q}_3:y^2=(x^2+3)((x-i)^2-3^2)((x+i)^2-3^2),\]
 where $i$ is a square root of $-1$.  Reducing the defining equation modulo $3$ gives a semistable curve. In fact, the above equation  viewed as a scheme over $\mathbb{Z}_3$, along with the usual chart at infinity, gives the stable model of $C$.  After base-changing to $\mathbb{Z}_3^{\textup{nr}}$ its special fibre consists of $2$ rational curves, intersecting transversally at the $3$ points $(0,0)$, $(i,0)$ and $(-i,0)$. The final two  intersection points are swapped by the Frobenius element $F\in G_k$, whilst the hyperelliptic involution $\iota: (x,y)\mapsto (x,-y)$ fixes each intersection point but  swaps the  $2$ components. The minimal proper regular model   of $C$ over $\mathbb{Z}_3^{\textup{nr}}$ is  obtained by blowing up once each at the intersection points  $(i,0)$ and $(-i,0)$. Its special fibre, along with the actions of $F$ and $\iota$, is thus as pictured:
 \vspace{-5pt}
 \begin{center}
\begin{figure}[!htbp]
\begin{tikzpicture}[scale=0.6]
\draw [line width=1pt] (1.1,-3.2)-- node[above]{} node[below, font=\small]{} ++(5,1);  
\draw [line width=1pt] (1.1,-2.9)-- node[above]{} node[below, font=\small]{} ++(5,-1);  
\draw  [line width=1pt] (4,-2.1)-- ++(0,-2); 
\draw  [line width=1pt] (5.5,-1.9)--   ++(0,-2.4); 
\draw [to-to](4,-1.5)-- node[above]{$F$} (5.5,-1.5);
\draw [to-to](6.6,-2)--  (6.6,-4);
\draw (7,-3) node {$\iota$};
\end{tikzpicture}
\end{figure}
\end{center}
Here each component pictured is a rational curve of multiplicity $1$.
Write $L= \mathbb{Q}_3(i)$ for the unique  unramified quadratic extension of $\mathbb{Q}_3$.
 Since $F$ fixes the $2$ components drawn horizontally we see from \Cref{min reg def lemma} that $\epsilon(C/\mathbb{Q}_3)=0$. Similarly, we have $\eta(C)=0$ since $\iota$ fixes the $2$ components drawn vertically. However, each $\iota \circ F$-orbit of components has size $2$, so appealing to \Cref{min reg def lemma} once more we find $\epsilon(C^L/\mathbb{Q}_3)=1$.   
 Thus  \[\epsilon(C/K)+\epsilon(C^L/K) \not \equiv \eta(C)~~\textup{ (mod 2)}.\]
 
 However, one can also show (e.g. using the description of $\Phi(\bar{k})$ detailed later in   \Cref{sec:component_group_abstract_description}) that $\Phi(\bar{k})\cong \mathbb{Z}/8\mathbb{Z}$ with   $F$  acting as multiplication by $5$. Thus we have 
 \[\dim \Phi(\bar{k})[2] =1\quad \textup{ and }\quad \dim H^1(L/K,\Phi(k_L))=0.\]  In particular, \Cref{unchangedness of norm} holds in this example, even though neither of the individual congruences in \Cref{is_the_congruence_easier} hold.
\end{example}

\subsection{Establishing \Cref{unramified claim} in odd residue characteristic}

Assume  that the residue characteristic of $K$ is odd. Under this assumption we now establish the congruence \Cref{unramified claim}.

\begin{lemma} \label{KES lemma}
We have 
\[\mathfrak{f}(J/K)=\mathfrak{f}(J[2])+\dim\Phi(\bar{k})[2],\]
where here $\mathfrak{f}(J[2])$ denotes the Artin conductor of the $G_K$-module $J[2]$. 
\end{lemma}

\begin{proof}
This is observed by \v{C}esnavi\v{c}ius in \cite[Lemma 4.2]{MR3860395} (we remark that the cited result uses the assumption that $K$ has odd residue characteristic).
\end{proof}

In light of \Cref{KES lemma}, to establish \Cref{unramified claim} it remains to show that 
\begin{equation}  \label{yet_another)unram_simplificaiton}
\mathfrak{f}(J[2]) \equiv v(\Delta_C)+\eta(C) ~~ \text{ (mod 2)}.
\end{equation}

Fix a Weierstrass equation $C/K:y^2=f(x)$ for $C$, where $f(x)\in K[x]$ is a squarefree polynomial of (without loss of generality) even degree $2g+2$. Let $E/K^{\text{nr}}$ be the field extension $E=K^{\text{nr}}(J[2])$, and set $G=\text{Gal}(E/K^{\textup{nr}})$. As explained in \Cref{2-tors sect}, $E$ coincides with the splitting field of $f(x)$ over $K^{\text{nr}}$ (recall that we are assuming $g\geq 2$). Let $G=G_0 \vartriangleright G_1 \vartriangleright G_2 \vartriangleright ...$ be the ramification filtration on $G$, and write $g_i=|G_i|$. Thus $G_1$ is the wild inertia group of $E/K^{\text{nr}}$ and is a $p$-group where $p=\text{char}(k)$ (so in particular has odd order), and $G/G_1$ is cyclic. Let $\mathcal{R}$ denote the $G$-set of roots of $f(x)$ in $E$. By definition we have
 \[\mathfrak{f}(J[2])=\sum_{i=0}^{\infty} \frac{g_i}{g_0} \textup{codim}J[2]^{G_i}.\]

\begin{notation}
With $f(x)\in K[x]$ as above, define $\eta(f)\in \{0,1\}$ to be equal to $1$ if the genus $g$ of $C$ is even and each irreducible factor of $f(x)$ over $K^{\text{nr}}$ has even degree. Otherwise, set $\eta(f)=0$.
\end{notation}

\begin{lemma} \label{comparison of conductors}
Let $V=\mathbb{C}[\mathcal{R}]$ be the complex permutation representation of $G$ associated to the  set of roots  of $f(x)$. Then we have
\[\mathfrak{f}(J[2])\equiv \mathfrak{f}(V)+\eta(f)~~(\textup{mod }2).\]
\end{lemma}

\begin{proof}
This will follow from the definitions of $\mathfrak{f}(J[2])$ and $\mathfrak{f}(V)$, along with a comparison between $\text{codim}_\mathbb{C}V^{G_i}$ and $\text{codim}_{\mathbb{F}_2}J[2]^{G_i}$ for each $i$ (afforded by \Cref{two torsion}).

First let $i\geq1$ so that $G_i$ has odd order. Then necessarily $f(x)$ has an odd degree factor over $E^{G_i}$, so it follows from \Cref{two torsion} that $\dim_{\mathbb{F}_2}J[2]^{G_i}=\dim_{\mathbb{C}}V^{G_i}-2$. Since also $\dim_{\mathbb{F}_2}J[2]=\dim_{\mathbb{C}}V-2$  we see that
\[\sum_{i=1}^{\infty}\frac{g_i}{g_0}\text{codim}_{\F_2}J[2]^{G_i}=\sum_{i=1}^{\infty}\frac{g_i}{g_0}\text{codim}_{\mathbb{C}}V^{G_i}.\] 
It remains  to show that
\[\text{codim}_{\F_2}J[2]^G\equiv\text{codim}_{\mathbb{C}}V^G+\eta(f)~~~\textup{(mod  ~2)}.\]

When $g$ is even, \Cref{two torsion} gives $\dim_{\F_2}J[2]^G=\dim_{\mathbb{C}}V^G-2+\eta(f)$ and we are done. So suppose that $g$ is odd. If $f(x)$ has an odd degree factor over $K^{\text{nr}}$ then again we conclude immediately from \Cref{two torsion}. Finally, suppose each irreducible factor of $f(x)$ over $K^{\text{nr}}$ has even degree. Write $f(x)=c_ff_0(x)$ where $c_f$ is the leading coefficient of $f(x)$ and $f_0(x)$ is monic. Applying \Cref{two torsion} once again it suffices to show that there is a unique quadratic subextension of $E/K^\text{nr}$ over which $f_0(x)$ factors into $2$ distinct, conjugate polynomials. To see this, first note that there is a unique quadratic subextension of $E/K^\text{nr}$. Indeed, any such extension must necessarily be contained in $E^{G_1}$, and $E^{G_1}/K^{\text{nr}}$ is cyclic and (by the assumption on the degrees of the irreducible factors of $f(x)$ over $K^{\textup{nr}}$) has even degree. To see that $f_0(x)$ admits the required factorisation over this extension, let $S=\{h_1,...,h_l\}$ be the set of irreducible factors of $f_0(x)$ over $E^{G_1}$, each of which necessarily has odd degree. The cyclic group $G/G_1$ acts on $S$ and since each factor of $f_0(x)$ over $K^{\text{nr}}$ has even degree, each orbit of $G/G_1$ on $S$ has even order. Denote these disjoint orbits by $S_1,...,S_t$, and for $1\leq i \leq t$ write $S_i=\{h_{i,1}(x),...,h_{i,d_i}(x)\}$. Fix a generator $\sigma$ of $G/G_1$ and assume without loss of generality that $\sigma h_{i,j} (x)=h_{i,j+1(\text{mod}~d_i)}(x)$.Then the polynomial
\[h(x)=\prod_{i=1}^{t}\prod_{j~\text{odd}}h_{i,j}(x)\]
is fixed by $\sigma^2$, has $\sigma h(x)\neq h(x)$, and is such that $f_0(x)=h(x)\cdot \sigma h(x)$.        
\end{proof} 

Next, we relate the Artin conductor of $V=\mathbb{C}[\mathcal{R}]$ to the discriminant $\Delta_f$ of $f(x)$. 

\begin{lemma} \label{lem:conductor_congruence}
Let $V=\mathbb{C}[\mathcal{R}]$ be as above.  Then 
\[\mathfrak{f}(V) \equiv v(\Delta_f) ~~~\textup{(mod 2)}.\] 
\end{lemma}

\begin{proof}
Write  $f(x)=f_1(x)...f_t(x)$ as a product of irreducible polynomials over $K$, and write $\mathcal{R}_i$ for the set of roots of $f_i(x)$ in $E$.  Then $V$ is a direct sum of the permutation modules $V_i=\mathbb{C}[\mathcal{R}_i]$, so $\mathfrak{f}(V)$ is the sum of the $\mathfrak{f}(V_i)$. Further,  since for polynomials $h_1,h_2$ we have $\Delta_{h_1h_2}=\Delta_{h_1}\Delta_{h_2} \text{Res}(h_1,h_2)^2$  (here $\text{Res}(h_1,h_2)$ denotes the resultant of $h_1,h_2$), we see that the discriminant of $f(x)$ is, up to squares in $K$, the product of the discriminants of the $f_i(x)$. In this way we  reduce to the case where $f(x)$ is irreducible, which we treat now. 

Assuming that $f(x)$ is irreducible, let $F/K$ be the splitting field of $f(x)$ and let $H$ be the stabiliser in $\textup{Gal}(F/K)$ of a root $r\in \mathcal{R}$. Then $V \cong \mathbb{C}[\textup{Gal}(F/K)/H]$ so by the conductor-discriminant formula \cite[VI.2 corollary to Proposition 4]{MR554237} we have $\mathfrak{f}(V)=v(\Delta_{F^{H}/K})$, where $\Delta_{F^{H}/K}$ denotes the discriminant of $F^{H}/K$. Since $F^{H}=K(r)$ we have $v(\Delta_{F^{H}/K}) \equiv v(\Delta_f) ~\text{(mod 2)}$, as desired.  
\end{proof}

Combined, \Cref{comparison of conductors,lem:conductor_congruence} establish the congruence 
\begin{equation} \label{conductor_vs_disc_odd_res}
\mathfrak{f}(J[2])\equiv v(\Delta_f)+\eta(f) ~~\textup{ (mod 2)}.
\end{equation}
To prove \Cref{unramified claim} it remains to reinterpret the `correction' term $\eta(f)$. 

\begin{lemma} \label{geom def}
Let $F/K$ be a finite unramified extension and denote by $F'/F$ the unique quadratic unramified extension of $F$. Then, provided $F/K$ is sufficiently  large,  we have $\eta(f)=\epsilon(C^{F'}/F).$ In particular,  we have $\eta(f)=\eta(C)$.
\end{lemma}

\begin{proof}
Arguing as in the proof of \Cref{comparison of conductors} we see that $f(x)$ either has an odd degree factor over   $K^{\text{nr}}$, or factors as a product of $2$ conjugate polynomials over the unique quadratic extension of  $K^{\text{nr}}$. From this we deduce that for every sufficiently large unramified extension $F/K$, either $f(x)$ has an odd degree factor over $F$, or $f(x)$ factors as a product of $2$ conjugate polynomials over some  totally ramified quadratic extension $L/F$ (the latter happening if and only if each irreducible factor of $f(x)$ over $K^{\text{nr}}$ has even degree). In the latter case, by enlarging $F/K$ further if necessary we may also assume that the leading coefficient of $f(x)$ is a norm from this quadratic extension. Since the quadratic twist $C^{F'}/F$ is given by the equation $C^{F'}:y^2=uf(x)$ where $u$ is a non-square unit in $F$, the claim that $\eta(f)=\epsilon(C^{F'}/F)$ follows from \Cref{deficiency lemma} and the fact that $(u,L/F)=-1$. That $\eta(f)=\eta(C)$ now follows from \eqref{invariants_after_large_extension}.
\end{proof}

\begin{cor} \label{unramified kramer tunnell}
Under the assumption that $K$ has odd residue characteristic, \eqref{unramified claim} holds for $C$.
\end{cor} 

\begin{proof}
 \Cref{geom def} allows us to replace $\eta(f)$ with $\eta(C)$ in \eqref{conductor_vs_disc_odd_res}, hence establishing \eqref{yet_another)unram_simplificaiton}. Combining this with \Cref{KES lemma} gives the result. 
\end{proof}

\subsection{Residue characteristic $2$}

 We now give certain conditions under which  the congruence \eqref{unramified claim} holds (or rather, can be shown to hold) when the residue characteristic of $K$ is $2$. Thus for the rest of this subsection we assume that $K$ is a finite extension of $\mathbb{Q}_2$. 
 
 \subsubsection{Good reduction of the Jacobian}
 If the Jacobian $J/K$ of $C$ has good reduction then we have
\begin{equation} \label{trivial_quantities}
\mathfrak{f}(J/K)=0\quad \textup{ and } \quad \Phi(\bar{k})=0.
\end{equation}
 Moreover, we have the following: 
\begin{proposition} \label{valuation of discriminant}
Under the assumption that $J/K$ has good reduction, we have
\[v(\Delta_C)\equiv 0~~ ~~(\textup{mod }2).\]
\end{proposition}
\begin{proof}
Let $\mathcal{J}/\mathcal{O}_{K}$ be the N\'{e}ron model of $J$. The assumption that $J$ has good reduction over
$K$ implies that $\mathcal{J}[2]$ is a finite flat group scheme over $\mathcal{O}_{K}$ \cite[Proposition 20.7]{MR861974}. Letting $e$ denote the absolute ramification index of $K$, it is a theorem of Fontaine that $G_{K}^{u}$ acts trivially on $\mathcal{J}[2](\bar{K})=J[2]$ provided $u>2e-1$ \cite[Th{\'e}or{\`e}me  A]{MR807070}. Note that we are using Serre's upper numbering for the higher ramification groups. Let $F=K\left(\sqrt{\Delta_{C}}\right)$ and $G=\text{Gal}(F/K)$, noting that $F\subseteq K(J[2])$. Combining the above discussion with Herbrand's theorem (see, for example,  \cite[IV, Lemma 3.5]{MR554237}) we see that $G^{u}$ is trivial for $u\geq2e$. In particular, the conductor  of $F/K$, which we denote $\mathfrak{f}(F/K)$, satisfies $\mathfrak{f}(F/K)\leq2e$. 

Now suppose, for contradiction, that $v(\Delta_{C})$ is odd. We thus have $F=K\left(\sqrt{\pi}\right)$ for some uniformiser $\pi$ of $K$. Letting $\sigma$ denote the
non-trivial element of $G$ this gives \[v_{F}\big(\sigma\left(\sqrt{\pi}\right)-\sqrt{\pi}\big)=v_{F}(2)+1=2e+1,\]
where $v_F$ denotes the normalised valuation on $F$.
 From this we obtain $\mathfrak{f}(F/K)=2e+1$, contradicting the bound above. Thus  $v(\Delta_C)$ is even.
\end{proof}

\begin{remark}
This proposition is trivially true also if $J$ has good reduction and the residue characteristic of $K$ is odd. Indeed, then $J[2]$ is unramified  so  $\Delta_C$ is a square in $K^{\text{nr}}$. In particular, $\Delta_C$ has even valuation. 
\end{remark}

\begin{lemma} \label{vanishing_epsilon_good_red_jac}
Under the assumption that $J/K$ has good reduction, we have $\eta(C)=0$.
\end{lemma}

\begin{proof} 
 Let $\mathcal{C}/\mathcal{O}_K$ denote the minimal regular model of $C$. Since $J$ has good reduction, the curve $C/K$ is semistable and the dual graph  of $\mathcal{C}_{\bar{k}}$ is a tree \cite[Proposition 10.1.51]{MR1917232}.\footnote{By definition, the dual graph of $\mathcal{C}_{\bar{k}}$ is the finite connected graph with a vertex for each irreducible component of $\mathcal{C}_{\bar{k}}$, and  such that vertices corresponding to components $\Gamma_1$ and $\Gamma_2$ are joined by one edge for each ordinary double point of $\mathcal{C}_{\bar{k}}$ lying on both $\Gamma_1$ and $\Gamma_2$.} Moreover, as there are no exceptional curves in $\mathcal{C}_{\bar{k}}$, each leaf corresponds to a positive genus component (which necessarily has multiplicity $1$). Since the quotient of $\mathcal{C}_{\bar{k}}$ by the hyperelliptic involution has arithmetic genus zero, the hyperelliptic involution necessarily fixes every leaf, hence acts trivially on the dual graph. 
\end{proof}

\begin{cor} \label{eq:equation_holds_with_good_char_2}
If $J/K$ has good reduction then \eqref{unramified claim} holds for $C$.
\end{cor}

\begin{proof}
Combine \eqref{trivial_quantities} with \Cref{valuation of discriminant} and \Cref{vanishing_epsilon_good_red_jac}.
\end{proof}

\subsubsection{Semistable curves of genus $2$}

When the genus of $C$ is $2$ we can draw on results of Liu  \cite{MR1302311} to establish additional cases of  \eqref{unramified claim}.

\begin{proposition} \label{equation_genus_2_prop}
Suppose that $C/K$ is semistable and has genus $2$.  Then \eqref{unramified claim} holds for $C$.
\end{proposition}

\begin{proof}
This follows from  Liu's genus $2$ version of Ogg's formula   \cite[Theoreme 1]{MR1302311}. Specifically, combining Theoreme 1, Theoreme 2 and Proposition 1 of \cite{MR1302311}, one obtains (independently of the residue characteristic of $K$) 
\[f(J/K) \equiv v(\Delta_C)+n-1+\frac{d-1}{2}~~\textup{ (mod 2)}\]
where $n$ is the number of irreducible components of $\mathcal{C}_{\bar{k}}$ (as usual $\mathcal{C}$ denotes the minimal regular model of $C$ over $\mathcal{O}_K$) and where $d$  is  defined in the statement of Liu's Theoreme 1. In \cite[Section 5.2]{MR1302311}, Liu computes the term $\frac{d-1}{2}$ in a large number of cases, though not all in residue characteristic $2$, in terms of the structure of $\mathcal{C}_{\bar{k}}$ (more specifically, in terms of the `type' of $\mathcal{C}_{\bar{k}}$ as classified in \cite{MR0369362} and \cite{MR0201437}). This includes in particular all cases where $C/K$ has semistable reduction. It is then easy to establish \Cref{unramified claim} for all semistable curves of genus $2$ by combining this with the description,  detailed in \cite[Section 8]{MR1285783}, of the component group of a genus $2$ curve in terms of its type. 
\end{proof}

\subsection{Proof of  \Cref{prop:unram_cases_of_conjecture}} 
The above results combine to prove \Cref{prop:unram_cases_of_conjecture}, conditional on the soon to be proven \Cref{unchangedness of norm}.
 \begin{proof}[Proof of \Cref{prop:unram_cases_of_conjecture} (conditional on \Cref{unchangedness of norm})]
 Combine \Cref{unramifed reduction} with: \Cref{unramified kramer tunnell} in Case (i), \Cref{eq:equation_holds_with_good_char_2} in Case (ii), and \Cref{equation_genus_2_prop} in Case (iii).
\end{proof}

\section{The proof of Proposition 9.8} \label{proof of compatibility}

Maintaining the setup of the previous section, we now turn to proving \Cref{unchangedness of norm}.  We will  deduce this from \Cref{homomorphism theorem}. This is a result applying to arbitrary curves, and  which may be of independent interest. We begin by recalling a well-known description of the component group $\Phi(\bar{k})$ of the Jacobian of a (not necessarily hyperelliptic) curve in terms of its minimal regular model.

\subsection{The component group via the minimal regular model}  \label{sec:component_group_abstract_description} 

See \cite[Section 1]{MR1717533} and \cite[Chapter 9]{MR1045822} for details of what follows. Let $X$ be a smooth, proper, geometrically connected curve of genus $g$ over $K$, let $\mathcal{X}/\mathcal{O}_K$ be its minimal proper regular model, and let $\mathcal{X}_{\bar{k}}$ denote the special fibre of $\mathcal{X}$, base-changed to $\bar{k}$. Let $\{\Gamma_i\}_{i\in I}$   be the set of irreducible components of $\mathcal{X}_{\bar{k}}$. For each $i\in I$, denote by $d_i$ the multiplicity of $\Gamma_i$. Let $\mathbb{Z}^I$ denote the free $\mathbb{Z}$-module on the $\Gamma_i$ and define $\alpha:\mathbb{Z}^I\rightarrow\mathbb{Z}^I$ by
\begin{equation*}
\alpha(D)=\sum_{i\in I}(D\cdot\Gamma_i)\Gamma_i    
\end{equation*}
where  $D\cdot\Gamma_i$ is the intersection number between $D$ and $\Gamma_i$. Further, define $\beta:\mathbb{Z}^I\rightarrow\mathbb{Z}$ by setting $\beta(\Gamma_i)=d_i$ and extending linearly. We have  $ \text{im}(\alpha)\subseteq \ker(\beta)$. The natural action of $G_k$ on $\mathcal{X}_{\bar{k}}$ makes $\mathbb{Z}^I$ into a  $G_k$-module, and $\alpha$ is equivariant for this action. Endowing $\mathbb{Z}$ with trivial $G_k$-action, the same is true of $\beta$. In this way, both $\text{im}(\alpha)$ and $\ker(\beta)$ become $G_k$-modules.

Let $J/K$ be the Jacobian of $X$, and denote by $\Phi$ its N\'{e}ron component group. As explained in  \cite[Theorem 1.1]{MR1717533}, by work of Raynaud we have an exact sequence of $G_k$-modules
\begin{equation} \label{the_component_group_explicit_sequence}
0\longrightarrow \text{im}(\alpha)\longrightarrow \ker(\beta) \longrightarrow \Phi(\bar{k}) \longrightarrow 0.
\end{equation}
Denoting by $F\in G_k$ the Frobenius element, we thus have
\[\left| \Phi(k) \right|=\left| \left(\ker(\beta)/ \text{im}(\alpha)\right)^{F}\right|.\]
Note that $F$ acts on $\mathbb{Z}^I$ as a permutation of $I$, commuting with $\alpha$ and $\beta$, and preserving the arithmetic genus of components.

We recall also from \Cref{min reg def remark} that $X$ is deficient over $K$ if and only if
\[\text{gcd}_{i\in I}\left\{d_i \cdot |\text{orb}_{F}(\Gamma_i)|\right\}\]
does not divide $g-1$, where $\text{orb}_{F}(\Gamma_i)$ denotes the $F$-orbit of $\Gamma_i$. 
 
 \subsection{The result for general curves: statement}
 
Maintaining the notation of the previous subsection, denote by $\mathfrak{G}$ the group of all permutations of $I$ commuting with the maps $\alpha$ and $\beta$ and preserving the arithmetic genus of components. By assumption, $\mathfrak{G}$ acts on $\textup{im}(\alpha)$ and $\ker(\beta)$. Via the sequence \eqref{the_component_group_explicit_sequence} we have an induced action of $\mathfrak{G}$  on $\Phi(\bar{k})$. For $\sigma \in \mathfrak{G}$ we define 
\[q(\sigma)=\begin{cases}1~~&~~\textup{gcd}_{i\in I}\{d_i \cdot |\textup{orb}_{\sigma}(\Gamma_i)|\} ~\textup{ divides } ~g-1,\\
2~~&~~\textup{otherwise.} \end{cases}\]
In particular, viewing the Frobenius $F\in G_k$ as an element of $\mathfrak{G}$, we have $q(F)=2^{\epsilon(X/K)}$.
 
We will obtain \Cref{unchangedness of norm} as a consequence of the following:

\begin{theorem} \label{homomorphism theorem}
 Let $\mathfrak{G}$ be as above.  Then the map
$D:\mathfrak{G}\rightarrow \mathbb{Q}^{\times}/\mathbb{Q}^{\times 2}$
defined by
\[D(\sigma)=q(\sigma) \cdot \frac{\left|\Phi(\bar{k})\right|}{\left|\Phi(\bar{k})^\sigma\right|}\]
is a homomorphism.
\end{theorem}

Before proving \Cref{homomorphism theorem}, we explain how to use it to deduce \Cref{unchangedness of norm}. 

\subsection{Deducing \Cref{unchangedness of norm} from \Cref{homomorphism theorem}.} 

\begin{proof}[Proof of \Cref{unchangedness of norm} (conditional on  \Cref{homomorphism theorem})]
Maintaining the notation above, suppose $X=C$ is a hyperelliptic curve over $K$. The hyperelliptic involution $\iota$ of $C$ extends to an automorphism of the minimal regular model of $C$, and may therefore be viewed as an element of $\mathfrak{G}$. Moreover, as the induced automorphism $\iota_{*}$ of the Jacobian of $C$ is multiplication by $-1$, the action on $\Phi$ induced by $\iota$ is multiplication by $-1$ also (cf. proof of \cite[Theorem 1.1]{MR1717533}). Thus
\[\Phi(\bar{k})[2]=\Phi(\bar{k})^\iota.\]
Let  $L$ denote the unique quadratic unramified extension of $K$ and, as usual, denote by $F$ the Frobenius element in $G_k$. Then we have 
\[\left| H^1\left(\text{Gal}(k_L/k),\Phi(k_L)\right) \right|=\left| \frac{\ker\left(1+F| \Phi(\bar{k})^{F^2}\right)}{\text{im}\left(1-F|\Phi(\bar{k})^{F^2}\right)}\right|=\frac{\left|\Phi(\bar{k})^{\iota \circ F}\right|\cdot \left|\Phi(\bar{k})^{F}\right|}{\left|\Phi(\bar{k})^{F^2}\right|}.\]
Here for the first equality we are using the description of the cohomology of cyclic groups given in \cite[Section 8]{MR0219512}.

From \Cref{deficiency in extensions} we have $\epsilon(C/L)=0$, hence $ q(F^2)  =1$. Moreover, we have $\eta(C)=\text{ord}_2~ q(\iota)$, $\epsilon(C/K)=\text{ord}_2 ~q(F)$, and $\epsilon(C^L/K)=\text{ord}_2~q(\iota \circ F)$. To obtain the last equality we note that, since the formation of minimal regular models commutes with unramified base-change, we may identify the geometric special fibre of the minimal regular model of $C^L/K$  with that of $C/K$, save with $G_k$-action  twisted by the hyperelliptic involution.

On the other hand, it follows from \Cref{homomorphism theorem} that
\[D(F)D(\iota \circ F)D(F^2)=D(\iota)\]
as elements of  $\mathbb{Q}^{\times}/\mathbb{Q}^{\times 2}$. Taking 2-adic valuations of this equation, and interpreting the resulting terms via the discussion above, we obtain the congruence of  \Cref{unchangedness of norm}.
\end{proof}

 \subsection{Proof of \Cref{homomorphism theorem}.}
 
In what follows we take the notation introduced in the statement of \Cref{homomorphism theorem}. 

\begin{lemma} \label{pass_to_rep_theory}
For each $\sigma\in \mathfrak{G}$ we have 
\[D(\sigma)=\left|\det\big(\alpha | (\sigma-1)\mathbb{Q}^I\big)\right| .\]
(That is, as elements of $\mathbb{Q}^{\times}/\mathbb{Q}^{\times 2}$, $D(\sigma)$ is the absolute value of the determinant of $\alpha$ viewed as an endomorphism of $(\sigma-1)\mathbb{Q}^I$.)
\end{lemma}
 
\begin{proof} 
Fix $\sigma\in \mathfrak{G}$, define $d=\text{gcd}_{i\in I} \{d_i\}$ and $d'(\sigma)=\text{gcd}_{i\in I} \{|\textup{orb}_{\sigma}(\Gamma_i)|\cdot d_i\}.$ By \cite[Proof of Theorem 1.17]{MR1717533} we have
\[ \left|\Phi(\bar{k})^\sigma\right|\cdot q(\sigma)=\left|\frac{\ker(\beta)^\sigma}{\text{im}(\alpha)^\sigma}\right|\cdot \frac{d'(\sigma)}{d}.\]
We remark that, in order to apply the cited result, we are using the assumption that $\mathfrak{G}$ preserves the arithmetic genus of components. 

In what follows, to ease notation we write $\Lambda$ for the $\mathbb{Z}[\mathfrak{G}]$-module $\mathbb{Z}^I$. 
Now \[\ker(\beta)^\sigma/\text{im}(\alpha)^\sigma\cong \ker\Big(\beta:\Lambda^\sigma/\textup{im}(\alpha)^\sigma\longrightarrow d'(\sigma)\mathbb{Z}\Big).\]
We now apply the snake lemma to the commutative diagram with exact rows
\[
\xymatrix{0\ar[r] & \frac{\Lambda^{\sigma}}{\textup{im}(\alpha)^{\sigma}}\ar[d]^{\beta_{1}}\ar[r] & \frac{\Lambda}{\textup{im}(\alpha)}\ar[d]^{\beta_{2}} \ar[r] & \frac{\Lambda}{\textup{im}(\alpha)+\Lambda^{\sigma}}\ar[r]\ar[d]^{\beta_{3}} & 0\\
0\ar[r] & d'(\sigma)\mathbb{Z}\ar[r] & d\mathbb{Z}\ar[r] & d\mathbb{Z}/d'(\sigma)\mathbb{Z}\ar[r] & 0,
}
\]
where each vertical arrow is induced by $\beta$. Noting that all vertical arrows are surjective, this gives
\[D(\sigma)=\frac{dq(\sigma)^2}{d'(\sigma)}|\ker(\beta_3)|=q(\sigma)^2\frac{d^2}{d'(\sigma)^2}\left| \frac{\Lambda}{\textup{im}(\alpha)+\Lambda^{\sigma}} \right|.\]
Thus as a function $\mathfrak{G} \rightarrow \mathbb{Q}^{\times}/\mathbb{Q}^{\times 2}$ we have
\[D(\sigma)=\left| \frac{\Lambda}{\textup{im}(\alpha)+\Lambda^{\sigma}} \right|.\]
Now $\sigma-1$ induces an isomorphism
\[\frac{\Lambda}{\textup{im}(\alpha)+\Lambda^{\sigma}} \stackrel{\sim}{\longrightarrow}\frac{(\sigma-1)\Lambda}{(\sigma-1)\alpha \left(\Lambda\right)}=\frac{(\sigma-1)\Lambda}{\alpha\left((\sigma-1)\Lambda\right)},\]
where for the last equality we use that $\sigma$ commutes with $\alpha$. 

To conclude, note that $(\sigma-1)\Lambda$ is a free $\mathbb{Z}$-module of finite rank, and $\alpha$ is a linear endomorphism of this group. By properties of Smith normal form, the order of the group  
\[\frac{(\sigma-1)\Lambda}{\alpha\left((\sigma-1)\Lambda\right)}\]
 is equal to the absolute value of the determinant of $\alpha$ as an endomorphism of the $\mathbb{Q}$-vector space $(\sigma-1)\mathbb{Q}^I$. This gives the result.  
\end{proof}

The passage from $\mathbb{Z}[\mathfrak{G}]$-modules to $\mathbb{Q}[\mathfrak{G}]$-modules provided by \Cref{pass_to_rep_theory} allows us to make use of representation theory in characteristic zero. Note that the matrix representing $\alpha$ on $\mathbb{Q}^I$  with respect to the natural permutation basis  is symmetric (it's just the intersection matrix associated to $\mathcal{X}_{\bar{k}}$). We thus see that the minimal polynomial of $\alpha$ as an endomorphism of $\mathbb{Q}^I$ splits over $\mathbb{R}$. Moreover, the kernel of $\alpha$ is  $(\sum_{i\in I}d_i \Gamma_i)\cdot \mathbb{Q}$, which is fixed by $\mathfrak{G}$.   These observations motivate (and allow us to apply) the following lemmas.

\begin{lemma} \label{rep lemma 1}
Let $G$ be a finite cyclic group with generator $\sigma$, and  let $V$ be a $\mathbb{Q}[G]$-representation. Let $\alpha\in \textup{End}_{\mathbb{Q}[G]}V$ be a $G$-endomorphism of $V$ whose minimal polynomial splits over $\mathbb{R}$ and is such that $\ker(\alpha) \subseteq V^G$. Then
\[\det\left(\alpha | (\sigma-1)V \right)\equiv\det\left(\alpha | V_{-1,\sigma}\right)~(\textup{mod }\mathbb{Q}^{\times 2}),\]
where $V_{-1,\sigma}$ is the $(-1)$-eigenspace for $\sigma$ on $V$.   
\end{lemma}

\begin{proof}
Let $V=\bigoplus_{i=1}^{n}V_i^{d_i}$ be an isotypic decomposition of $V$, so that each $V_i$ is an irreducible $\mathbb{Q}[G]$-representation and $V_i\ncong V_j$ for $i\neq j$. Suppose, without loss of generality, that $V_1$ is the trivial representation. Now $\alpha$ preserves this decomposition and $(\sigma-1)V=\bigoplus_{i=2}^{n}V_i^{d_i}$. Since $\ker(\alpha) \subseteq V^G$, the restriction of $\alpha$ to each $V_i^{d_i}$ is non-singular. Thus we reduce to the case where $\alpha$ is non-singular and $V=W^d$ for an irreducible non-trivial $\mathbb{Q}[G]$-representation $W$. Let $\chi$ be the character of a complex irreducible constituent of $W$. We can suppose that $\chi$ is non-real (so that $\chi(\sigma)\notin \{\pm 1\}$), in which case we wish to show that  $\textup{det}\left(\alpha \right) \in \mathbb{Q}^{\times 2}$. Now $\text{End}_{\Q[G]}V\cong \text{Mat}_d\left(\text{End}_{\Q[G]}W\right)$ is a finite dimensional simple algebra over $\Q$. Set $D=\text{End}_{\Q[G]}W$ so that $D$ is a division algebra, and let $K/\Q$ be the centre of $D$. We  have $K\cong \Q(\chi)$ as $\mathbb{Q}$-algebras, where $\Q(\chi)$ is the character field of $\chi$. Se, for example, \cite{MR0122892} for proofs of the representation theoretic facts used above. Note that $K/\Q$ is abelian.

Now via the diagonal embedding of $K$ into $\text{End}_{\Q[G]}V$, $V$ becomes a $K[G]$-module. Since $K$ is the centre of $\text{End}_{\Q[G]}V$, each $\Q[G]$-endomorphism of $V$ is in fact $K$-linear, so  we may view $\alpha$ as a $K[G]$-endomorphism of $V$. Denoting $\text{det}_K(\alpha)$ the determinant of $\alpha$ viewed as a $K$-endomorphism, we have 
\[\text{det}(\alpha)=N_{K/\Q} \left(\text{det}_{K}\left(\alpha \right)\right)\] 
 (see, for example, \cite[Theorem A.1]{MR0222054}). 

As $\chi$ is assumed non-real, the field $K$ in not totally real hence there is an index $2$ totally real subfield $K^{+}$ of $K$ (recall that $K/\mathbb{Q}$ is abelian). We claim that  $\det(\alpha)$ is in $K^{+}$. Indeed, since the minimal polynomial of $\alpha$ as a $\mathbb{Q}$-endomorphism of $V$ splits over $\mathbb{R}$, each root of the minimal polynomial of $\alpha$ as a $K$-endomorphism of $V$ is totally real. It follows that  $\det(\alpha)$ is a product of totally real numbers, hence in $K^{+}$. Thus 
\[\det(\alpha)=N_{K/\Q} \left(\text{det}_{K}(\alpha)\right)= N_{K^+/\Q}(\text{det}_{K}(\alpha))^2\in \Q^{\times 2}\]
as desired.
\end{proof}

\begin{lemma} \label{rep_theory_lemma_2}
Let $G$ be a finite group and $V$ a $\mathbb{Q}[G]$-representation. Let $\alpha\in \textup{End}_{\mathbb{Q}[G]}V$ be a $G$-endomorphism of $V$ whose minimal polynomial splits over $\mathbb{R}$ and is such that $\ker(\alpha) \subseteq V^G$. 
Then the function
$\phi:G\rightarrow \mathbb{Q}^{\times}/\mathbb{Q}^{\times 2}$
defined by
\[\phi(\sigma)=\textup{det}\left(\alpha | V_{-1,\sigma}\right)\]
is a homomorphism.  
\end{lemma}

\begin{proof}
Similarly to the proof of \Cref{rep lemma 1}, by considering an isotypic decomposition of $V$ we reduce to the case where $\alpha$ is non-singular and  $V=W^d$ for some $d\geq 1$ and  irreducible $\mathbb{Q}[G]$-representation  $W$. As in that proof, let $D$ be the division algebra $D=\text{End}_{\Q[G]}W$ so that $\text{End}_{\Q[G]}V\cong \text{Mat}_d(D)$, let $K/\Q$ be the centre of $D$, and let $\chi$ be the character of a complex irreducible constituent of $W$. Once again we have $K\cong \Q(\chi)$, $K/\Q$ is abelian, and we may view $\alpha$ as a $K[G]$-endomorphism of $V$.  Moreover, we similarly have
\[\text{det}\left( \alpha | V_{-1,\sigma} \right)=N_{K/\Q} \left(\text{det}_{K}\left(\alpha | V_{-1,\sigma}\right)\right) \]  
for all $\sigma \in G$. Denoting by $\phi_K$   the function $G\rightarrow K^{\times}/K^{\times2}$ given by 
\[\phi_K(\sigma)= \text{det}_K\left(\alpha | V_{-1,\sigma}\right),\]
we thus have $\phi =N_{K/\mathbb{Q}}\circ\phi_K$.

First suppose that $K$ in not totally real. Then there is an index $2$ totally real subfield $K^{+}$ of $K$. Since the minimal polynomial of $\alpha$ as a $\mathbb{Q}$-endomorphism of $V$ splits over $\mathbb{R}$,  $\text{det}\left(\alpha | V_{-1,\sigma}\right)$ lies  in $K^{+}$. Thus
\[\phi (\sigma)=N_{K/\mathbb{Q}}\circ\phi_K(\sigma)=\left(N_{K^{+}/\mathbb{Q}}  \circ \phi_K(\sigma)\right)^2 \in \mathbb{Q}^{\times 2}, \]
hence $\phi $ is trivial in this case.

Now assume that $K$ is totally real, or equivalently that $\chi$ is real valued. Let $m$ be the Schur index of $\chi$ (over $\mathbb{Q}$ or equivalently $K$). Suppose first that $\chi$ is realisable over $\mathbb{R}$. Then, via a chosen embedding $K\hookrightarrow \mathbb{R}$, we have $V\otimes_K \mathbb{R} \cong U^{md}$ for some irreducible real representation $U$. Fix $\sigma \in G$. Then $V_{-1,\sigma} \otimes_{K} \mathbb{R} = (V \otimes_{K} \mathbb{R})_{-1,\sigma}=(U_{-1,\sigma})^{md}$. View $\alpha$ as a matrix $M\in \text{Mat}_{md}(\mathbb{R})$ via the identification $\text{End}_{\mathbb{R}[G]}\left(U^{md}\right)\cong \text{Mat}_{md}(\text{End}_{\mathbb{R}[G]}U)=\text{Mat}_{md}(\mathbb{R})$. The determinant of $\alpha$, viewed as a $K$-endomorphism of $(U_{-1,\sigma})^{md}$, is then equal to $\text{det}(M)^{\text{dim}U_{-1,\sigma}}$. In fact, one sees that $\text{det}(M)$ is equal to $\text{Nrd}(\alpha)\in K^{\times}$ where here $\text{Nrd}$ denotes the reduced norm on the central simple algebra $A=\text{End}_{\Q[G]}V$ over $K$.

 We claim that the congruence
\begin{equation} \label{matrix_congruence}
\text{dim}U_{-1,\sigma}+\text{dim}U_{-1,\tau}\equiv\text{dim}U_{-1,\sigma \tau}~~\text{(mod 2)}
\end{equation}
holds for all $\sigma$ and $\tau$ in $G$. Combined with the above discussion this shows that  $\phi_K$, hence $\phi$, is a homomorphism in this case.    To prove the claim we note that  $U$ is a real vector space and each $\sigma \in G$ acts on $U$ as a finite order matrix $N_\sigma$ which is hence diagonalisable over $\mathbb{C}$. Base-changing to $\mathbb{C}$, diagonalising $N_\sigma$, and noting that the eigenvalues of $N_\sigma$ are roots of unity appearing in complex-conjugate pairs, one sees that for each $\sigma \in G$ we have
\[(-1)^{\text{dim}U_{-1,\sigma}}=\text{det}(N_\sigma).\]
The congruence \eqref{matrix_congruence} now follows from multiplicativity of the determinant.

Finally, suppose that $\chi$ is not realisable over $\mathbb{R}$. Then we have $V\otimes_K \mathbb{\mathbb{C}} \cong U^{md}$ where $U$ is an irreducible representation over $\mathbb{C}$ such that $U$,   hence $U^{md}$, possesses a non-degenerate $G$-invariant alternating form. Denote by $\left\langle ~ ,~ \right\rangle$ such a form on $U^{md}$. The argument for the previous case again gives $\text{det}_K\left(\alpha | V_{-1,\sigma}\right)= \text{Nrd}(\alpha)^{\text{dim}U_{-1,\sigma}}$. We claim that  $\text{dim}U_{-1,\sigma}$ is even for each $\sigma \in G$, from which it follows that $\phi_K$,  hence $\phi$, is trivial. Indeed, the pairing $\left \langle ~,~\right \rangle$ gives a $G$-equivariant isomorphism from $U$ to its dual $U^*$. This isomorphism respects the $\sigma$-eigenspace decomposition on each side and hence restricts to an isomorphism $U_{-1,\sigma} \stackrel{\sim}{\longrightarrow} U_{-1,\sigma}^*$ whose associated bilinear pairing is non-degenerate and alternating. Thus $\dim U_{-1,\sigma}$ is even.
\end{proof}

\begin{proof}[Proof of \Cref{homomorphism theorem}]
In the notation of \Cref{sec:component_group_abstract_description}, let $V$ denote the $\mathbb{Q}[\mathfrak{G}]$-representation $\mathbb{Q}^I$. For $\sigma \in \mathfrak{G}$, combining \Cref{pass_to_rep_theory} with \Cref{rep lemma 1} we see that we have 
\[D(\sigma)=\det\left(\alpha | V_{-1,\sigma}\right)~\in \mathbb{Q}^{\times}/\mathbb{Q}^{\times 2}.\]
(That the assumptions of \Cref{rep lemma 1} are satisfied follows from the discussion preceding the statement of that lemma.)
The result now follows from \Cref{rep_theory_lemma_2}.
\end{proof}

\subsection{Computing Tamagawa numbers modulo squares} 

The proof of \Cref{homomorphism theorem} facilitates the computation of Tamagawa numbers of hyperelliptic curves, at least up to squares, and to end the section we record this in the following proposition. To make the statement more self-contained we summarise now the notation used.

\begin{notation}
We take the following notation:
\begin{itemize}
\item $K$ a non-archimedean local field,
\item $X/K$ a smooth, proper, geometrically connected curve of genus $g$,
\item $\Phi$ the N\'{e}ron component group (scheme) of the Jacobian of $X$,
\item $\mathcal{X}/\mathcal{O}_K$ the minimal proper regular model of $X$,
\item $\mathcal{X}_{\bar{k}}$ the special fibre of $\mathcal{X}$ base-changed to $\bar{k}$,
\item $ \Gamma_1,...,\Gamma_n $ the  irreducible components of $\mathcal{X}_{\bar{k}}$, and  $r_i$ the size of the $G_k$-orbit of $\Gamma_i$,
 
\item $\epsilon(X/K)\in \{0,1\}$, defined to be equal to $1$ if $X$ is deficient over $K$, and $0$ otherwise,
\item $F$ the Frobenius element in $G_k$.
 \end{itemize}
\end{notation}

\begin{proposition} \label{tam computations}
Take the notation above. Moreover, let $S_1,...,S_t$ be the even sized orbits of $G_k$ on the set $\{\Gamma_1,...,\Gamma_n\}$ of irreducible components. For each $1\leq i \leq t$  let $m_i=|S_i|$, let  $\Gamma_{i,1}$ be a representative of the orbit $S_i$, and define
\[\epsilon_i=\sum_{i=0}^{m_i-1}(-1)^i F^i (\Gamma_{i,1}).\]
 Then 
\[2^{\epsilon(X/K)}\cdot \frac{|\Phi(\bar{k})|}{|\Phi(k)|}  \equiv  \left| \textup{det}\left(\frac{1}{m_j} \left\langle \epsilon_{i} ,
 \epsilon_{j}\right\rangle \right)_{1\leq i,j \leq t}\right|~~(\textup{mod }~\mathbb{Q}^{\times 2}),\]
where $\left\langle  \cdot, \cdot \right \rangle$ denotes the intersection pairing on $\mathcal{X}_{\bar{k}}$.
\end{proposition}

\begin{proof}
 \Cref{pass_to_rep_theory} combined with \Cref{rep lemma 1} gives
\begin{equation} \label{alpha_formula_in_practice}
2^{\epsilon(X/K)}\cdot \frac{|\Phi(\bar{k})|}{|\Phi(k)|} \equiv  \left| \textup{det}\left(\alpha | V_{-1}\right)\right| ~~(\text{mod }~\mathbb{Q}^{\times 2}),
\end{equation}
where $V_{-1}$ denotes the $(-1)$-eigenspace of  $F$ on the permutation module $V=\Q^I$, and $\alpha$ is the linear map defined from the intersection pairing, as detailed in \Cref{sec:component_group_abstract_description}. Now $\{\epsilon_1,...,\epsilon_t\}$ forms a basis for  $V_{-1}$ and, using $G_k$-invariance of the intersection pairing, for each $1\leq i \leq t$ we have
\[\alpha(\epsilon_i)=\sum_{j=1}^{t}\left \langle\epsilon_i , \Gamma_{j,1}\right \rangle \epsilon_j=\sum_{j=1}^{t}\frac{1}{r_j}\left \langle \epsilon_i ,\epsilon_j\right \rangle \epsilon_j.\]
Combined with \Cref{alpha_formula_in_practice} this gives the result.
\end{proof}

\section{Ramified extensions in odd residue characteristic: generalities} \label{min reg model ramified}

Let $K$ be a nonarchimedean local field of odd residue characteristic. In this section we  consider \Cref{Kramer Tunnell} when the quadratic extension $L/K$ is ramified. Specifically, across Sections \ref{min reg model ramified} to \ref{completion_ram_quad_odd_res} we will prove the following:

\begin{proposition}[=\Cref{later_ramified_cases}] \label{prop:ram_cases_of_conjecture}
Let $C/K$ be a hyperelliptic curve and let $L/K$ be a ramified quadratic extension. If $C/K$ has semistable reduction, then \Cref{Kramer Tunnell} holds for $C$ and the extension $L/K$.
\end{proposition}

Thus for this section we fix a ramified quadratic extension $L/K$, and fix a hyperelliptic curve $C/K$ with semistable reduction. As usual we denote by $J$ the Jacobian of $C$.  Recall from \Cref{norm map as Tamagawa numbers} that, since $K$ has odd residue characteristic, we can express the cokernel of the local norm map in terms of Tamagawa numbers:
\begin{equation} \label{equation_norm_as_tamag_product}
\dim J(K)/\N J(L) = \text{ord}_2 \frac{c(J/K)c(J^L/K)}{c(J/L)}.
\end{equation}
We begin by describing a method for computing the ratio $\frac{c(J/L)}{c(J/K)}$ up to rational squares for general semistable curves. Separately, we will compute $c(J^L/K)$ up to squares by analysing the minimal regular model of the quadratic twist of $C$ by $L$. Since $C^L/K$ is not semistable, we use results from \Cref{proof of compatibility} to facilitate this computation. As we shall see, the terms in \Cref{Kramer Tunnell} involving deficiency and root numbers will naturally appear along the way. For a more   precise description of the strategy for proving \Cref{prop:ram_cases_of_conjecture}, see \Cref{ram:strategy} below.

\subsection{The minimal proper regular model of a semistable curve} \label{min reg model}
For proofs and more details of what follows we refer to \cite{SGA7IX}. For the specific formulation detailed below we refer to \cite[Section 2]{DDMM18} and the references therein.

Denote by $\mathcal{C}/\mathcal{O}_K$ the minimal proper regular model of $C$, and denote by
$\mathcal{C}_{\bar{k}}$ the special fibre of $\mathcal{C}$,  base-changed to $\bar{k}$. Since $C/K$ is assumed semistable, $\mathcal{C}_{\bar{k}}$ is a semistable curve over $\bar{k}$.  Let $\Upsilon_C$ denote the dual graph $\mathcal{C}_{\bar{k}}$; by definition this is the finite connected graph with a vertex for each irreducible component of $\mathcal{C}_{\bar{k}}$, and  such that vertices corresponding to components $\Gamma_1$ and $\Gamma_2$ are joined by one edge for each ordinary double point of $\mathcal{C}_{\bar{k}}$ lying on both $\Gamma_1$ and $\Gamma_2$ (in particular $\Upsilon_C$ may have loops and multiple edges). We view $\Upsilon_C$ as a metric space where we give each edge length $1$. Denote by $H_1(\Upsilon_C,\mathbb{Z})$ the first singular homology group of $\Upsilon_C$. Since $\mathcal{C}_{\bar{k}}$ is the base-change from $k$ to $\bar{k}$ of the special fibre of $\mathcal{C}$,  $H_1(\Upsilon_C,\mathbb{Z})$ carries a natural $G_k$-action.\footnote{Due to the possible presence of loops and multiple edges in $\Upsilon_C$, the $G_k$-action on $H_1(\Upsilon_C,\mathbb{Z})$ need not be fully determined by the $G_k$-action on the irreducible components of $\mathcal{C}_{\bar{k}}$. When there is ambiguity  one needs some additional information concerning the ordinary double points to pin down the action; see e.g. \cite[Section 2.1]{DDMM18} for more details.} Moreover, $H_1(\Upsilon_C,\mathbb{Z})$ carries a natural non-degenerate,  symmetric,  $G_k$-invariant bilinear pairing 
\[P:H_1(\Upsilon_C,\mathbb{Z})\times  H_1(\Upsilon_C,\mathbb{Z})\rightarrow \mathbb{Z}\] (informally,  $P(\gamma,\gamma')$ is the signed length of $\gamma\cap \gamma'$). The pairing $P$ induces an injection 
\begin{equation} \label{eq:pairing_injection_homology}
H_1(\Upsilon_C,\mathbb{Z})\hookrightarrow H_1(\Upsilon_C,\mathbb{Z})^\vee:=\textup{Hom}\left(H_1(\Upsilon_C,\mathbb{Z}),\mathbb{Z}\right),
\end{equation}
sending $\gamma$ to $P(-,\gamma)$.
The component group $\Phi(\bar{k})$ of $J/K$ is then $G_k$-equivariantly isomorphic to the  cokernel of this map:
\begin{equation} \label{component_gorup_homology}
\Phi(\bar{k})=H_1(\Upsilon_C,\mathbb{Z})^\vee/H_1(\Upsilon_C,\mathbb{Z}).
\end{equation}
In particular, we have 
\begin{equation}\label{eq:tam_1}
c(J/K)=\left|\left(\frac{H_1(\Upsilon_C,\mathbb{Z})^\vee}{H_1(\Upsilon_C,\mathbb{Z})}\right)^{G_k}\right|.\end{equation}
Moreover, the root number $w(J/K)$ of $J$ is encoded in the $G_k$-invariants of $H_1(\Upsilon_C,\mathbb{Z})$:
\begin{equation*}
w(J/K)=(-1)^{\textup{rk}H_1(\Upsilon_C,\mathbb{Z})^{G_k}}.
\end{equation*}
If we replace $K$ by $L$ and repeat the above constructions for the base-change $C_L$ of $C$ to $L$, then the dual graph $\Upsilon_{C_L}$ is obtained from $\Upsilon_C$ by replacing each edge by a path consisting of $2$ edges.\footnote{To see this one can argue as follows. First, the base change of $\mathcal{C}$ to the ring of integers $\mathcal{O}_{K^{\textup{nr}}}$ coincides with the minimal regular model $\mathcal{C}'/\mathcal{O}_{K^{\textup{nr}}}$ of $C$ over $\mathcal{O}_{K^{\textup{nr}}}$, hence $\mathcal{C}_{\bar{k}}$ coincides with the special fibre of $\mathcal{C}'$. Since $\mathcal{C}'$ is both semistable and regular, each singular point $x$ of $\mathcal{C}_{\bar{k}}$ is a split ordinary double point of thickness $1$ (in the sense of \cite[Definition 10.3.23]{MR1917232}). After base-changing $\mathcal{C}'$ to $\mathcal{O}_{L^{\textup{nr}}}$, the point $x$ becomes an ordinary double point of  thickness $2$ in  $\mathcal{C}'\times_{\mathcal{O}_{K^{\textup{nr}}}}\mathcal{O}_{L^{\textup{nr}}}$, as follows from the description of the completed local ring at $x$ given in \cite[Corollary 3.22]{MR1917232} (the factor $2$ arising as the ramification index of $L/K$). The minimal regular model of $C_L$ over $\mathcal{O}_{L^{\textup{nr}}}$ is then obtained by blowing up $\mathcal{C}'\times_{\mathcal{O}_{K^{\textup{nr}}}}\mathcal{O}_{L^{\textup{nr}}}$ once at each such $x$, which has the claimed effect on the dual graph (cf. \cite[Lemma 10.3.21, Corollary 10.3.25]{MR1917232}).
} In particular, the homology of the new dual graph with its $G_k$-action is unchanged, but the pairing gets multiplied by $2$. Thus we have 
\begin{equation}\label{eq:tam_2}
c(J/L)=\left|\left(\frac{H_1(\Upsilon_C,\mathbb{Z})^\vee}{2H_1(\Upsilon_C,\mathbb{Z})}\right)^{G_k}\right| \quad\textup{ and }\quad w(J/L)=(-1)^{\textup{rk}H_1(\Upsilon_C,\mathbb{Z})^{G_k}}. \end{equation}

\subsection{The group $\mathfrak{B}_{C/K}$.}
Following work of Betts--Dokchitser \cite{MR3933907}, the quantities appearing in \eqref{eq:tam_1} and  \eqref{eq:tam_2} can be neatly packaged together in the following way. Temporarily writing $\Lambda=H^1(\Upsilon_C,\mathbb{Z})$, define the finite abelian group $\mathfrak{B}_{C/K}$ by
\begin{equation} \label{eq:betts_group_def_C}
\mathfrak{B}_{C/K}=\textup{im}\left(H^1\big(G_k,\Lambda\big)\longrightarrow H^1\big(G_k,\Lambda^\vee\big)\right),
\end{equation}
where the map is induced by \eqref{eq:pairing_injection_homology}. Combining \eqref{eq:tam_1} and  \eqref{eq:tam_2} with \cite[Theorem 1.4.2]{MR3933907} then gives
\begin{equation} \label{betts_dokchitser_input}
w(J/L)\cdot (-1)^{\textup{ord}_2\frac{c(J/L)}{c(J/K)}}=(-1)^{\dim\mathfrak{B}_{C/K}[2]}.
\end{equation}
 
 \begin{remark}
 If $\mathcal{T}$ denotes the toric part of the Raynaud parametrisation of $J/K$, and $X(\mathcal{T})$ denotes its character group, then $X(\mathcal{T})$ carries a natural $G_k$-action and a non-degenerate symmetric pairing $X(\mathcal{T})\otimes X(\mathcal{T})\rightarrow  \mathbb{Z}$ (see \cite{MR0427326,SGA7IX}). As explained in \cite[Section 2]{DDMM18}, it follows from work of Raynaud  that $X(\mathcal{T})\cong H_1(G_k,\mathbb{Z})$ as $G_k$-modules  with a pairing. One can then alternatively obtain \eqref{betts_dokchitser_input} directly from \cite[Theorem 1.1.1 (i)]{MR3933907}. 
 \end{remark}
 
 \subsection{Strategy for the proof of \Cref{prop:ram_cases_of_conjecture}} \label{ram:strategy}
 
 In light of \eqref{equation_norm_as_tamag_product} and \eqref{betts_dokchitser_input}, \Cref{Kramer Tunnell} for $C$ and $L/K$ is the equivalent to the assertion
\begin{equation} \label{ramified_goal}
(-1)^{\dim \mathfrak{B}_{C/K}[2]}\stackrel{?}{=}(-1)^{\textup{ord}_2c(J^L/K)}(\Delta_C,L/K)(-1)^{\epsilon(C/K)+\epsilon(C^L/K)} .
\end{equation}
In   \Cref{betts_group_parity} below we give a general result, \Cref{prop:betts_group_mod_2}, which facilitates  the computation of the parity of the dimension of the $2$-torsion of the group  
  \[\mathfrak{B}_\Lambda:=\textup{im}\big(H^1\big(G_k,\Lambda\big)\longrightarrow H^1\big(G_k,\Lambda^\vee\big)\big)\] 
associated to an arbitrary $G_k$-lattice $\Lambda$ equipped with a non-degenerate symmetric bilinear pairing. 

In \Cref{sec:clusters}  we summarize results from \cite{DDMM18} which give  an explicit  description of the lattice $H_1(\Upsilon_C,\mathbb{Z})$ attached to a semistable hyperelliptic curve $C/K:y^2=f(x)$ in terms of combinatorial data associated to the $p$-adic distances between the roots of $f(x)$. Combined with the results of \Cref{betts_group_parity} mentioned above, this enables the explicit computation of  $\dim\mathfrak{B}_{C/K}[2]~(\textup{mod }2)$ for arbitrary semistable hyperelliptic curves; we present the result of this computation as \Cref{mod 2 betts group cor}. (Strictly speaking, we only carry out   these computations over a suitably large  odd degree unramified extension of $K$. This suffices for the application to \Cref{Kramer Tunnell}  thanks to \Cref{odd degree extension}, and has the advantage that several statements in \cite{DDMM18} simplify after such an extension.)

Separately, in \Cref{sec:ram_twist_hyp_curve} we  present an explicit combinatorial description of the minimal proper regular model of a ramified quadratic twist of a semistable hyperelliptic curve. This will be deduced from work of Faraggi--Nowell \cite{MR4201122} which more generally describes the minimal regular SNC model of a hyperelliptic curve $X$ over a local field of odd residue characteristic, under the assumption that $X$ attains semistable reduction after a tamely ramified extension of the base field. We combine this description with \Cref{tam computations}  to describe explicitly the quantity $(-1)^{\textup{ord}_2 c(J^L/K)+\epsilon(C^L/K)}$; we present this result as \Cref{main_ram_quad_twist_cor}.

 Finally, in \Cref{completion_ram_quad_odd_res}, we combine Corollaries \ref{mod 2 betts group cor} and \ref{main_ram_quad_twist_cor} to establish \eqref{ramified_goal}.

 \section{The parity of   $\dim\mathfrak{B}_{\Lambda}[2]$} \label{betts_group_parity}

The aim of this section is to prove \Cref{prop:betts_group_mod_2}, which gives an explicit criterion for determining the parity of $\dim \mathfrak{B}_\Lambda[2]$, where $\mathfrak{B}_{\Lambda}$ is the group defined in \Cref{betts_group_defi_section} below and considered by Betts--Dokchitser in \cite{MR3933907}. This can be viewed as a complement to the results of \cite[Section 2]{MR3933907}.

Let $k$ be a finite field. Take $\Lambda$ to be a (discrete) $\mathbb{Z}[G_k]$-module, free and of finite rank as a $\mathbb{Z}$-module, and equipped with a non-degenerate $G_k$-invariant symmetric bilinear pairing 
\begin{equation}\label{eq:pairing_on_lambda}
\left \langle~,~\right \rangle: \Lambda \times \Lambda \longrightarrow \mathbb{Z}.\end{equation}
We extend $\left \langle ~,~\right \rangle$  bilinearly to a pairing on the $\mathbb{Q}[G_k]$-module $V:=\Lambda \otimes_\mathbb{Z} \mathbb{Q}$ and write $\Lambda^\vee$ for the dual lattice 
\begin{equation}  \label{reinerpret_lambda}
\Lambda^\vee=\{v\in V~~\colon~~\left \langle v, \lambda \right \rangle \in \mathbb{Z}~~\textup{for all } \lambda \in \Lambda\}.
\end{equation}
The map $v \mapsto \left \langle v,-\right \rangle$ identifies $\Lambda^\vee$ with $\textup{Hom}(\Lambda,\mathbb{Z})$. We denote by $\Phi$ the finite abelian group $\Lambda^\vee/\Lambda$, the discriminant group of the lattice.  The pairing on $V$ restricts to a $G_k$-invariant pairing $\Lambda^\vee \times \Lambda \rightarrow \mathbb{Z}$, and further induces a non-degenerate symmetric bilinear pairing
\begin{equation} \label{pairing_on_discriminant_group}
\overline{\left \langle ~,~\right \rangle}:\Phi \times \Phi \longrightarrow \mathbb{Q}/\mathbb{Z}.
\end{equation}

\subsection{The group $\mathfrak{B}_\Lambda$} \label{betts_group_sub_sec} \label{betts_group_defi_section}

Define the finite abelian group  
\[\mathfrak{B}_{\Lambda}:= \textup{im}\big(H^1(G_k,\Lambda)\longrightarrow H^1(G_k,\Lambda^\vee)\big) =\ker\big(H^1(G_k,\Lambda^\vee)\longrightarrow H^1(G_k,\Phi)\big). \]
Consider the pairing
\begin{equation} \label{pairing_on_B}
H^1(G_k,\Lambda)\otimes H^1(G_k,\Lambda)\longrightarrow H^2(G_k,\mathbb{Z})\stackrel{\sim}{\longrightarrow}H^1(G_k,\mathbb{Q}/\mathbb{Z})\stackrel{\textup{eval. at }F}{\longrightarrow}\mathbb{Q}/\mathbb{Z}.
\end{equation}
Here the first map is composition of cup-product with the pairing \eqref{eq:pairing_on_lambda}, and the second is the inverse of the coboundary map $\delta:H^1(G_k,\mathbb{Q}/\mathbb{Z})\rightarrow H^2(G_k,\mathbb{Z})$ arising from the short exact sequence 
\begin{equation}\label{basic_exact_seq}
0\longrightarrow \mathbb{Z}\longrightarrow \mathbb{Q}\longrightarrow \mathbb{Q}/\mathbb{Z}\longrightarrow 0.
\end{equation}
The final map is given by evaluating cocycles at the Frobenius element $F\in G_k$. 

We have the following result of Betts--Dokchitser.

\begin{proposition}
 Lifting to $H^1(G_k,\Lambda)$ and applying the pairing \eqref{pairing_on_B} induces a non-degenerate antisymmetric bilinear pairing 
\[(~,~):\mathfrak{B}_\Lambda \times \mathfrak{B}_\Lambda \longrightarrow \mathbb{Q}/\mathbb{Z}.\]
\end{proposition}

\begin{proof}
This is \cite[Proposition 2.2.2]{MR3933907}. The key point is the antisymmetry of the top pairing, and non-degeneracy of the bottom pairing, in the commutative diagram 
 \begin{equation} \label{commutative_betts_cup_diag}
\begin{tikzcd}
H^1(G_k,\Lambda)\otimes H^1(G_k,\Lambda) \arrow{d}{\alpha \otimes 1} \arrow{r}{\cup}  &H^2(G_k,\mathbb{Z})\arrow[d,equals]    \\
H^1(G_k,\Lambda^\vee)\otimes H^1(G_k,\Lambda)   \arrow{r}{ \cup}  &H^2(G_k,\mathbb{Z}) ,
\end{tikzcd}
\end{equation} 
where  $\alpha$ is induced by the inclusion of $\Lambda$ into $\Lambda^\vee$. See \cite[Proposition 2.2.2]{MR3933907} for a proof of these facts.
\end{proof}

\begin{cor}
 The order of $\mathfrak{B}_\Lambda$ is either a square or twice a square. Moreover, $\mathfrak{B}_\Lambda$ has square order if and only if 
$\dim  \mathfrak{B}_\Lambda[2]$ is even. 
\end{cor}

\begin{proof}
This is a formal consequence of the existence of a non-degenerate antisymmetric $\mathbb{Q}/\mathbb{Z}$-valued bilinear pairing on $\mathfrak{B}_\Lambda$; see  \cite[Theorem 2.4.1]{MR3933907} for a proof.
\end{proof}
 
Now consider the map $\mathfrak{B}_\Lambda\rightarrow \mathbb{Q}/\mathbb{Z}$ given by $x\mapsto (x,x)$. This is a homomorphism by antisymmetry of $(~,~)$, so by non-degeneracy there is a unique $\mathfrak{c}\in  \mathfrak{B}_\Lambda$ such that 
\begin{equation}
(x,x)=(\mathfrak{c},x) \quad \textup{ for all }x\in \mathfrak{B}_\Lambda.
\end{equation}
It follows from the arguments of \cite[Section 6]{MR1740984} that one has
\begin{equation} \label{poonen_stoll_odd_even_criterion}
\dim  \mathfrak{B}_\Lambda[2]\equiv 0~~(\textup{mod }2)\quad \quad \textup{if and only if }\quad \quad (\mathfrak{c},\mathfrak{c})=0.
\end{equation}
In \Cref{poonen_stoll_for_betts_groups} below we give an explicit description of this class $\mathfrak{c}$. The construction involves quadratic refinements of the pairings \eqref{eq:pairing_on_lambda} and \eqref{pairing_on_discriminant_group}.

\subsection{Quadratic refinements of the pairings \eqref{eq:pairing_on_lambda} and \eqref{pairing_on_discriminant_group}} \label{subsec:quad_forms_on_lattices}
We begin with some notation.

\begin{notation}
For abelian groups $A$ and $M$,  call a function $q:A\rightarrow M$ a quadratic map if the function $B_q:A\times A \rightarrow M$ defined by $B_q(a_1,a_2)=q(a_1+a_2)-q(a_1)-q(a_2)$ is bilinear. We call $q$ a quadratic form if moreover, for all $a\in A$ and $n\in \mathbb{Z}$, we have $q(na)=n^2q(a)$.  
We say that $q$ is a quadratic refinement of $B_q$. 
Denote by $\mathcal{Q}_\Lambda$ the set of $\mathbb{Z}$-valued quadratic refinements of \eqref{eq:pairing_on_lambda}, and by $\mathcal{Q}_\Phi$ the set of $\mathbb{Q}/\mathbb{Z}$-valued quadratic refinements of \eqref{pairing_on_discriminant_group}.
\end{notation}
 
Now define the subset $S\subseteq V$ as 
\begin{equation} \label{eq:set_S_quadratic_refinements}
S=\big\{v\in \Lambda^\vee~~\colon~~\left \langle \lambda,\lambda \right \rangle \equiv \left \langle \lambda, v\right \rangle~(\textup{mod }2)~~\textup{ for all }\lambda \in \Lambda\big\}.
\end{equation}
One checks using the fact that  $\lambda \mapsto \left \langle \lambda,\lambda \right \rangle ~(\textup{mod }2)$
is a homomorphism that $S$ is nonempty.
For $v\in S$ denote by $q_v: \Lambda \rightarrow \mathbb{Z}$ the quadratic map
\begin{equation} \label{quad_refinement_formula_lattice}
q_v(\lambda)=\tfrac{1}{2}\left(\left\langle \lambda,\lambda \right \rangle + \left \langle \lambda,v\right \rangle\right).
\end{equation}
Sending $v$ to $q_v$ gives a bijection from $S$ to $\mathcal{Q}_\Lambda$. Moreover, taking $\lambda \in \Lambda^\vee$ in the formula \eqref{quad_refinement_formula_lattice},  and then reducing modulo $\mathbb{Z}$, gives a quadratic refinement 
$\overline{q}_v: \Phi \rightarrow \mathbb{Q}/\mathbb{Z} $
of  the pairing \eqref{pairing_on_discriminant_group}. This is a quadratic form if and only if $v\in \Lambda$. The map $v \mapsto \overline{q}_v$ is a bijection between  $S/2\Lambda$ (the quotient of $S$ by the action of $2\Lambda$)  and $\mathcal{Q}_{\Phi}$. 

\subsection{Cohomology classes associated to quadratic refinements}

Since the pairing \eqref{eq:pairing_on_lambda}  is $G_k$-invariant,  $G_k$ acts on $\mathcal{Q}_\Lambda$. Explicitly, for $\sigma\in G_k$ and $q\in \mathcal{Q}_\Lambda$ we define $^{\sigma}q:\Lambda\rightarrow \mathbb{Z}$ by setting $^{\sigma}q(\lambda)=q(\sigma^{-1}\lambda)$.  Given $q_1, q_2\in \Lambda$ we have $q_1-q_2\in \textup{Hom}(\Lambda,\mathbb{Z})=\Lambda^\vee$; thus $\Lambda^\vee$ acts simply transitively on $\mathcal{Q}_\Lambda$. In particular, associated to $\mathcal{Q}_\Lambda$ is a class $\mathfrak{q}_\Lambda \in H^1(G_k,\Lambda^\vee)$, explicitly represented by the $1$-cocycle 
$\sigma \mapsto {}^{\sigma}q-q$
for any $q \in \mathcal{Q}_\Lambda$. We similarly have  $\mathfrak{q}_\Phi\in H^1(G_k,\Phi)$ associated to $\mathcal{Q}_\Phi$.\footnote{To be more precise, mimicking the construction of $\mathfrak{q}_\Lambda$ yields a class in $H^1(G_k,\Phi^*)$ where $\Phi^*=\textup{Hom}(\Phi,\mathbb{Q}/\mathbb{Z})$; we transport this class to $H^1(G_k,\Phi)$ via the isomorphism $\Phi \stackrel{\sim}{\rightarrow}\Phi^*$ provided by the pairing  \eqref{pairing_on_discriminant_group}.}

\begin{remark} \label{rem:explicit_coccle_lattice_forms}
The discussion in \Cref{subsec:quad_forms_on_lattices} above provides a more explicit description of the classes $\mathfrak{q}_\Lambda$ and $\mathfrak{q}_\Phi$. Let $v\in S$ and let $q_v$ (resp. $\overline{q}_v$) be the associated element of $\mathcal{Q}_\Lambda$ (resp. $\mathcal{Q}_{\Phi}$). Computing the associated cocycles one sees that $\mathfrak{q}_\Lambda$ and $\mathfrak{q}_\Phi$ are represented by the $1$-cocycles 
\[\sigma\mapsto  \tfrac{1}{2}(\sigma v-v) \in \Lambda^\vee\quad \textup{ and }\quad \sigma\mapsto  ~\tfrac{1}{2}(\sigma v-v)~~(\textup{mod }\Lambda)\] 
respectively. Note in particular that $\mathfrak{q}_\Lambda$ maps to $\mathfrak{q}_\Phi$ under the natural map $H^1(G_k,\Lambda^\vee)\rightarrow H^1(G_k,\Phi)$. 
\end{remark}

\begin{lemma} \label{poonen_stoll_for_betts_groups}
With the notation above, we have the following. 
\begin{itemize}
\item[(i)]  The element $\mathfrak{q}_\Lambda\in H^1(G_k,\Lambda^\vee)$ lies in $\mathfrak{B}_\Lambda$.
\item[(ii)]   We have $(x,x)=(\mathfrak{q}_\Lambda,x)$ for all $x\in \mathfrak{B}_\Lambda$. Thus $\mathfrak{q}_\Lambda$ is the class $\mathfrak{c}$ of \Cref{betts_group_sub_sec}.
\end{itemize}
\end{lemma}

\begin{proof}
It follows from \cite[Corollary 2.8]{MR2915483} that for all $\rho \in H^1(G_k,\Lambda)$ we have
\begin{equation} \label{poonen_rains_input}
\mathfrak{q}_\Lambda\cup \rho= \rho \cup \rho  \quad\textup{ inside } H^2(G_k,\mathbb{Z}). 
\end{equation}
Here both cup-products are induced by the pairing $\left \langle ~,~\right \rangle$. From this identity and commutativity of \eqref{commutative_betts_cup_diag} it follows that $\mathfrak{q}_\Lambda\cup \rho=0$ for all $\rho \in \ker \left(H^1(G_k,\Lambda)\stackrel{}{\rightarrow} H^1(G_k,\Lambda^\vee)\right)$. It now follows formally from  the stated properties of the pairings in \eqref{commutative_betts_cup_diag} that $\mathfrak{q}_\Lambda\in \mathfrak{B}_\Lambda$. This proves part (i), and part (ii) now follows from \eqref{poonen_rains_input} and the definition of the pairing $(~,~)$ on $\mathfrak{B}_\Lambda$.
\end{proof}

\begin{remark} \label{invariant_refinement_remark}
Since $\mathfrak{q}_\Lambda$ is a lift of $\mathfrak{q}_\Phi$ to $H^1(G_k,\Lambda^\vee)$, it follows from  \Cref{poonen_stoll_for_betts_groups} (i) that the class $\mathfrak{q}_\Phi\in H^1(G_k,\Phi)$ is trivial. In particular, the pairing \eqref{pairing_on_discriminant_group} on $\Phi$ admits a $G_k$-invariant quadratic refinement. From the discussion in \Cref{subsec:quad_forms_on_lattices}, such a quadratic refinement is necessarily of the form $\overline{q}_v$ for some $v\in S$. The $G_k$-invariance of $\overline{q}_v$ means that such a $v$ satisfies $\sigma v-v\in 2\Lambda$ for all $\sigma \in G_k$. 
\end{remark}

\subsection{The order of $\mathfrak{B}_{\Lambda}$ modulo squares}
We now give the promised criterion for determining the parity of $\dim  \mathfrak{B}_\Lambda[2]$. As usual, $F\in G_k$ denotes the Frobenius element. The set $S$ is as defined in \eqref{eq:set_S_quadratic_refinements}.  

\begin{proposition} \label{prop:betts_group_mod_2}
There exists $x\in S$ such that $\tfrac{1}{2}(Fx-x)\in \Lambda$. For any such $x$ we have
\[\dim  \mathfrak{B}_\Lambda[2]\equiv\left \langle x,\tfrac{1}{2}(Fx-x) \right \rangle ~~(\textup{mod }2).\]
\end{proposition}

\begin{proof}
For existence, take any $x\in S$ for which the associated quadratic form $\overline{q}_x$ is $G_k$-invariant (cf. \Cref{invariant_refinement_remark}).  

Now fix such an $x$ and denote by $a:G_k\rightarrow \Lambda$ the $1$-cocycle  $a(\sigma)= \frac{1}{2}(\sigma x-x)$. Its class in  $H^1(G_k,\Lambda)$ is a lift of $\mathfrak{q}_{\Lambda} \in \mathfrak{B}_{\Lambda}$ to $H^1(G_k,\Lambda)$.  Consider the commutative diagram 
\begin{equation} \label{betts_diag_1_2}
\begin{tikzcd}
H^0(G_k,V/\Lambda^\vee)\otimes H^1(G_k,\Lambda) \arrow{d}{\delta\otimes 1} \arrow{r}{\cup}  &H^1(G_k,\mathbb{Q}/\mathbb{Z})\arrow{d}{\delta}    \\
H^1(G_k,\Lambda^\vee)\otimes H^1(G_k,\Lambda)   \arrow{r}{ \cup}  &H^2(G_k,\mathbb{Z}).
\end{tikzcd}
\end{equation}
Here the coboundary maps $\delta$ arise  from the short exact sequence \eqref{basic_exact_seq} and the corresponding sequence given by tensoring \eqref{basic_exact_seq} by $\Lambda^\vee$,
 and both cup-products are induced by the pairing $\left \langle ~,~\right \rangle$. The element $\frac{1}{2}x \in V/\Lambda^\vee$ defines an element of $H^0(G_k,V/\Lambda^\vee)$ which  maps under $\delta$ to $\mathfrak{q}_\Lambda$. From commutativity of \eqref{betts_diag_1_2} and the definition of the pairing $(~,~)$ on $\mathfrak{B}_\Lambda$, we find
  \[(\mathfrak{q}_\Lambda,\mathfrak{q}_{\Lambda})=\left \langle \tfrac{1}{2}x, a(F)\right \rangle =\tfrac{1}{2}\left \langle x,\tfrac{1}{2}(Fx-x) \right \rangle \in \tfrac{1}{2}\mathbb{Z}/\mathbb{Z}.\]
The result now follows from \eqref{poonen_stoll_odd_even_criterion} and \Cref{poonen_stoll_for_betts_groups} (ii).
\end{proof}

\begin{remark} \label{main_betts_dimension_rephrase}
Identifying $\Lambda^\vee$ with $\textup{Hom}(\Lambda,\mathbb{Z})$ via the map $v \mapsto \left \langle v,- \right \rangle$ leads to the following rephrasing of \Cref{prop:betts_group_mod_2}: given $\phi \in \textup{Hom}(\Lambda,\mathbb{Z})$ such that 
$\left \langle \lambda,\lambda \right \rangle \equiv \phi(\lambda)~~(\textup{mod }2)$ for all $\lambda\in \Lambda$,
and such that $\frac{1}{2}(F \phi -\phi )=\left \langle \lambda,- \right \rangle$ for a (necessarily unique) $\lambda \in \Lambda$, we have
\[\dim  \mathfrak{B}_\Lambda[2]\equiv\phi(\lambda) ~~(\textup{mod }2).\] 
Indeed, the unique $x\in \Lambda^\vee$ for which $\phi=\left \langle x,-\right \rangle$ satisfies the conditions of \Cref{prop:betts_group_mod_2}.
\end{remark}

\section{Clusters and the group $\mathfrak{B}_{C/K}$ for semistable hyperelliptic curves} \label{sec:clusters}  
We take the notation of \Cref{min reg model ramified}. In particular, $K$ denotes a nonarchimedean local field with residue field $k$ of odd characteristic, and $C/K$ denotes a hyperelliptic curve with semistable reduction. We henceforth fix a Weierstrass equation $y^2=f(x)$ for $C$, where  $f(x)\in K[x]$ is a squarefree polynomial of degree $2g+1$ or $2g+2$ for $g\geq 2$ the genus of $C$. We denote by $\mathcal{R}$ the set of roots of $f(x)$ in $\K^s$ and denote by $c_f$ the leading coefficient of $f(x)$. Thus
\[f(x)=c_f\prod_{r\in \mathcal{R}}(x-r).\]

\subsection{Clusters} \label{subsec_clusters_defs}  
We now recall several results from \cite{DDMM18}, which provides a framework for studying invariants of hyperelliptic curves over local fields of odd residue characteristic. We refer to that work for more details  (cf. also \cite{hyperuser}). The central object is that of a \textit{cluster}. In what follows we denote by $v:\bar{K}\rightarrow \mathbb{Q}\cup \{\infty\}$ the extension to $\bar{K}$ of the normalised valuation on $K$.

\begin{defi}
A \textit{cluster} is a non-empty subset $\c\subset\mathcal{R}$ of the form $\c = D \cap \mathcal{R}$
for some disc $D=\{x\!\in\! \bar{K} \mid v(x-z)\!\geq\! d\}$
where  $z\in \bar{K}$ and $d\in \mathbb{Q}$. 
If $|\s|>1$ then $\s$ is said to be \emph{proper} and its \emph{depth} $d_\s$ is defined as 
$$
  d_\s = \min_{r,r' \in \mathfrak{s}} v(r-r').
$$
We call any element $z_\s$ of the minimal disc cutting out a proper cluster $\s$ a \textit{centre} for $\s$. 
\end{defi}
 
We summarize some terminology for clusters. 

\begin{defi} \label{def:cluster_terminology}
Given clusters $\s_1\neq\s_2$ with $\s_1$ a maximal subcluster of $\s_2$, we say that $\s_1$  is a \emph{child} of $\s_2$, denote this $\s_1<\s_2$, and refer to $\s_2$ as the \emph{parent} of $\s_1$. Any cluster $\s\neq \mathcal{R}$ has a unique parent $P(\s)$. We define the \emph{relative depth} of a proper cluster $\s\neq \mathcal{R}$ as
\[\delta_\s:=d_{\s}-d_{P(\s)}\geq 0.\]
We call a cluster \emph{even} (resp. \emph{odd}) if it contains an even (resp. odd) number of roots, and call it \emph{\ub}~if it is even and all its children are even also. We call a cluster $\s$ {\em principal} if $|\s|\ge 3$, save when either $\s=\cR$ is even and has exactly two children, or when $\s$ has a child of size $2g$. A cluster of size $2$ is called a \emph{twin}, and  a non-\ub\ cluster that has a child of size~$2g$ is called a \emph{cotwin}. For a principal cluster $\s$, its \textit{genus} $g(\s)$ is defined as \[g(\s)=\big\lfloor \frac{1}{2}(\#\{\textup{odd children of } \s\}-1)\big\rfloor.\] Finally, for clusters  $\c_1$, $\c_2$, we write $\c_1\wedge \c_2$ for the smallest cluster containing both $\c_1$ and $\c_2$. 
\end{defi}

\begin{example} \label{example_of_a_cluster_picture}  
Take $C$ to be the hyperelliptic cuve 
\[C/\mathbb{Q}_3:y^2=(x^2+3)((x-i)^2-3^2)((x+i)^2-3^2)\]
 considered in \Cref{unram_genus_2_ub_example}, where $i$ is a square root of $-1$.  The proper clusters are 
\[\mathcal{R}=\{\pm i\sqrt{3}, i\pm 3, -i \pm 3\},~\mathfrak{t}_1=\{\pm i \sqrt{3}\},~\mathfrak{t}_2=\{i \pm 3\},~\mathfrak{t}_3=\{-i \pm 3\}.\]
The unique principal cluster is $\mathcal{R}$ and it is \ub, has depth $0$ and genus $0$. There are $3$ twins: $\mathfrak{t}_1$, $\mathfrak{t}_2$ and $\mathfrak{t}_3$, and we have $\delta_{\mathfrak{t}_1}=1/2$, $\delta_{\mathfrak{t}_2}=\delta_{\mathfrak{t}_3}=1$. We display this information pictorially as shown:
\begin{center}
		\clusterpicture            
  \Root[A] {} {first} {r1};
  \Root[A] {} {r1} {r2};
  \Root[A] {6} {r2} {r3};
  \Root[A] {} {r3} {r4};
  \Root[A] {6} {r4} {r5};
  \Root[A] {} {r5} {r6};
  \ClusterLDName c1[][\tfrac{1}{2}][ ] = (r1)(r2);
    \ClusterLDName c2[][1][ ] = (r3)(r4);
      \ClusterLDName c3[][1][ ] = (r5)(r6);
  \ClusterLDName c4[][0][] = (c1)(c3)(c3);
\endclusterpicture 
\end{center}
Here we draw roots as 
\smash{\raise4pt\hbox{\clusterpicture\Root[A]{}{first}{r1};\endclusterpicture}}
and draw ovals around roots to represent a proper cluster. The subscript on the outer cluster is its depth, and on all other clusters it is the relative depth. We refer to this diagram as the \textit{cluster picture} of $C$. 

We remark that the description of the minimal regular model of $C$ given previously in \Cref{unram_genus_2_ub_example} now follows immediately from \cite[Theorems 1.11 and 8.6]{DDMM18}.
\end{example}

For the rest of this section we make the following assumption, which will lead to several simplifications in the results of \cite{DDMM18}. 

\begin{assumption} \label{assump:no_cotwin}
We assume that $|\mathcal{R}|=2g+2$ and  that there are no clusters of size $2g$ or $2g+1$.
\end{assumption}

\begin{remark} \label{can_reduce_to_no_cotwins_remark}
 By \cite[Theorem 15.2]{DDMM18}, any semistable hyperelliptic curve over $K$ is isomorphic to a curve satisfying \Cref{assump:no_cotwin} over any suitably large odd degree unramified extension of $K$ (the key point being that if the residue field of $K$ is sufficiently large then one can make a change of variables to force \Cref{assump:no_cotwin} to be satisfied). 
 \end{remark}

We now summarize certain results from \cite{DDMM18}, using   \Cref{assump:no_cotwin} to simplify several statements. First, the fact that $C$ is semistable forces several constraints on the possible clusters and their depths. Specifically we have:

\begin{theorem}[\cite{DDMM18} Theorem 1.8] \label{semi criterion}
Semistability of $C$ is equivalent to the following three conditions:
\begin{enumerate}
\item
the extension $K(\cR)/K$ has ramification degree at most 2,
\item
every proper cluster is invariant under the action of the inertia group of $K$,
\item
every principal cluster $\s$ has $d_\c \in\Z$ and 
$\nu_\s \in 2\Z$,
\end{enumerate}
where $\nu_\s$ is the quantity
\begin{equation} \label{deifnition_of_quantity_nu}
\nu_\s= v(c_f)+|\s|d_\s+\sum_{r\notin\c} d_{\{r\}\wedge \s}.
\end{equation}
\end{theorem}

Note that by part (1) of \Cref{semi criterion}, each inertia orbit of roots has size at most $2$ (i.e. the irreducible factors of $f(x)$ over $K^\textup{nr}$ are linear or quadratic), and every cluster $\s$ has $d_\s\in \frac{1}{2}\mathbb{Z}$. The set of proper cluster $\mathfrak{s}$ with   $d_\s\notin \mathbb{Z}$ will be of particular importance.

\begin{notation} \label{notat_T_set}
Let $T$ denote the set of proper clusters $\s$ with $d_\s\notin \mathbb{Z}$. 
\end{notation}

\begin{lemma} \label{integer relative depths}
If $r\neq r'$ are inertia conjugate elements of $\mathcal{R}$, then $\s=\{r,r'\}$ is a cluster with $d_\s\notin \mathbb{Z}$. Moreover, every proper cluster $\s$ with $d_\s\notin \mathbb{Z}$ takes this form. 
\end{lemma} 

\begin{proof}
If $r$ and $r'$ are inertia conjugate roots then $v(r-r')\in 1/2+\mathbb{Z}$ (cf. \cite[Lemma C.2]{DDMM18}), so   the minimal cluster containing both $r$ and $r'$ has non-integer depth. In light of \Cref{assump:no_cotwin}, it follows from  \cite[Lemma 4.2]{DDMM18} that $\{r,r'\}$ is a cluster. Moreover, \cite[Lemma 4.2]{DDMM18} shows further  that any proper cluster $\s$ with $d_\s\notin \mathbb{Z}$ takes this form.
\end{proof}

A consequence of \Cref{integer relative depths} is that $T$ is naturally in bijection with the set of inertia-conjugate pairs of roots of $f(x)$.

\subsection{Signs associated to clusters}

By \Cref{semi criterion}(2), the assumption that $C$ is semistable means that $\textup{Gal}(K^{\textup{nr}}/K)=G_k$ acts on the set of proper clusters. We will augment this action by adding in certain signs associated to even clusters (cf. \cite[Definition 1.12]{DDMM18}).

\begin{notation} \label{notat:s_star_notation}
 For a cluster $\s$,  we write $\s^*$ for the smallest cluster $\s^*\supseteq\s$ whose parent is not \ub, and set $\s^*=\cR$ if no such cluster exists. 
\end{notation}

\begin{defi} \label{definition:epsilon_signs}
For even clusters $\s$ fix a choice of $\theta_\c = \sqrt{c_f\prod\nolimits_{r \notin \c} (z_\c-r)}$, where $z_s$ is any centre for $\s$.
Still assuming $\c$ is  even,  define 
$\epsilon_\s:G_K\to \{\pm 1\}$ by
$$
  \epsilon_{\c}(\sigma) \equiv \frac{\sigma(\theta_{\c^*})}{\theta_{(\sigma\s)^*}} \mod \mathfrak{m}.
$$
Here $\mathfrak{m}$ denotes the maximal ideal of the ring of integers of $\bar{K}$, so that `$\textup{mod}~ \mathfrak{m}$' denotes reduction to the residue field $\bar{k}$.
\end{defi}

\begin{remark} \label{signs remark 1}
If $\s\neq \mathcal{R}$ is an even cluster, or $\s=\mathcal{R}$ is \ub, then 
\[v\big(c_f\prod_{r\notin \s}(z_\s-r)\big)=\nu_\s-|\s|d_\s\]
is an even integer
(here $\nu_\s$ is as defined  in the statement of \Cref{semi criterion}).  Indeed, by \cite[Lemma C.5]{DDMM18} we have $\nu_\s-|\s|d_\s=\nu_{P(\s)}-|\s|d_{P(\s)}$, and by  Lemmas 4.2 and 4.9 of op. cit.  we see that one of $\s$ or $P(\s)$ must have both integral depth and even $\nu$. In particular, it follows that $\theta_\s\in K^\textup{nr}$. Thus  $\epsilon_\s$ descends to a function $G_k\rightarrow \{\pm 1\}$.
\end{remark}
 
\begin{remark} \label{signs remark 2}
As explained in \cite[Remark 1.14]{DDMM18}, although the function  $\epsilon_\s$ depends on the choice of $\theta_\s$, the restriction of $\epsilon_\s$ to the stabiliser of $\s$ does not. In fact, if $K_\s$ denotes the fixed field of $\bar{K}$ by the stabiliser of $\s$, then $K_\s$ is a finite unramified extension of $K$, $\theta_{\s^*}^2\in K_\s$, and $\epsilon_\s$ restricted to the stabiliser of $\s$ is the quadratic character associated to the extension $K_\s(\theta_{\s^*})/K_\s$. 
\end{remark}

\begin{example} \label{epsilon_as_in_example}
Let $C$ be as in \Cref{example_of_a_cluster_picture} and let $\mathfrak{s}$ be any of the $4$ even clusters. Then we have $\mathfrak{s}^*=\mathcal{R}$. We can take $z_\mathcal{R}=0$ and $\theta_{\mathcal{R}}=1$. Then $\epsilon_\mathfrak{s}(\sigma)=1$ for all $\sigma$.  

We remark that the description of the Frobenius action on the special fibre of the minimal regular model of $C$ detailed  previously in \Cref{unram_genus_2_ub_example} can be read off from this data, coupled with the Frobenius action on the set of proper clusters; see \cite[Theorems 8.6]{DDMM18}.  
\end{example}

\subsection{Description of the lattice}
We retain the notation from \Cref{min reg model}. In particular, $\Upsilon_C$ denotes the dual graph of the (geometric special fibre of the) minimal proper regular model of $C$. Here we recall from \cite{DDMM18} a description of the $\mathbb{Z}[G_k]$-module $H_1(\Upsilon_{C},\Z)$ along with its associated pairing.  It will be convenient to first define an auxiliary lattice $\Pi$ which is closely related to $H_1(\Upsilon_{C},\Z)$, but which is simpler to describe.

\begin{defi}
Let $\mathrm{A}$ be the set of even non-\ub~clusters excluding $\mathcal{R}$. Define 
\[ \Pi=\bigoplus_{\mathfrak{s}\in \mathrm{A}}\mathbb{Z}\ell_\mathfrak{s}.\] 
Further, let $\mathrm{B}$ be the subset of $\mathrm{A}$ consisting of clusters $\s$ with $\s^*=\mathcal{R}$.   We  endow   $\Pi$ with the symmetric  pairing \[\left \langle ~,~\right \rangle:\Pi\times \Pi\longrightarrow \mathbb{Z}\]
given by
\begin{equation} \label{lattice_pairing}
\langle \ell_{\c_1},\ell_{\c_2} \rangle=\begin{cases}
        0&   \c^*_1 \neq \c^*_2, \\
        2(d_{\c_1\wedge\c_2}-d_{P(\c^*_1)})& \c_1,\c_2\notin \mathrm{B},~~\c^*_1 =\c^*_2, \\
        2(d_{\c_1\wedge\c_2}-d_{\mathcal{R}})&  \c_1,\c_2\in \mathrm{B}. 
        \end{cases}
\end{equation}
 We further endow  $\Pi$ with the $G_k$-action given by $\sigma \cdot \ell_\s=\epsilon_\s(\sigma)\ell_{\sigma \s}$. Note that the pairing  $\left \langle ~,~\right \rangle$ is invariant for this action. 
\end{defi}

It will be useful to note that the pairing on $\Pi/2\Pi$ induced from that on $\Pi$ has a very simple form.

\begin{lemma} \label{mod 2 pairing}
For all clusters $\s_1,\s_2\in \mathrm{A}$ we have
\[\left \langle \ell_{\s_1},\ell_{\s_2}\right \rangle\equiv \begin{cases}1 ~~(\textup{mod }2)~~&~~\s_1=\s_2\in T,\\ 0~~(\textup{mod }2)~~&~~\textup{else}. \end{cases}\]
\end{lemma}

\begin{proof}
Combine \Cref{integer relative depths} with the formula \eqref{lattice_pairing}.
\end{proof}

 \begin{defi} \label{label_lambda_def_concrete}
Define the lattice
\[ \Lambda= \Bigl\{ \sum_{\c\in \mathrm{A}} a_{\c}\ell_{\c} \in \Pi \Bigm|~~\sum_{\c\in \mathrm{B}} a_{\c}=0\Bigr\}\]
along with the pairing and $G_k$-action induced from that on $\Pi$.
 \end{defi}

\begin{theorem}[\cite{DDMM18} Theorem 1.14] \label{lattice comparison}
We have  $\Lambda\cong H_1(\Upsilon_{C},\Z)$ as $\mathbb{Z}[G_k]$-modules equipped with a pairing.
\end{theorem}

\subsection{The group $\mathfrak{B}_{C/K}$ for semistable hyperelliptic curves} \label{sec:hyp_betts_group_calculation}

Let $\mathfrak{B}_{C/K}$ be the group defined in \eqref{eq:betts_group_def_C}. Let $\Lambda$ be as in \Cref{label_lambda_def_concrete} above, and let $\mathfrak{B}_\Lambda$ be the associated group defined in \Cref{betts_group_defi_section}. By \Cref{lattice comparison} we have $\mathfrak{B}_{C/K} \cong \mathfrak{B}_\Lambda$. Recall from \Cref{notat_T_set} the definition of the set $T$. We denote by $F\in G_k$ the Frobenius element, and for a cluster $\mathfrak{s}$ we denote by $\textup{Orb}_\s$ its $G_k$-orbit.

 \begin{proposition} \label{h1 pi computation}
Let $N$ be the number of $G_k$-orbits $O\in T/G_k$ with $\prod_{\mathfrak{s}\in O}\epsilon_\mathfrak{s}(F)=-1$.   

Then we have 
\[\dim \mathfrak{B}_{C/K}[2] \equiv\begin{cases}N+1~~(\textup{mod }2)~~&~~\epsilon_{\mathcal{R}}(F)=-1,~~g \textup{ even, } \textup{all }\s\in \mathrm{B}\setminus T\textup{ have }\mathrm{|\textup{Orb}_\s|} \textup{ even,}\\ N~~(\textup{mod }2)~~&~~\textup{otherwise.} \end{cases}\]
\end{proposition}

We begin with the following lemma which will be needed during the proof. 

\begin{lemma} \label{genus congruence}
Suppose that $\mathrm{B}\neq \emptyset$ and that all $\s\in \mathrm{B}\setminus T$ have $|\textup{Orb}_\s|$ even. Then
\[ |\mathrm{B}\cap T|\equiv g-1~~(\textup{mod }2).\]
\end{lemma}

\begin{proof}
 Note that the assumption that $B$ is non-empty means that $\mathcal{R}$ is \ub. Let $\s\neq \mathcal{R}$ be a proper cluster with $\s\notin \mathrm{B}\cap T$. We  claim that $|\textup{Orb}_\s|$ is even. Indeed, the assumptions on $\s$ mean in particular that $\s$ is contained in a child $\mathfrak{s}'$ of $\mathcal{R}$. Clearly  $\s'$ cannot be in $T$. Thus $\mathfrak{s}'\in B\setminus T$, hence $|\textup{Orb}_{\s'}|$ is even by assumption. Since $\s\subseteq \s'$ it follows that $|\textup{Orb}_\s|$ is even also, proving the claim. Next, combining \Cref{lattice comparison} and \cite[Theorem 1.10]{DDMM18} with \cite[Lemma 10.3.18]{MR1917232} gives
\begin{equation}\label{genus formula}
g=\textup{rk}\Lambda+\sum_{\s~\textup{principal}}g(\s).
\end{equation}
The assumption that $\mathrm{B}\neq \emptyset$ means that either $\mathcal{R}$ is non-principal or $g(\mathcal{R})=0$. Thus each principal cluster of positive genus has an even sized $G_k$-orbit, so the second term on the right hand side of \Cref{genus formula} is an even integer. We therefore have
\[ |\mathrm{A}|-1= \textup{rk}\Lambda\equiv g~~(\textup{mod }2).\]
By the initial claim, every cluster $\s \in\mathrm{A}$ has $|\textup{Orb}_\s|$ even, save possibly for those clusters in  $\mathrm{B}\cap T$. Thus $|\mathrm{A}|\equiv ~~|\mathrm{B}\cap T|~~(\textup{mod }2)$
and the result follows.
\end{proof}

\begin{proof}[Proof of \Cref{h1 pi computation}. ]
We will deduce the result from \Cref{prop:betts_group_mod_2}. 

 \textbf{Case 1: either $\mathrm{B}=\emptyset$, or $\mathrm{B}\neq \emptyset$ and $\epsilon_\mathcal{R}(F)=1$.}
Consider the element $t=\sum_{\mathfrak{s}\in T}\ell_\mathfrak{s}$ of $\Pi$.
 By \Cref{mod 2 pairing}, for all $\lambda \in \Lambda$ we have 
\[\left \langle \lambda,\lambda\right \rangle \equiv \left \langle \lambda,t \right \rangle \textup{ (mod }2).\] 
 Further, we have 
 \[ Ft-t=\sum_{\mathfrak{s}\in T}(\epsilon_{F^{-1}\mathfrak{s}}(F)-1)\ell_\mathfrak{s}.\]
 It follows that $Ft-t \in 2\Lambda$ (note that if $\mathrm{B}\neq \emptyset$ then for any $\mathfrak{s}\in T\cap \mathrm{B}$ we have $\epsilon_{F^{-1}\mathfrak{s}}(F)=\epsilon_\mathcal{R}(F)=1$). Taking $x=t$ in  \Cref{prop:betts_group_mod_2} then gives 
 \begin{eqnarray*}
 \dim \mathfrak{B}_{C/K}[2] & \equiv &\Big\langle  \sum_{\mathfrak{s}\in T}\ell_\mathfrak{s}~,~\sum_{\mathfrak{s}\in T}\tfrac{1}{2}(\epsilon_{F^{-1}\mathfrak{s}}(F)-1)\ell_\mathfrak{s} \Big \rangle~~\quad(\textup{mod }2)\\ 
 &\equiv & \#\{\mathfrak{s}\in T~~\colon~~\epsilon_\mathfrak{s}(F)=-1\}~~\quad(\textup{mod }2),
 \end{eqnarray*}
 the last congruence following from \Cref{mod 2 pairing}.\footnote{When $|T\cap B|$ is odd the element $t$ of $\Pi$ is not in $\Lambda$, so    \Cref{prop:betts_group_mod_2} does not naively apply. However, in this case we can take $\phi=\left \langle t,-\right \rangle \in \textup{Hom}(\Lambda,\mathbb{Z})$ in  \Cref{main_betts_dimension_rephrase} to see that the conclusion concerning $\dim \mathfrak{B}_{C/K}[2]$ remains valid.} Thus we have the result in this case.
 
\textbf{Case 2: $\mathrm{B}\neq\emptyset$, $\epsilon_\mathcal{R}(F)=-1$, $|\mathrm{B}\cap T|$ even.} Write $\mathrm{B}\cap T=\{\mathfrak{s}_1,...,\mathfrak{s}_{2k}\}$. This time we set $t=\sum_{\mathfrak{s}\in T\setminus \mathrm{B}}\ell_\mathfrak{s}+\sum_{i=1}^{2k}(-1)^i\ell_{\mathfrak{s}_i}$, noting that $t\in \Lambda$. As with the previous case, taking $x=t$ in  \Cref{prop:betts_group_mod_2} gives the result (note that by \Cref{genus congruence} we are trying to show that $\dim \mathfrak{B}_{C/K}[2] \equiv N$ (mod $2$) in this case).

\textbf{Case 3: $\mathrm{B}\neq \emptyset$, $\epsilon_\mathcal{R}(F)=-1$, $|\mathrm{B}\cap T|$ odd, $|\mathrm{Orb}_\mathfrak{s}|$ odd for some $\mathfrak{s}\in \mathrm{B}\setminus T$.} Choose some $\mathfrak{s}_1\in \mathrm{B}\setminus T$ with $m_1=|\textup{Orb}_{\mathfrak{s}_1}|$ odd, and write $m_2=|\mathrm{B}\cap T|$.  This time, take 
\[t=\sum_{\mathfrak{s}\in T\setminus \mathrm{B}}\ell_\mathfrak{s}+m_1\cdot\sum_{\mathfrak{s}\in T\cap \mathrm{B}}\ell_\mathfrak{s}-m_2\cdot\sum_{\mathfrak{s}\in \textup{Orb}_{\mathfrak{s}_1}}\ell_{\mathfrak{s}_1},\]
which lies in $\Lambda$ by construction. Again, we conclude by taking $x=t$ in \Cref{prop:betts_group_mod_2}.

\textbf{Case 4: $\mathrm{B}\neq \emptyset$, $\epsilon_\mathcal{R}(F)=-1$, $|\mathrm{B}\cap T|$ odd, all $\s\in \mathrm{B}\setminus T$ have $|\mathrm{Orb}_\s|$ even.} Note that in this case we have $g$ even by \Cref{genus congruence}, so we want to show that $\dim \mathfrak{B}_{C/K}[2] \equiv N+1$ (mod $2$). Since each $\s\in \mathrm{B}\setminus T$ has $|\mathrm{Orb}_\s|$ even, we can partition $\mathrm{B}\setminus T$ into two disjoint sets $\mathrm{B}_0$ and $\mathrm{B}_1$ with $F(\mathrm{B}_0)=\mathrm{B}_1$. For $\mathfrak{s}\in \mathrm{A}$, write $\ell_\mathfrak{s}^\vee$ for the element of $\textup{Hom}(\Lambda,\mathbb{Z})$ sending $\ell_\mathfrak{s}$ to $1$, and sending $\ell_{\mathfrak{s}'}$ to $0$ for each $\mathfrak{s}'\neq \mathfrak{s}$. Consider the element 
\[\phi= \Big\langle \sum_{\mathfrak{s}\in T\setminus \mathrm{B}}\ell_\mathfrak{s},-\Big\rangle +\sum_{\mathfrak{s}\in \mathrm{B}_0}\ell_\mathfrak{s}^\vee - \sum_{\mathfrak{s}\in \mathrm{B}_1}\ell_\mathfrak{s}^\vee ~~\in \textup{Hom}(\Lambda,\mathbb{Z}).\]
Then  we have $\left \langle \lambda,\lambda\right \rangle \equiv  \phi(\lambda) \textup{ (mod }2)$, as follows from \Cref{mod 2 pairing} upon noting that $\Lambda$ is by definition the collection of elements $\sum_{\mathfrak{s}}n_\mathfrak{s}\ell_\mathfrak{s} \in \Pi$ for which $\sum_{\mathfrak{s}\in \mathrm{B}}n_\mathfrak{s}=0$. Moreover, we have
\[F\phi -\phi=\Big \langle \sum_{\mathfrak{s}\in T\setminus \mathrm{B}}(\epsilon_{F^{-1}\mathfrak{s}}(F)-1)\ell_\mathfrak{s},-\Big \rangle.\]
Since $\sum_{\mathfrak{s}\in T\setminus B}(\epsilon_{F^{-1}\mathfrak{s}}(F)-1)\ell_\mathfrak{s}\in 2\Lambda$ we can apply \Cref{main_betts_dimension_rephrase} to $\phi$, giving
\begin{eqnarray*}
 \dim \mathfrak{B}_{C/K}[2] & \equiv &\phi\Bigg(\frac{1}{2}\sum_{\mathfrak{s}\in T\setminus \mathrm{B}}(\epsilon_{F^{-1}\mathfrak{s}}(F)-1)\ell_\mathfrak{s}\Bigg)\quad(\textup{mod }2)\\ 
 &\equiv & \#\{\mathfrak{s}\in T\setminus \mathrm{B}~~\colon~~\epsilon_\mathfrak{s}(F)=-1\}~~\quad(\textup{mod }2).
 \end{eqnarray*}
 This latter quantity is congruent to $N+1$ modulo $2$. Indeed, every element of  $\mathrm{B}\cap T$ has $\epsilon_\mathfrak{s}(F)=-1$ by assumption. Since moreover $|\mathrm{B}\cap T|$  is assumed odd, the claimed congruence follows. 
\end{proof}

\begin{remark}
Instead of appealing to \Cref{prop:betts_group_mod_2}, an alternative approach to proving \Cref{h1 pi computation} might be to draw on work of Betts \cite[Section 3]{MR4330930} (see also \cite[Section 10]{hyperuser}), which gives a description in terms of clusters for the individual Tamagawa numbers $c(J/L)$ and $c(J/K)$. From this   one might then hope to prove the result by computing explicitly the quotient $c(J/L)/c(J/K)$ and appealing to \eqref{betts_dokchitser_input}. However, the description of   Tamagawa numbers given in that work becomes sufficiently complicated in the presence of \ub~ clusters that we have elected to avoid this approach.
\end{remark}

 Before stating the final result of the section we require one further piece of notation.
 
 \begin{notation} \label{notat:etaC}
Define $\kappa(C)\in \{0,1\}$ as follows. We set $\kappa(C)=1$ if  $\mathcal{R}=\s_1\sqcup \s_2$ is a disjoint union of $2$ odd $G_k$-conjugate clusters $\mathfrak{s}_1$ and $\mathfrak{s}_2$ with both $\delta_{s_1}$ and $\delta_{s_2}$ odd (note in particular that this forces $C$ to have even genus). We set $\kappa(C)=0$ otherwise. 
 \end{notation}
 
\begin{cor} \label{mod 2 betts group cor}
We have
\[\dim \mathfrak{B}_{C/K}[2]+\epsilon(C/K)\equiv \kappa(C)+\#\Big\{G_k\textup{-orbits }O\subseteq T\textup{ with }\prod_{\mathfrak{t}\in O}\epsilon_\mathfrak{t}(F)=-1\Big\}~~\textup{(mod }2).\]
\end{cor}

\begin{proof}
Combine \Cref{h1 pi computation} with  \cite[Theorem 1.23]{DDMM18} (the cited result gives an explicit description of   deficiency in terms of clusters; to apply it recall that we have a running assumption that $\mathcal{R}$ has no cotwins).
\end{proof}

\section{Ramified quadratic twists of semistable hyperelliptic curves} \label{sec:ram_twist_hyp_curve}

We retain the notation and setup of the previous section. Thus $K$ is a nonarchimedean local field of odd residue characteristic, and $C/K:y^2=f(x)$ is a semistable hyperelliptic curve over $K$. We continue to impose \Cref{assump:no_cotwin}, so that $f(x)$ has even degree and the set $\mathcal{R}$ of roots of $f(x)$ in $\bar{K}$ has no cotwins in the sense of \Cref{def:cluster_terminology}.  Let $L/K$ be a ramified quadratic extension of $K$, and write $L=K(\sqrt{\pi})$ for some uniformiser $\pi\in K$. 

\subsection{The minimal regular model of $C^L$} \label{min_reg_model_twist_sec}
We now give an explicit `cluster picture' description of the special fibre of the minimal regular model of the quadratic twist $C^L:y^2=\pi f(x)$ of $C$ by $L$.  As we shall see, even though $C^L$ is no longer semistable over $K$, one can still give a simple description of its minimal regular model  in terms of clusters. To avoid confusion when comparing invariants of $C$ with invariants of $C^L$ later we will make the following convention.

\begin{convention} \label{where_are_clusters_convention}
Unless stated otherwise, in this section we will view all clusters as being associated to the polynomial $f(x)$ defining $C$, as opposed to the polynomial $\pi f(x)$ defining $C^L$. Since both the clusters themselves and the associated functions $d_\mathfrak{s}$ and $\delta_{s}$ (depth and relative depth) are functions purely of the set of roots $\mathcal{R}$, they are unchanged under replacing $f(x)$ by $\pi f(x)$. Thus the distinction here is irrelevant. However, for a cluster $\mathfrak{s}$, the functions $\nu_\mathfrak{s}$ and $\epsilon_\mathfrak{s}$ (see \eqref{deifnition_of_quantity_nu}, \Cref{definition:epsilon_signs})  are defined with reference to the leading coefficient of the polynomial in question, hence may change upon replacing $f(x)$ by $\pi f(x)$. For example, for a proper cluster $\mathfrak{s}$, $\nu_\mathfrak{s}$ is one larger when $\mathfrak{s}$ is viewed as a cluster for $C^L$ than when it is viewed as a cluster for $C$.  
\end{convention}

In several statements below we will need to distinguish the following special case. 

\begin{notation}  
We say that $\mathcal{R}$ is \textit{atypical} if $\mathcal{R}=\mathfrak{s}_1\sqcup \mathfrak{s}_2$ is a disjoint union of $2$ odd proper clusters $\mathfrak{s}_1$ and $\mathfrak{s}_2$, with both $\delta_{\mathfrak{s}_1}$ and $\delta_{\mathfrak{s}_2}$ odd. 
\end{notation}

The description of the special fibre of the minimal regular model of $C^L$ that we present below follows from work of Faraggi--Nowell \cite{MR4201122}, which more generally gives an explicit description of the special fibre of the minimal regular SNC model for hyperelliptic curves with tame reduction (that is, attaining semistable reduction after a tamely ramified extension of the base field). As is apparent from the statement of \Cref{dual_graph_semistable_twist_prop}  below, their description simplifies significantly for quadratic twists of semistable hyperelliptic curves.

\begin{remark}
For an alternative, but related, approach to constructing regular models of hyperelliptic curves over nonarchimedean local fields of odd residue characteristic, see the works of Srinivasan \cite{Srin2019} and Obus--Srinivasan \cite{ObSrin2019}. 
\end{remark}

In what follows we denote by $\mathcal{X}$ the minimal regular model of $C^L$ over $\mathcal{O}_K$, and denote by $\mathcal{X}_{\bar{k}}$ its special fibre, base-changed to $\bar{k}$.

\begin{notation}
In describing $\mathcal{X}_{\bar{k}}$ we will use the following terminology; see \cite[Definition 3.1]{MR4201122} for more details. By a chain of $n$ rational curves of multiplicity $d$, $n\geq 0$, $d\geq 1$, we mean a collection of irreducible components $\Gamma_1$, ..., $\Gamma_n$ of $\mathcal{X}_{\bar{k}}$, each isomorphic to $\mathbb{P}^1_{\bar{k}}$, such that $\Gamma_i$ intersects $\Gamma_{i+1}$   transversally for each $i$, and such that each $\Gamma_i$ has multiplicity $d$ in  $\mathcal{X}_{\bar{k}}$. We depict this situation below. By a crossed tail, we mean a chain of rational curves $\Gamma_1$, ..., $\Gamma_n$, along with $2$ additional irreducible components, the `crosses', both isomorphic to $\mathbb{P}^1_{\bar{k}}$ and intersecting $\Gamma_n$ transversally. Again, this situation is depicted below. In all the crossed tails we consider,  each $\Gamma_i$ has multiplicity $2$, whilst the crosses have multiplicity $1$.
\vspace{-10pt}
\begin{center}
\begin{figure}[!htbp]
\begin{minipage}{.65\textwidth}

\begin{center}
\phantom{spacespacesp}\begin{tikzpicture} [scale=0.94]
\draw [line width=1.2pt] (-8.9,-0.8)-- node[above, font=\small]{$\Gamma_1\qquad $} node[below, font=\small]{$d\quad$ } ++(1.5,0.9); 
\draw [line width=1.2pt] (-8,0.1)-- node[above]{} node[below, font=\small]{$d\quad $ } ++(1.5,- 0.9); 
\draw [line width=1.2pt] (-7.1,-0.8)-- node[above]{$d$} node[below, font=\small]{ } ++(1.5,0.9); 
\draw (-4.9,-0.1) node{$\cdots$};
\draw [line width=1.2pt] (-4,0.1)-- node[above]{$d$} node[below, font=\small]{ } ++(1.5,- 0.9); 
\draw [line width=1.2pt] (-3.1,-0.8)-- node[above, font=\small]{$\Gamma_n$} node[below, font=\small]{$\qquad d$ } ++(1.5,0.9); 
\draw (-5,-2) node{$\textup{Chain of }n\textup{ rational curves of multiplicity }d$};
\end{tikzpicture}
\end{center}
\end{minipage}%
\begin{minipage}{.35\textwidth}
\begin{center}
\phantom{spa}\begin{tikzpicture}[scale=0.5]
\draw  [line width=1.2pt] (2.7,2.3)-- node[above]{$  \Gamma_1~$} node[below, font=\small]{ } ++(-1.5,-2); 
\draw  [line width=1.2pt] (1.2,1)-- ++(1.5,-2); 
\draw  [line width=1.2pt] (1.2,-1.7)-- node[above, font=\small]{$\qquad \Gamma_n \quad $ } node[below, font=\small]{ } ++(1.5,-2); 
\draw (2.4,1.2) node{$2$};
\draw (1.8,-0.5) node{$2$};
\draw (1.3,-2.5) node{$2$};
\draw (2.05,-1.15) node{$\iddots$};
\draw (1.9,-4.6) node{$\textup{Crossed tail}$};
\draw [line width=1.2pt] (1.1,-3.7)-- node[above]{} node[below, font=\small]{ } ++(1.5,0.9); 
\draw [line width=1.2pt] (1.3,-4.2)-- node[above]{} node[below, font=\small]{ } ++(1.5,0.9); 
\end{tikzpicture}
\end{center}
\end{minipage}
\end{figure}
\end{center}
\end{notation}
\begin{proposition} \label{dual_graph_semistable_twist_prop} 
All irreducible components of $\mathcal{X}_{\bar{k}}$ intersect transversally, and no three components intersect at a point. Moreover:
\begin{itemize}
\item every principal cluster $\mathfrak{s}$ for $C$ contributes to $\mathcal{X}_{\bar{k}}$ a single component $\Gamma_\mathfrak{s}$ of genus $0$ and multiplicity $2$,
\item components corresponding to principal clusters $\mathfrak{s}'<\mathfrak{s}$ are linked by:  
\begin{itemize}
\item  a chain of $\frac{1}{2}\delta_{\mathfrak{s}'}$ rational curves of multiplicity $1$ if $\mathfrak{s}'$ is odd,
 \item  a chain of $(2\delta_{\mathfrak{s}'}-1)$ rational curves of multiplicity $2$ if $\mathfrak{s}'$ is even,
\end{itemize}
\item for $\mathfrak{s}$ principal, each twin  $\mathfrak{t}<\mathfrak{s}$ contributes   a crossed tail $T_\mathfrak{t}$ whose first component intersects $\Gamma_\mathfrak{s}$, and which consists of $2\delta_\mathfrak{t}$ rational curves of multiplicity $2$, with the crosses having multiplicity $1$,
\item if $\mathcal{R}=\mathfrak{s}_1\sqcup \mathfrak{s}_2$ and both $\mathfrak{s}_1$ and $\mathfrak{s}_2$ are odd, then  $\Gamma_{\mathfrak{s}_1}$ and $\Gamma_{\mathfrak{s}_2}$ are linked by a chain of $\frac{1}{2}(\delta_{\mathfrak{s}_1}+\delta_{\mathfrak{s}_2})$ rational curves of multiplicity $1$,
\item if $\mathcal{R}=\mathfrak{s}_1\sqcup \mathfrak{s}_2$ and both $\mathfrak{s}_1$ and $\mathfrak{s}_2$ are even, then  $\Gamma_{\mathfrak{s}_1}$ and $\Gamma_{\mathfrak{s}_2}$ are linked by a chain of $2\delta_{\mathfrak{s}_1}+2\delta_{\mathfrak{s}_2}-1$ rational curves of multiplicity $2$,
\item  for a principal cluster $\mathfrak{s}$, each child of size $1$, $\{r\}<\mathfrak{s}$ say,   contributes a single rational curve  $T_{r}$ of multiplicity $1$, intersecting $\Gamma_\mathfrak{s}$.
\end{itemize} 
\end{proposition}

\begin{proof}
This essentially follows from specialising \cite[Theorems 7.12 and 7.18]{MR4201122} to the case in hand, noting that the minimal regular SNC model coincides with the minimal regular model in this case (more precisely, the description in the statement of this proposition is the description of the minimal regular SNC model of $C^L$ obtained from  \cite[Theorems 7.12 and 7.18]{MR4201122}, and visibly has no exceptional curves in its special fibre). The only catch is that the statements of \cite[Theorems 7.12 and 7.18]{MR4201122} contain some minor errors and as such one does not quite recover the description of $\mathcal{X}_{\bar{k}}$ given above. A corrected version of these results appears in the PhD thesis of Nowell \cite[Theorems 9.23, 9.31 and 9.32]{NowellThesis}, from which one obtains the claimed statement. 

When applying the results cited above, recall that we are imposing \Cref{assump:no_cotwin}. The necessary invariants which form the required input for  \cite[Theorems 7.12 and 7.18]{MR4201122} and \cite[Theorems 9.23, 9.31 and 9.32]{NowellThesis} are described in \Cref{min_ref_model_invariants} below. 
\end{proof}

\begin{caution}
In the following lemma only, we view clusters as being associated to $\pi f(x)$ rather than $f(x)$, since it is invariants of the former which constitute the required input for the results of \cite{MR4201122}.
\end{caution}

See \cite[Table 3]{MR4201122} and the references therein for the definitions of the invariants appearing in the following lemma. Briefly, for a proper cluster $\mathfrak{s}$ for $C^L$, the quantities  $d_\mathfrak{s}$, $\nu_\mathfrak{s}$ and $\delta_{\mathfrak{s}}$ are as defined in \Cref{subsec_clusters_defs}, but with $\pi f(x)$ in place of $f(x)$. By definition we have $\lambda_{\mathfrak{s}}=\tfrac{1}{2}\nu_{\mathfrak{s}}-d_{\mathfrak{s}}\sum_{\mathfrak{s}'<\mathfrak{s}}\lfloor \frac{|\mathfrak{s}'|}{2} \rfloor$ (we caution that this is the function denoted $\tilde{\lambda}_{\mathfrak{s}}$ in \cite[Notation 1.19]{DDMM18}).  The quantity $e_\mathfrak{s}$ is the minimal integer such that both $e_\mathfrak{s}d_\mathfrak{s}\in \mathbb{Z}$ and $e_\mathfrak{s}\nu_{\mathfrak{s}}\in 2\mathbb{Z}$. When $\mathfrak{s}$ is even, the \textit{invariant} $\epsilon_{\mathfrak{s}}\in \{\pm 1\}$ in the statement is given by evaluating the \textit{function} $\epsilon_{\mathfrak{s}}$ of \Cref{definition:epsilon_signs} (which no longer factors through $G_k$ in general since $C^L/K$ is not semistable) at any topological generator of the tame inertia group of $K$. For our purposes we may take as a definition that $\epsilon_{\mathfrak{s}}=(-1)^{\nu_{\mathfrak{s}^*}-|\mathfrak{s}^*|d_{\mathfrak{s}^*}}$ for $\mathfrak{s}^*$ as in \Cref{notat:s_star_notation}. We will not use this variant of $\epsilon_\mathfrak{s}$ anywhere else in the paper. Recall also from \Cref{notat_T_set} the definition of the set $T$.  

\begin{lemma} \label{min_ref_model_invariants} 
Let $\mathfrak{s}$ be a proper cluster for $C^L$ (i.e. for the polynomial $\pi f(x)$). Then $\mathfrak{s}$ is fixed by the inertia group $I_K$ of $K$, and all of the following hold:
\begin{itemize}
\item[(i)] we have $d_\mathfrak{s}\in \mathbb{Z}$ unless $\mathfrak{s}\in T$, in which case $d_\mathfrak{s}\in 1/2+\mathbb{Z}$,
\item[(ii)]   $\nu_\mathfrak{s}$ is odd unless either $\mathfrak{s}=\mathcal{R}$ and $\mathcal{R}$ is atypical, or $\mathfrak{s}\in T$. In these cases, $\nu_\mathfrak{s}$ is even.
\item[(iii)] if $\mathfrak{s}$ is even then $\epsilon_\mathfrak{s}=-1$ unless $\mathfrak{s}=\mathcal{R}$ is atypical, in which case $\epsilon_\mathcal{\mathfrak{s}}=1$,
\item[(iv)] if $|\mathfrak{s}|\geq 3$ then $e_\mathfrak{s}=2$ unless $\mathfrak{s}=\mathcal{R}$ is atypical, in which case $e_\mathfrak{s}=1$,
\item[(v)] if $\mathfrak{s}$ is principal then $\lambda_\mathfrak{s}\in \frac{1}{2}+\mathbb{Z}$,
\item[(vi)] if $\mathfrak{s}'<\mathfrak{s}$ are principal clusters with $\mathfrak{s}'$ odd, then $\delta_{\mathfrak{s}'}$ is even.
\end{itemize}
\end{lemma}

\begin{proof}
As noted above, the proper clusters for $C^L$ and their associated depths are the same as those for $C$ (and whether or not a cluster $\mathfrak{s}$ is proper/principal/odd/even/\ub~ is similarly independent of whether we view $\mathfrak{s}$ as a cluster for $C$ or $C^L$). However, given a cluster $\mathfrak{s}$ for $C$, when we view it as a cluster for $C^L$ the quantity $\nu_\mathfrak{s}$ increases by $1$ since the leading coefficient of $\pi f(x)$ has valuation one greater than that of $f(x)$. All claims are now a formal consequence of \Cref{semi criterion}, which applies since $C:y^2=f(x)$ is semistable. Explicitly, the claim that each $I_K$-orbit of proper clusters has size $1$ is part (2) of \Cref{semi criterion}. Part (i) is \Cref{integer relative depths}. Part (ii) for $\mathfrak{s}\notin T$ is  \cite[Lemma 4.7]{DDMM18}, whilst for $\mathfrak{s}\in T$ this follows from \cite[Lemma C.5]{DDMM18}, combined with \cite[Lemma 4.7]{DDMM18} applied to the parent of $\mathfrak{s}$. Parts (iii) and (iv) follow  from parts (i) and (ii). Part (v) follows from (ii). Finally, for part (vi) see \cite[Lemma C.7]{DDMM18}.
\end{proof}

It is convenient to package the description of $\mathcal{X}_{\bar{k}}$ given in \Cref{dual_graph_semistable_twist_prop} in terms of the following graph.

\begin{notation}
Define $\mathcal{T}$ to be the graph consisting of one vertex for each irreducible component of $\mathcal{X}_{\bar{k}}$, with vertices $v$ and $v'$ joined by an edge if and only if the corresponding components intersect.  We give each vertex a weight $d_v\in \{1,2\}$ according to the multiplicity of the corresponding component in $\mathcal{X}_{\bar{k}}$.
\end{notation}

\begin{remark} \label{rem:leaves_of_dual_graph}
We see from \Cref{dual_graph_semistable_twist_prop} that $\mathcal{T}$ is a connected tree. Note that any vertex of $\mathcal{T}$ of degree at least $3$ has weight $2$. The leaves of $\mathcal{T}$ correspond to the components $\Gamma_{r}$ for $r\in \mathcal{R}$ not in a twin, along with the `crosses' on the crossed tails $T_\mathfrak{t}$ for twins $\mathfrak{t}$. In particular, $\mathcal{T}$ has $|\mathcal{R}|=2g+2$ leaves. Moreover, each leaf has weight $1$.
\end{remark}

\begin{example}[Ramified quadratic twist of good reduction]
Suppose $f(x)\in \mathcal{O}_{K}[x]$ is monic, has degree $2g+2$ for some $g\geq 2$, and is such that the reduction $f(x)$ $(\textup{mod }\pi)$ is separable. Then $C:y^2=f(x)$ has good reduction, and $C^L/K$ is the hyperelliptic curve $C^L:y^2=\pi f(x)$. We now use \Cref{dual_graph_semistable_twist_prop} to describe $\mathcal{X}_{\bar{k}}$ in this case. The assumptions mean that $f(x)$ has a single proper cluster, given by the full set of roots $\mathcal{R}$. This cluster has depth $0$ and has $2g+2$ children, with each individual root $r\in \mathcal{R}$ contributing a child $\{r\}<\mathcal{R}$ of size $1$. The cluster picture is thus the following:
\begin{center}
		\clusterpicture   
  \Root[A] {} {first} {r1};
  \Root[A] {} {r1} {r2};
  \Root[A] {} {r2} {r3};
  \Root[A] {} {r3} {r4};
  \Root[Dot] {} {r4}{r5};
  \Root[Dot] {} {r5}{r6};
  \Root[Dot] {} {r6}{r7};
   \Root[A] {} {r7}{r8};
    \Root[A] {} {r8}{r9};
  \ClusterLDName c4[][0][] = (r1)(r2)(r3)(r4)(r5)(r6)(r7)(r8)(r9);
\endclusterpicture 
\end{center}   
By  \Cref{dual_graph_semistable_twist_prop}, $\mathcal{X}_{\bar{k}}$  consists of one component $\Gamma_{\mathcal{R}}$ of genus $0$ and multiplicity $2$, intersected transversely by $2g+2$ rational curves of multiplicity $1$, one for each root $r\in \mathcal{R}$, as depicted below. The graph $\mathcal{T}$ consists of $2g+2$ vertices of weight $1$, each joined to a common vertex $v_{\mathcal{R}}$ of weight $2$, as shown below also. In the picture we do not label multiplicities/weights unless they are greater than $1$.
\begin{center}
\begin{figure}[!htbp]
\begin{minipage}{.5\textwidth}
\begin{center}
\begin{tikzpicture}[scale=0.5]
\draw  [line width=1pt] (0.3,2.3)-- ++(0,-1.3); 
\draw  [line width=1pt] (1,2.3)-- ++(0,-1.3); 
\draw  [line width=1pt] (1.7,2.3)-- ++(0,-1.3); 
\draw  [line width=1pt] (2.4,2.3)-- ++(0,-1.3); 
\draw  [line width=1pt] (4.4,2.3)-- ++(0,-1.3); 
\draw  [line width=1pt] (5.1,2.3)-- ++(0,-1.3); 
 \draw [line width=1pt] (0,2)-- node[above]{$\quad\qquad  $} node[below, font=\small]{} ++(5.4,0);
\draw (3.4,2.5) node{$2$};
\draw (3.4,1.4) node{$\cdots$};
\draw (2.7,-0.6) node{\scalebox{1.3}{$\mathcal{X}_{\bar{k}}$}};
\end{tikzpicture}
\end{center}
\end{minipage}%
\begin{minipage}{.5\textwidth}
\begin{center}
\begin{tikzpicture}[scale=0.4]
\draw (4,2.6) node {$2$};
\draw (4,-2) node{\scalebox{1.3}{$\mathcal{T}$}};
	\BlackVertices
\VertexLN[x=2.5,y=-0.2,L=\relax]{0}{};
\VertexLN[x=3.5,y=-0.8,L=\relax]{1}{};
	\VertexLN[x=1.5,y=0.4,L=\relax]{3}{};
\VertexLN[x=0.8,y=1,L=\relax]{2}{};
	\VertexLN[x=4,y=2,L=\relax]{4}{};
	 
	\VertexLN[x=6.5,y=0.5,L=\relax]{9}{ };  
\VertexLN[x=7.5,y=1,L=\relax]{10}{ };  
\draw (4.9,0) node {$\iddots$};
	\BlackEdges
\Edge(3)(4)
\Edge(0)(4)
\Edge(1)(4)
\Edge(2)(4)
\Edge(4)(10)
\Edge(4)(9)
\end{tikzpicture}
\end{center}
\end{minipage}
\end{figure}
\end{center}
\vspace{-10pt}
This description of $\mathcal{X}_{\bar{k}}$ is consistent with work of Sadek \cite[Theorem 3.7]{MR3264495}.
\end{example}

\begin{example} \label{ex:recurring_continued_example}
Take $C$ to be the semistable genus $2$  hyperelliptic curve over $\mathbb{Q}_3$ considered previously in    \Cref{unram_genus_2_ub_example,example_of_a_cluster_picture}, and take $L=\mathbb{Q}_3(\sqrt{3})$, so that $C^L$ is the curve
\[C^L/\mathbb{Q}_3:y^2=3(x^2+3)((x-i)^2-3^2)((x+i)^2-3^2)\]
for $i$ a square root of $-1$. As in \Cref{example_of_a_cluster_picture} the cluster picture is as shown:
\begin{center}
		\clusterpicture            
  \Root[A] {} {first} {r1};
  \Root[A] {} {r1} {r2};
  \Root[A] {6} {r2} {r3};
  \Root[A] {} {r3} {r4};
  \Root[A] {6} {r4} {r5};
  \Root[A] {} {r5} {r6};
  \ClusterLDName c1[][\tfrac{1}{2}][ ] = (r1)(r2);
    \ClusterLDName c2[][1][ ] = (r3)(r4);
      \ClusterLDName c3[][1][ ] = (r5)(r6);
  \ClusterLDName c4[][0][] = (c1)(c3)(c3);
\endclusterpicture 
\end{center}
As explained previously in  \Cref{example_of_a_cluster_picture}, the full set of roots $\mathcal{R}$ is the unique principal cluster, and (as shown in the picture) there are $3$ twins $\mathfrak{t}_1$, $\mathfrak{t}_2$ and $\mathfrak{t}_3$, with $\delta_{\mathfrak{t}_1}=1/2$ and $\delta_{\mathfrak{t}_2}=\delta_{\mathfrak{t}_3}=1$. 
 By  \Cref{dual_graph_semistable_twist_prop}, $\mathcal{X}_{\bar{k}}$ consists of one component $\Gamma_{\mathcal{R}}$ of multiplicity $2$,  along with $3$ crossed tails, as depicted below. The corresponding graph $\mathcal{T}$ is pictured also. 
\begin{center}
\begin{figure} [!htbp]
\begin{minipage}{.5\textwidth}
\begin{center}
\begin{tikzpicture}[scale=0.55]
\draw  [line width=1pt] (0.3,2.3)-- ++(0,-2.3); 
\draw  [line width=1pt] (2.7,2.3)-- ++(-1,-2); 
\draw  [line width=1pt] (5.1,2.3)-- ++(-1,-2); 
\draw  [line width=1pt] (1.7,0.8)-- ++(1,-2); 
\draw  [line width=1pt] (4.1,0.8)-- ++(1,-2); 
\draw [line width=1pt] (0,2)-- node[above]{$2\quad \quad  \quad\Gamma_{\mathcal{R}}$} node[below, font=\small]{} ++(5.4,0);
\draw (0,1.5) node{$2$};
\draw (2.6,1.4) node{$2$};
\draw (5,1.4) node{$2$};
\draw (2.4,0.2) node{$2$};
\draw (4.8,0.2) node{$2$};
\draw (2.7,-2.4) node{\scalebox{1.3}{$\mathcal{X}_{\bar{k}}$}};
\draw [line width=1pt] (-0.8,0.8)-- node[above]{} node[below, font=\small]{ } ++(1.5,0); 
\draw [line width=1pt] (-0.8,0.3)-- node[above]{} node[below, font=\small]{ } ++(1.5,0); 
\draw [line width=1pt] (1.4,-0.5)-- node[above]{} node[below, font=\small]{ } ++(1.5,0); 
\draw [line width=1pt] (1.4,-1)-- node[above]{} node[below, font=\small]{ } ++(1.5,0); 
\draw [line width=1pt] (3.8,-0.5)-- node[above]{} node[below, font=\small]{ } ++(1.5,0); 
\draw [line width=1pt] (3.8,-1)-- node[above]{} node[below, font=\small]{ } ++(1.5,0); 
\end{tikzpicture}
\end{center}
\end{minipage}%
\begin{minipage}{.5\textwidth}
\begin{center}
\begin{tikzpicture}[scale=0.45]
	\draw (3.5,0.2) node {$2$};
\draw (3.5,-1.1) node {$2$};
\draw (6,-1.1) node {$2$};
\draw (3.5,2.2) node {$2$};
\draw (1.5,1) node {$2$};
\draw (6.5,1) node {$2$};
\draw (4.1,2.5) node {$~ ~ v_\mathcal{R}$};
\draw (4,-3.3) node{\scalebox{1.3}{$\mathcal{T}$}};
	\BlackVertices
	\VertexLN[x=1,y=-0.5,L=\relax]{1}{};
	\VertexLN[x=2,y=-0.5,L=\relax]{2}{};
	\VertexLN[x=1.5,y=0.5,L=\relax]{3}{};
	\VertexLN[x=4,y=2,L=\relax]{4}{};
	\VertexLN[x=4,y=0.5,L=\relax]{5}{}; 
	\VertexLN[x=4,y=-1,L=\relax]{6}{}; 
	\VertexLN[x=3.5,y=-2,L=\relax]{7}{};
	\VertexLN[x=4.5,y=-2,L=\relax]{8}{};
	\VertexLN[x=6.5,y=0.5,L=\relax]{9}{ }; 
\VertexLN[x=6.5,y=-1,L=\relax]{10}{}; 
\VertexLN[x=6,y=-2,L=\relax]{11}{}; 
\VertexLN[x=7,y=-2,L=\relax]{12}{}; 
	\BlackEdges
	\Edge(1)(3)
\Edge(2)(3)
\Edge(3)(4)
\Edge(4)(5)
\Edge(5)(6)
\Edge(6)(7)
\Edge(6)(8)
\Edge(4)(9)
\Edge(9)(10)
\Edge(10)(11)
\Edge(10)(12)
\end{tikzpicture}
\end{center}
\end{minipage}
\end{figure}
\end{center}
\vspace{-10pt}
In particular, in the terminology of the Namikawa--Ueno classification \cite{MR0369362}, $C^L/K$ has type $I_{1-2-2}^*$.
\end{example}

Returning to the general case, we now describe the $G_k$-action on the set of irreducible components of $\mathcal{X}_{\bar{k}}$ (equivalently the induced $G_k$-action on $\mathcal{T}$). To do this we introduce the following notation.

\begin{notation} \label{gamma_t_notation}
For each twin  $\mathfrak{t}=\{r_1,r_2\}$ let $\eta_{\mathfrak{t}}\in K^{s \times}$ be a choice of square root of 
\[\frac{(r_1-r_2)^2}{(-\pi)^{2d_\mathfrak{t}}},\]
noting that $d_\mathfrak{t}=v(r_1-r_2)$ so that the displayed quantity is a unit (we can have $d_\mathfrak{t}\in \frac{1}{2}+\mathbb{Z}$, so we need not have a canonical choice of square root). In particular, $\eta_\mathfrak{t}\in \mathcal{O}_{K^{\textup{nr}}}^{\times}$.
Define the function 
$\gamma_{\mathfrak{t},L}:G_K\rightarrow \{\pm 1\}$
by the formula 
\[\gamma_{\mathfrak{t},L}(\sigma)=\frac{\sigma(\eta_{\mathfrak{t}})}{\eta_{\sigma\mathfrak{t}}}.\]
The function $\gamma_{\mathfrak{t},L}$ factors through $\textup{Gal}(K^{\textup{nr}}/K)$. Thus we view $\gamma_{\mathfrak{t},L}$ as a function on $G_k$ also. In particular, we can speak about $\gamma_{\mathfrak{t},L}(F)$ where $F\in G_k$ is the Frobenius element. The function $\gamma_{\mathfrak{t},L}$ may depend on the choice of square root $\eta_\mathfrak{t}$, but its restriction to the stabiliser of $\mathfrak{t}$ does not.
We remark that we include $L$ in the notation for  $\gamma_{\mathfrak{t},L}$ since, when $d_\mathfrak{t}\in 1/2+\mathbb{Z}$, it depends on  the class of the  uniformiser $\pi$ in $K^\times/K^{\times 2}$. 
\end{notation}

We stress that \Cref{where_are_clusters_convention} is in place, which is  relevant for the   function $\epsilon_{\mathfrak{s}}$.

\begin{proposition} \label{prop:frob_action_min_reg_twist}
Let $\sigma \in G_k$. The action of $\sigma$ on the set of irreducible components of $\mathcal{X}_{\bar{k}}$ is   determined by:
\begin{itemize}
\item for $\mathfrak{s}$ principal, the component $\Gamma_\mathfrak{s}$ is sent to $\Gamma_{\sigma\mathfrak{s}}$,
\item for each $r\in \mathcal{R}$ not in a twin, the component $\Gamma_r$ is sent to $\Gamma_{\sigma r}$,
\item for a twin $\mathfrak{t}$ with $\mathfrak{t}\notin T$, the crossed tail $T_\mathfrak{t}$ is sent to $\gamma_\mathfrak{t}(\sigma)T_{\sigma\mathfrak{t}}$,\footnote{Here $-T_\mathfrak{t}$ denotes the crossed tail $T_\mathfrak{t}$ with crosses swapped; strictly speaking we should fix a labelling $\pm$ of the crosses to pin down the action, and this choice is closely related to the choices of square root in \Cref{gamma_t_notation} above. However, it will only be relevant in what follows to know whether the stabiliser of a twin $\mathfrak{t}$ fixes or swaps the crosses on $T_{\mathfrak{t}}$, and for this we can safely ignore this subtlety.}
\item for a twin $\mathfrak{t}\in T$, the crossed tail $T_\mathfrak{t}$ is sent to $\epsilon_\mathfrak{t}(\sigma)\gamma_\mathfrak{t}(\sigma)T_{\sigma\mathfrak{t}}$.
\end{itemize}
\end{proposition}

\begin{proof}
This essentially follows from \cite[Theorem 7.21]{MR4201122}, however in some cases the Frobenius action is not correctly computed in that work. The argument given in loc. cit. applies to show that the action is as claimed in the first two bullet points. However, for a crossed tail  $T_\mathfrak{t}$ corresponding to a twin $\mathfrak{t}$,  the computation given there is incorrect. We explain now how to correctly compute the action in this case. 

Since $C$ is semistable over $K$, the curve $C^L$ becomes semistable over $L$. The results of \cite{DDMM18} then apply to give an explicit description of the minimal proper regular model $\widetilde{\mathcal{X}}/\mathcal{O}_{L^\textup{ur}}$ of $C^L/L^{\textup{ur}}$  in terms of clusters for $C^L$ and their associated invariants. In particular, equations for the components of the special fibre $\widetilde{\mathcal{X}}_{\bar{k}}$ can be read off from \cite[Proposition 5.20]{DDMM18}. By uniqueness of the minimal regular model, the full Galois group $G_K$ acts semilinearly on $\widetilde{\mathcal{X}}_{\bar{k}}$, and this action factors through $\textup{Gal}(L^{\textup{ur}}/K)$. This action is described explicitly in terms of clusters in \cite[Theorem 6.2]{DDMM18}. Writing $G=\textup{Gal}(L^{\textup{ur}}/K^{\textup{nr}})$, the quotient $\widetilde{\mathcal{X}}/G$ is an $\mathcal{O}_{K^{\textup{nr}}}$-model for $C^L$, closely related to its minimal regular model. In particular, as explained in the proof of \cite[Theorem 7.21]{MR4201122}, the $G_k$-action on the crossed tails of $\mathcal{X}_{\bar{k}}$ can be read off from the $G_k$-action on the singular points of the special fibre of $\widetilde{\mathcal{X}}/G$, which can in turn be calculated using the results of \cite{DDMM18} mentioned above. To carry out this calculation, fix a twin $\mathfrak{t}=\{r_1,r_2\}$. We take as a centre for $\mathfrak{t}$ the quantity $z_\mathfrak{t}:=\frac{r_1+r_2}{2}\in K^{\textup{nr}}$.  As described in \cite[Proposition 5.20]{DDMM18}, associated to $\mathfrak{t}$ is  the component $\Gamma_{\mathfrak{t}}$ of $\widetilde{\mathcal{X}}_{\bar{k}}$, given by the equation
\[\Gamma_{\mathfrak{t}}: y^2=c_\mathfrak{t}\Big(x^2-\frac{(r_1-r_2)^2}{4\pi^{2d_\mathfrak{t}}} \textup{mod~}\mathfrak{m} \Big) ~~\quad \textup{for }\quad c_\mathfrak{t}=\frac{c_f}{\pi^{v(c_f)}}\prod_{r\in \mathcal{R}\setminus \mathfrak{t}}\left(\frac{z_\mathfrak{t}-r}{\pi^{v(z_{\mathfrak{t}}-r)}}\right)~~\textup{mod~}\mathfrak{m}. \]
Here $\mathfrak{m}$ denotes the maximal ideal in $\mathcal{O}_{\bar{K}}$ and `$\textup{mod }\mathfrak{m}$' denotes reduction to the residue field $\bar{k}$. Recall that $2d_\mathfrak{t}\in \mathbb{Z}$ is odd if $\mathfrak{t}\in T$, and is even otherwise. Using the description of the invariant $\nu_\mathfrak{t}$ afforded by \Cref{min_ref_model_invariants} (ii), we see from \cite[Theorem 6.2]{DDMM18} that the generator $\tau$ of $G$ acts on $\Gamma_\mathfrak{t}$ as the automorphism
\[(x,y)\longmapsto \begin{cases}(x,-y)~~&~~\mathfrak{t}\notin T,\\ (-x,y) ~~&~~\mathfrak{t}\in T.\end{cases}\]
The relevant singular points of the special fibre of $\widetilde{\mathcal{X}}/G$ arise as the image under the quotient map of the fixed points of the action of $\tau$ on $\Gamma_{\mathfrak{t}}$. 
These are the points $P_\mathfrak{t}^{\pm}\in \Gamma_\mathfrak{t}$ given by 
\[P_\mathfrak{t}^{\pm}=\big(\pm \tfrac{1}{2} \overline{\eta_{\mathfrak{t}}} ,0\big)~\textup{ if }\mathfrak{t}\notin T\quad \textup{and} \quad P_\mathfrak{t}^{\pm}=\big(0, \pm \tfrac{1}{2} \overline{\eta_{\mathfrak{t}}}  \sqrt{c_\mathfrak{t}} \big)~\textup{ if }\mathfrak{t}\in T,\]
where here $\overline{\eta_\mathfrak{t}}:=\eta_\mathfrak{t}~\textup{mod }\mathfrak{m}$. 
Since the points $P_{\mathfrak{t}}^{\pm }$ are fixed by $G$, the action of $\textup{Gal}(L^{\textup{ur}}/K)$ on these points descends to an action of $\textup{Gal}(K^{\textup{nr}}/K)=G_k$, and appealing to \cite[Theorem 6.2]{DDMM18} once more to determine this action we see that $\sigma \in G_k$ sends $P_{\mathfrak{t}}^{\pm}\in \Gamma_{\mathfrak{t}}$ to the point on $\Gamma_{\sigma \mathfrak{t}}$ given by acting coordinatewise on the expression for $P^{\pm}_{\mathfrak{t}}$ given above.  Now the points $P_{\mathfrak{t}}^{\pm}$ can be identified with the `crosses' on the crossed tail $T_\mathfrak{t}$ (cf. proof of \cite[Theorem 7.21]{MR4201122}). We thus see that the action is as claimed upon noting that, after making compatible choices of square roots, we have $\epsilon_\mathfrak{t}(\sigma)=\sigma(\sqrt{c_{\mathfrak{t}}})/\sqrt{c_{\sigma \mathfrak{t}}}$ (to justify this final equality, see \cite[Lemma 6.7]{DDMM18} and the surrounding discussion). 
\end{proof}

\begin{remark} \label{rem:galois_action_on_leaves}
In \Cref{rem:leaves_of_dual_graph} we described the leaves of $\mathcal{T}$. From  \Cref{prop:frob_action_min_reg_twist} we see that leaves corresponding to roots $r\in \mathcal{R}$ not lying in a twin  are permuted by $G_k$ as the corresponding roots are. Further, let us temporarily denote by $\mathcal{S}$ the subset of $\mathcal{R}$ consisting of roots lying in twins $\mathfrak{t}$ with $\mathfrak{t}\notin T$, noting that $\mathcal{S}\subseteq \mathcal{R}\cap K^{\textup{nr}}$.  If we further denote by $\mathcal{L}$ the set of leaves corresponding to the `crosses' on the crossed tails $T_\mathfrak{t}$ for $\mathfrak{t}\notin T$, then we see from \Cref{prop:frob_action_min_reg_twist} that $\mathcal{L}$ and $\mathcal{S}$ are isomorphic as $G_k$-sets.  
\end{remark}

\begin{example}
Returning to \Cref{ex:recurring_continued_example}, for all $\sigma \in G_k$ we have $\epsilon_{\mathfrak{t}_1}(\sigma)=\epsilon_{\mathfrak{t}_2}(\sigma)=\epsilon_{\mathfrak{t}_3}(\sigma)=1$ (cf. \Cref{epsilon_as_in_example}). One checks that we may take the functions $\gamma_{\mathfrak{t}_1}$, $\gamma_{\mathfrak{t}_2}$ and $\gamma_{\mathfrak{t}_3}$ to be identically $1$ also. Finally, the Frobenius element in $G_k$ fixes $\mathfrak{t}_1$ but swaps $\mathfrak{t}_2$ and $\mathfrak{t}_3$. We thus see from \Cref{prop:frob_action_min_reg_twist} that $F$ fixes the crossed tail corresponding to $\mathfrak{t}_1$ (the leftmost one in the picture) and swaps the crossed tails corresponding to $\mathfrak{t}_2$ and $\mathfrak{t}_3$. Moreover,  $F^2$ acts trivially on the full set of components, hence  the stabiliser in $G_k$ of $\mathfrak{t}_i$, $i=2,3$, acts trivially on the crosses of the corresponding crossed tail.
\end{example}

\subsection{The Tamagawa number up to squares} 

We now use the description of $\mathcal{T}$  along with its $G_k$-action, afforded by \Cref{dual_graph_semistable_twist_prop,prop:frob_action_min_reg_twist}, to compute the Tamagawa number of $J^L/K$ up to rational squares. We begin by describing the order of the component group over $\bar{k}$.  

\begin{lemma} \label{eq:geometric_component_group_size}
We have $|\Phi(\bar{k})|=2^{2g}$.
\end{lemma}

\begin{proof}
Since $\mathcal{T}$ is a tree, \cite[Proposition 9.6.6]{MR1045822} gives
\[|\Phi(\bar{k})|=\prod_{v\in \mathcal{T}}d_v^{\textup{deg}(v)-2} =\frac{\prod_{v\in \mathcal{T}}2^{\textup{deg}(v)-2} }{\prod_{\substack{v\in \mathcal{T}\\d_v=1}}2^{\textup{deg}(v)-2} },
 \]
where $\textup{deg}(v)$ denotes the degree of the vertex $v$.
Since $\mathcal{T}$ is a connected tree we have 
\[\prod_{v\in \mathcal{T}}2^{\textup{deg}(v)-2}=2^{\sum_{v\in \mathcal{T}}(\textup{deg}(v)-2)}=1/4.\]
Since any vertex of $\mathcal{T}$ of degree at least $3$ has multiplicity $2$, we find 
\[\prod_{\substack{v\in \mathcal{T}\\d_v=1}}2^{\textup{deg}(v)-2} =2^{-\#\{\textup{leaves of }\mathcal{T}\}}=2^{-2g-2},\]
the second equality following from \Cref{rem:leaves_of_dual_graph}. 
\end{proof}

We now turn to computing the size of the $G_k$-invariants of $\Phi(\bar{k})$ up to rational squares, which we will do with the aid of \Cref{tam computations}. We remark that an alternative appraoch might be to use the recipe \cite[Section 4.2]{Srin2016} of Srinivasan  for computing the Tamagawa number of a curve in terms of its minimal proper regular model.

 In what follows it will be convenient to work exclusively with the graph $\mathcal{T}$. To facilitate this, we transfer the intersection pairing between the components of $\mathcal{X}_{\bar{k}}$ to a pairing on the vertices of $\mathcal{T}$. Since by \Cref{dual_graph_semistable_twist_prop} all components intersect transversally, this pairing has a simple combinatorial description.

\begin{defi}
 For vertices $v$ and $v'$ of $\mathcal{T}$ define 
\[v\bullet v'=\begin{cases}0~~&~~v\textup{ not adjacent to }v',\\
1~~&~~v\neq v' ~ \textup{ and }v,v'\textup{ adjacent,}\\
-\frac{1}{d_v}\sum_{w\neq v}d_w ~v\bullet  w~~&~~v=v'. \end{cases}\]
Note that if vertices $v,v'$ of $\mathcal{T}$ correspond to components $\Gamma_v$ and $\Gamma_{v'}$ respectively, then $v\bullet v'$ is the intersection number between $\Gamma_v$ and $\Gamma_{v'}$. We extend this product bilinearly to the $\mathbb{Q}$-vector space $V$ with basis the vertices of $\mathcal{T}$. 
\end{defi}

\begin{notation}
For $v\in \mathcal{T}$ denote by $r_v$ the size of the $G_k$-orbit of $v$. If $r_v$ is even   write 
\[\epsilon_v=v-Fv+...-F^{r_v-1}v~\in V.\]
We now define a matrix $M$ with rows and columns indexed by the even length $G_k$-orbits of vertices of $T$ as follows. For each even length $G_k$-orbit $O$, pick a representative $v_O\in O$. Then the $(O,O')$-entry of $M$ is defined as 
\[M_{O,O'}=\frac{1}{r_{v_O}}\epsilon_{v_O}\bullet \epsilon_{v_{O'}}.\] 
\end{notation}

The relevance of the above definitions is that, by \Cref{tam computations}, we have 
\begin{equation} \label{eq:tam_combination_ram_twist}
|\det M|\equiv 2^{\epsilon(C^L/K)}\cdot \frac{|\Phi(\bar{k})|}{|\Phi(k)|} ~~(\textup{mod}~\mathbb{Q}^{\times 2}).
\end{equation}

\begin{proposition} \label{tam_twist_good_case}
Suppose that either $\mathcal{R}$ is a principal cluster, or $\mathcal{R}=\mathfrak{s}_1\sqcup \mathfrak{s}_2$ is a disjoint union of two proper clusters $\mathfrak{s}_1$ and $\mathfrak{s}_2$ which are not swapped by $G_k$. Then 
\[|\det M |=2^{\# \{\textup{even sized }G_k\textup{-orbits of leaves of }\mathcal{T}\}}.\]
\end{proposition}

We begin with a lemma, which  is a variant of \cite[Lemma 9.6.7]{MR1045822}.

\begin{lemma} \label{matrix_tree_lemma}
Let $\mathbb{T}$ be a rooted tree with root $R$. Let $N$ be a matrix with rational coefficients whose rows and columns are indexed by the vertices of $\mathbb{T}$. Suppose that $N_{v,v'}= 0$ unless either $v=v'$ or $v$ and $v'$ are adjacent in $\mathbb{T}$. Further, suppose that all rows of $N$ sum to 0, save possibly the row corresponding to the root $R$. Then
\[\textup{det}N=\left(\prod_{v\in \mathbb{T},v\neq R}-N_{v,v_{\textup{parent}}}\right)\left(\sum_{v\in \mathbb{T}}N_{R,v}\right)\]
where here for a vertex $v\neq R$ of $\mathbb{T}$, $v_{\text{parent}}$ denotes the parent of $v$ in $\mathbb{T}$ (that is, the vertex adjacent to $v$ on the unique path in $\mathbb{T}$ from $v$ to the root $R$). 
\end{lemma}

\begin{proof}
The strategy of proof is the same as that of \cite[Lemma 9.6.7]{MR1045822}, and is by induction on $n=|\mathbb{T}|$. If $n=1$ the result is clear, so assume $n>1$. Let $v\neq R$ be a leaf of $\mathbb{T}$ and order the vertices of $\mathbb{T}$ so that $v$ is the first vertex, and its parent $v'$ is the second (the determinant is independent of the ordering of vertices, this is just to enable us to write down the matrix explicitly). The   assumptions on $N$ mean that it has the form
\[
N=\left(\begin{array}{ccccc}
-N_{v,v'} & N_{v,v'} & 0 & 0 & 0\\N_{v',v} & N_{v',v'} & * & * & *\\
0 & * & * & * & *\\
0 & * & * & * & *\\
0 & * & * & * & *
\end{array}\right).
\]
If $N_{v,v'}=0$ then $\text{det}N$ is as claimed, so suppose $N_{v,v'}\neq 0$.
Adding column 1 to column 2 and then adding $\frac{N_{v',v}}{N_{v,v'}}\cdot (\textup{row }1)$  to row 2 does not change the determinant, and transforms the matrix above into  the matrix
\[
\left(\begin{array}{ccccc}
-N_{v,v'} & 0 & 0 & 0 & 0\\
0 & N_{v',v'}+N_{v',v} & * & * & *\\
0 & * & * & * & *\\
0 & * & * & * & *\\
0 & * & * & * & *
\end{array}\right).
\]
Here all entries indicated by a $*$ remain unchanged from the corresponding entries of $N$. Let $\widetilde{N}$ be the matrix obtained by removing the first row and column from this matrix, so that $\text{det}N=-N_{v,v'}\text{det}\widetilde{N}$. Letting $\widetilde{\mathbb{T}}$ be the rooted tree obtained  from $\mathbb{T}$ by removing the leaf $v$ (with root equal to the root $R$ of $\mathbb{T}$) we see that $\widetilde{N}$ satisfies the hypothesis of the statement with respect to $\widetilde{\mathbb{T}}$. By induction 
\[\text{det}N=-N_{v,v'}\text{det}\widetilde{N}=-N_{v,v'}\Bigg(\prod_{v\in \widetilde{\mathbb{T}},v\neq R}-N_{v,v_{\text{parent}}}\Bigg)\Bigg(\sum_{v\in \mathbb{T}}N_{R,v}\Bigg), \]
as desired.
\end{proof}

\begin{proof}[Proof of \Cref{tam_twist_good_case}]
If $\mathcal{R}$ is principal, denote by $R$ the vertex of $\mathcal{T}$ corresponding to the component $\Gamma_\mathcal{R}$. If $\mathcal{R}=\mathfrak{s}_1\sqcup \mathfrak{s}_2$ denote by $R$ the vertex of $\mathcal{T}$ corresponding to $\Gamma_{\mathfrak{s}_1}$. In either case, $R$ is fixed by $G_k$ and $d_R=2$.  We view $\mathcal{T}$ as a rooted tree with root $R$. For a vertex $v\neq R$ we denote by $P(v)$ the parent of $v$ in $\mathcal{T}$. We say that $v$ is a child of a vertex $w$ if $w=P(v)$.   

Now take a vertex $v$ of $\mathcal{T}$ with $r_v$ even, noting that this forces $v\neq R$.  If $v'$ is a child of $v$ then $r_{v}$ divides $r_{v'}$. In particular, $r_{v'}$ is even also. In this case we write $r_{v'}=m_{v'}r_v$, noting that $m_{v'}$ is the number of vertices in the $G_k$-orbit of $v'$ having parent $v$. One then computes $\epsilon_{v}\bullet \epsilon_{v'}=m_{v'}r_v=r_{v'}$, so we have
\begin{equation} \label{intersection_prod_a}
\frac{1}{r_v}\epsilon_v\bullet \epsilon_{v'}=m_{v'}\quad\textup{ and }\quad\frac{1}{r_{v'}}\epsilon_{v'}\bullet\epsilon_v=1.
\end{equation}
Moreover, we have $\epsilon_v\bullet \epsilon_v=r_v v\bullet v$, giving
\begin{equation} \label{intersection_prod_b}
\frac{1}{r_v}\epsilon_v\bullet\epsilon_v=-\frac{d_{P(v)}}{d_v}-\frac{1}{d_v}\sum_{v'\textup{ child of }v}d_{v'}.
\end{equation}

To make use of these computations, pick compatibly a representative for each even sized $G_k$-orbit of vertices in $\mathcal{T}$ in such a way that if $v$ is picked, then for each $G_k$-orbit containing a child of $v$, the chosen representative of that orbit is itself a child of $v$. The subgraph of $\mathcal{T}$ generated by all chosen representatives is a finite disjoint union of connected trees, $\mathcal{T}_1,...,\mathcal{T}_s$ say. Each $\mathcal{T}_i$ is naturally a rooted tree, with root the unique vertex of $\mathcal{T}_i$ closest to $R$, and we extend the notion of child/parent to $\mathcal{T}_i$. We caution however that we reserve the notation $P(v)$ for the parent of a vertex $v$  in the tree  $\mathcal{T}$. Now for $1\leq i \leq s$, define $N_i$ to be the matrix whose rows and columns are indexed by the vertices of $\mathcal{T}_i$, and such that the $(v,v')$-entry of $N_i$ is given by 
\begin{equation} \label{N_i_formula}
(N_i)_{v,v'}=\frac{d_vd_{v'}}{r_v}\epsilon_v\bullet\epsilon_{v'}=\begin{cases} d_vd_{v'}~~&~~v\textup{ a child of }v'\textup{ in }\mathcal{T}_i,\\ m_{v'}d_vd_{v'}~~&~~v\textup{ parent of }v'\textup{ in }\mathcal{T}_i, \\ -d_vd_{P(v)}-\sum_{w\textup{ child of }v\textup{ in }\mathcal{T}_i}m_wd_vd_w~~&~~v=v',\\ 0~~&~~\textup{otherwise},\end{cases}
\end{equation}
the second equality following from  \eqref{intersection_prod_a} and \eqref{intersection_prod_b}.
By construction we have 
\begin{equation*}
|\det M |=\prod_{i=1}^s\Big(|\det N_i |\prod_{v\in \mathcal{T}_i}d_v^{-2}\Big).
\end{equation*}
Applying \Cref{matrix_tree_lemma}  to each of the matrices $N_i$ we find
\begin{equation} \label{det_m_preliminary}
|\det M |=\prod_{i=1}^s\prod_{v\in \mathcal{T}_i}\frac{d_{P(v)}}{d_v}.
\end{equation}

\textbf{Claim:} For each $1\leq i\leq s$ we have 
\[\prod_{v\in \mathcal{T}_i}\frac{d_{P(v)}}{d_v}=\frac{d_{P(R_i)}}{2}2^{\#\{\textup{leaves of }\mathcal{T}\textup{ appearing in }\mathcal{T}_i\}}.\] 

\begin{proof}[Proof of claim]
First suppose that $\mathcal{T}_i$ consists of the single vertex $R_i$. Then $R_i$ is necessarily a leaf in $\mathcal{T}$, hence has weight $1$ and parent (in $\mathcal{T}$) of weight $2$. Thus the formula holds in this case.

Now assume that $\mathcal{T}_i$ consists of at least $2$ vertices, and for a vertex $v$ in $\mathcal{T}_i$  let $\textup{deg}_{\mathcal{T}_i}(v)$ denote the degree of $v$ when viewed as a vertex of $\mathcal{T}_i$ (as opposed  to a vertex of $\mathcal{T}$). Note that a vertex $v\neq R_i$ of $\mathcal{T}_i$ is the parent of $(\textup{deg}_{\mathcal{T}_i}(v)-1)$-many vertices of $\mathcal{T}_i$, whilst $R_i$ is the parent of $\textup{deg}_{\mathcal{T}_i}(v)$-many vertices of $\mathcal{T}_i$. Consequently, we have 
\[\prod_{v\in \mathcal{T}_i}\frac{d_{P(v)}}{d_v}=d_{R_i}d_{P(R_i)}\prod_{v\in \mathcal{T}_i}d_v^{\textup{deg}_{\mathcal{T}_i}(v)-2}=d_{R_i}d_{P(R_i)} \frac{ \prod_{v\in \mathcal{T}_i}2^{\textup{deg}_{\mathcal{T}_i}(v)-2}}{ \prod_{\substack{v\in \mathcal{T}_i \\ d_v=1}}2^{\textup{deg}_{\mathcal{T}_i}(v)-2} }.\]
Since $\mathcal{T}_i$ is a connected tree we have 
\[\prod_{v\in \mathcal{T}_i}2^{\textup{deg}_{\mathcal{T}_i}(v)-2}=2^{\sum_{v\in \mathcal{T}_i}(\textup{deg}_{\mathcal{T}_i}(v)-2)}=1/4.\]
On the other hand, if $v\in \mathcal{T}_i$ has $\textup{deg}_{\mathcal{T}_i}(v)\geq 3$ then $v$ necessarily has degree at least $3$ when viewed as a vertex of $\mathcal{T}$. It then follows from the description of $\mathcal{T}$ afforded by \Cref{dual_graph_semistable_twist_prop} that $d_v=2$.
All together, this gives 
\begin{equation} \label{almost_proved_tree_claim}
\prod_{v\in \mathcal{T}_i}\frac{d_{P(v)}}{d_v}=\frac{d_{R_i}d_{P(R_i)}}{4}2^{\#\{v\in \mathcal{T}_i~~\colon~~\textup{deg}_{\mathcal{T}_i}(v)=1,~d_{v}=1 \}}.
\end{equation}
Under the assumption that $\mathcal{T}_i$ has at least $2$ vertices, we see that a vertex $v\in \mathcal{T}_i$ is a leaf in $\mathcal{T}$ if and only if $\textup{deg}_{\mathcal{T}_i}(v)=1$  and $v\neq R_i$. When this is the case, $v$ necessarily has weight $1$. This observation combined with \eqref{almost_proved_tree_claim} proves the claim (note that if $\deg_{\mathcal{T}_i}(R_i)>1$ then, since $R_i\neq R,$ we see that $R_i$ must have degree at least $3$ in $\mathcal{T}$, hence weight $2$). 
\end{proof}
Returning to the proof of the proposition, note that the number of leaves of $\mathcal{T}$ which appear in some $\mathcal{T}_i$ is precisely the number of even sized $G_k$-orbits of leaves of $\mathcal{T}$. Combining the claim with \eqref{det_m_preliminary} thus gives 
\begin{equation} \label{almost_there_tam_twist}
|\det M|=2^{\# \{\textup{even sized }G_k\textup{-orbits of leaves of }\mathcal{T}\}}\cdot \prod_{i=1}^s\frac{d_{P(R_i)}}{2}.
\end{equation}
For each $1\leq i\leq s$ the parent of $R_i$ lies in an odd sized $G_k$-orbit and has a child lying in an even sized $G_k$-orbit. In particular, either $R_i=R$ or $R_i$ has degree at least $3$ in $\mathcal{T}$. Either way, $d_{R_i}=2$ and the result follows. 
 \end{proof}

In the remaining case, when $\mathcal{R}=\mathfrak{s}_1\sqcup \mathfrak{s}_2$ is a disjoint union of $2$ principal clusters swapped by $G_k$, the result is the following. 

\begin{proposition} \label{prop:bad_cases_tree_tamagawa}
Suppose that $\mathcal{R}=\mathfrak{s}_1\sqcup \mathfrak{s}_2$ is a disjoint union of $2$ principal clusters $\mathfrak{s}_1$ and $\mathfrak{s}_2$ that are swapped by $G_k$. Then
\[|\det M |=\begin{cases} 2^{\# \{\textup{even sized }G_k\textup{-orbits of leaves of }\mathcal{T}\}}~~&~~g\textup{ odd, or }g\textup{ even and }\mathcal{R}\textup{ not atypical,}\\ \frac{1}{2}\cdot 2^{ \# \{\textup{even sized }G_k\textup{-orbits of leaves of }\mathcal{T}\}}~~&~~g\textup{ even and } \mathcal{R}\textup{ atypical.}\end{cases}\]
\end{proposition}
 
 \begin{proof}
 We indicate how to adapt the proof of  \Cref{tam_twist_good_case} to these cases.
 
  First suppose that $g$ is odd. Then both $\mathfrak{s}_1$ and $\mathfrak{s}_2$ are even. The vertices of $\mathcal{T}$ corresponding to $\Gamma_{\mathfrak{s}_1}$ and $\Gamma_{\mathfrak{s}_2}$ are joined by a path consisting of an odd number of vertices of multiplicity $2$. Let $R$ be the middle vertex in this path, which is fixed by $G_k$ and has multiplicity $2$. With this definition of $R$, the proof of \Cref{tam_twist_good_case} applies verbatim to give the desired result.
  
  Now suppose $g$ is even, so that $\mathfrak{s}_1$ and $\mathfrak{s}_2$ are both odd. Since $\mathfrak{s}_1$ and $\mathfrak{s}_2$ are swapped by $G_k$ we have $\delta_{\mathfrak{s}_1}=\delta_{\mathfrak{s}_2}$, and the vertices of $\mathcal{T}$ corresponding to $\Gamma_{\mathfrak{s}_1}$ and $\Gamma_{\mathfrak{s}_2}$ are joined by a path consisting of $\delta_{\mathfrak{s}_1}$ vertices of weight $1$. If $\mathcal{R}$ is atypical then this path consists of an odd number of vertices, and we take as a root $R$ the middle vertex in this path. This is fixed by $G_k$ and has weight $1$. Following the proof of \Cref{tam_twist_good_case}, \eqref{almost_there_tam_twist} becomes 
  \begin{equation} \label{almost_there_tam_twist}
|\det M|=2^{\# \{\textup{even sized }G_k\textup{-orbits of leaves of }\mathcal{T}\}}\cdot  \frac{d_{R}}{2}.
\end{equation}
To see this, note that every vertex of $\mathcal{T}$ other than $R$ has an even sized $G_k$-orbit, so that $s=1$ in the proof of \Cref{tam_twist_good_case}. The result now follows immediately.

Finally, suppose that $g$ is even but that $R$ is not atypical, so that the vertices of $\mathcal{T}$ corresponding to $\Gamma_{\mathfrak{s}_1}$ and $\Gamma_{\mathfrak{s}_2}$ are joined by a path consisting of a (positive since $\delta_{\mathfrak{s}_1}\geq 1$) even number of vertices. Let $R_1$ and $R_2$ be the middle vertices on this path, noting that they have degree $2$ in $\mathcal{T}$ and weight $1$. Note that every vertex of $\mathcal{T}$ has an even sized $G_k$-orbit. As in the proof of  \Cref{tam_twist_good_case}, we compatibly pick a representative for each (even sized) $G_k$-orbit of vertices in $\mathcal{T}$, starting with $R_1$, and in such a way that if $v$ is picked  then, for each $G_k$-orbit containing a child of $v$, the chosen representative of the orbit is itself a child of $v$. Let $\mathcal{T}_1$ be the subtree of $\mathcal{T}$ generated by the chosen vertices. This is a connected tree and we take $R_1$ as a root for $\mathcal{T}_1$. As in the proof of \Cref{tam_twist_good_case}, define $N_1$ to be the matrix whose rows and columns are indexed by the vertices of $\mathcal{T}_1$ and such that the $(v,v')$-entry of $N_1$ is given by $(N_1)_{v,v'}=\frac{d_vd_{v'}}{r_v}\epsilon_v\bullet\epsilon_{v'}$. One then has $|\det M|=|\det N_1|\cdot \prod_{v\in \mathcal{T}_1}d_v^{-2}$. This time, the formula \eqref{N_i_formula} is valid provided $(v,v')\neq (R_1,R_1)$. Noting that $\epsilon_{R_1}=R_1-R_2$, one computes
\[(N_1)_{R_1,R_1}=\frac{d_{R_1}^2}{2}(R_1-R_2)\cdot (R_1-R_2)=-2d_{R_1}^2-\sum_{v\textup{ child of }R_1 \textup{ in }\mathcal{T}_1}m_vd_vd_{R_1}.\]
Again, the matrix $N_1$ satisfies the conditions of \Cref{matrix_tree_lemma}, and the row corresponding to $R_1$ sums to $2d_{R_1}^2$. Thus 
\[|\det M|=2\cdot \prod_{v\in \mathcal{T}_1,~v\neq R_1}\frac{d_{P(v)}}{d_v}.\]
Now $\mathcal{T}_1$ consists of at least $2$ vertices, $R_1$ has degree $1$ in $\mathcal{T}_1$ and weight $1$, and leaves of $\mathcal{T}_1$ other than $R_1$ correspond bijectively to (necessarily even sized) $G_k$-orbits of leaves in $\mathcal{T}$, all of which have weight $1$. Arguing as in the claim in the proof of \Cref{matrix_tree_lemma} now gives the result.
 \end{proof}
 
 Recall from \Cref{notat:etaC} that we set $\kappa(C)=1$ if   $\mathcal{R}=\s_1\sqcup \s_2$ is a disjoint union of two odd $G_k$-conjugate clusters with both $\delta_{s_1}$ and $\delta_{s_2}$ odd, and set $\kappa(C)=0$ otherwise. Putting everything together we obtain the following.

\begin{cor} \label{main_ram_quad_twist_cor}
We have 
\begin{eqnarray*}
\epsilon(C^L/K)+\textup{ord}_2c(J^L/K)&\equiv &\kappa(C)+\#\Big\{\textup{even-sized }G_k\textup{-orbits on }  \mathcal{R}\cap K^{\textup{nr}} \Big\}+\\&& \#\Big\{G_k\textup{-orbits }O\subseteq T\textup{ with }\prod_{\mathfrak{t}\in O}\epsilon_{\mathfrak{t}}(F)\gamma_{\mathfrak{t},L}(F)=-1\Big\}~~(\textup{mod }2).
\end{eqnarray*}
\end{cor}

\begin{proof}
With the matrix $M$ defined as above, combining \Cref{eq:geometric_component_group_size} with  \eqref{eq:tam_combination_ram_twist} gives 
\[\epsilon(C^L/K)+\textup{ord}_2c(J^L/K)\equiv \textup{ord}_2\det M~~(\textup{mod }2).\]
 \Cref{tam_twist_good_case,prop:bad_cases_tree_tamagawa} then give 
 \begin{equation}\label{eq:intermediate_equation_ram_twists}
 \textup{ord}_2\det M \equiv \kappa(C)+\# \{\textup{even sized }G_k\textup{-orbits of leaves of }\mathcal{T}\}~~(\textup{mod }2).
 \end{equation}
 As in \Cref{rem:leaves_of_dual_graph} there are $2g+2$ leaves of $\mathcal{T}$. By    \Cref{rem:galois_action_on_leaves}, the leaves corresponding to elements $r\in \mathcal{R}$, together with leaves arising as the crosses on the crossed tails $T_\mathfrak{t}$ for twins $\mathfrak{t}\notin T$, form a $G_k$-set isomorphic to $\mathcal{R}\cap K^{\textup{nr}}$. These leaves give rise to the second term on the right hand side of the statement. The remaining leaves arise as the crosses on the crossed tails $T_\mathfrak{t}$ corresponding to twins $\mathfrak{t}\in T$.  Using \Cref{prop:frob_action_min_reg_twist} once again we see that each $G_k$-orbit $O\subseteq T$ gives rise to a single even sized $G_k$-orbit of leaves if $\prod_{\mathfrak{t}\in O}\epsilon_{\mathfrak{t}}(F)\gamma_{\mathfrak{t},L}(F)=-1$, and either $0$ or $2$  such orbits  if $\prod_{\mathfrak{t}\in O}\epsilon_{\mathfrak{t}}(F)\gamma_{\mathfrak{t},L}(F)=1$ (according to whether $|O|$ is odd or even). This gives the result.
\end{proof}


\section{Proof of Proposition 11.1} \label{completion_ram_quad_odd_res}

For this section we take $K$ to be a non-archimedean local field of odd residue characteristic, take $L=K(\sqrt{\pi})$ to be a ramified quadratic extension of $K$, and let $C/K$ be a semistable hyperelliptic curve. We now combine the results of  Sections \ref{min reg model ramified}-\ref{sec:ram_twist_hyp_curve} to prove \Cref{prop:ram_cases_of_conjecture}. For convenience, we recall the statement.

\begin{proposition}[=\Cref{prop:ram_cases_of_conjecture}] \label{later_ramified_cases}
 \Cref{Kramer Tunnell} holds for $C$ and $L/K$. 
\end{proposition}

\begin{proof}
By \Cref{odd degree extension}, we are at liberty to replace $K$ with an arbitrarily large odd-degree unramified extension. In particular, we can without loss of generality assume that $C/K$ is given by an equation of the form $y^2=f(x)$ where $f(x)$ satisfies  \Cref{assump:no_cotwin} (see \Cref{can_reduce_to_no_cotwins_remark} for a justification of this). This allows us to use the results of \Cref{sec:hyp_betts_group_calculation} and \Cref{sec:ram_twist_hyp_curve} which were proven under this simplifying assumption.

Combining \Cref{mod 2 betts group cor} and  \Cref{main_ram_quad_twist_cor} with \eqref{equation_norm_as_tamag_product}  and  \eqref{betts_dokchitser_input} gives 
 \begin{eqnarray*}
w(J/L)\cdot(-1)^{\epsilon(C/K)+\epsilon(C^L/K)+\dim J(K)/\N J(L)}&=&(-1)^{\#\big\{\textup{even-sized }G_k\textup{-orbits on }  \mathcal{R}\cap K^{\textup{nr}} \big\}}\\&&\cdot(-1)^{\#\big\{G_k\textup{-orbits }O\subseteq T\textup{ with }\prod_{\mathfrak{t}\in O}\gamma_{\mathfrak{t},L}(F)=-1\big\}}.
\end{eqnarray*}
Here we are using the notation of \Cref{sec:clusters,sec:ram_twist_hyp_curve}, so that $\mathcal{R}$ denotes the set of roots of $f(x)$ in $K^s$, the set $T$ is as defined in \Cref{notat_T_set}, and the signs $\gamma_{\mathfrak{t},L}$ are as defined in \Cref{gamma_t_notation}.
To prove  \Cref{Kramer Tunnell} we see that it suffices to establish the equality
\begin{equation} \label{conj_number_of_orbs_disc}
(\Delta_C,L/K)\stackrel{?}{=}(-1)^{\#\big\{\textup{even-sized }G_k\textup{-orbits on }  \mathcal{R}\cap K^{\textup{nr}} \big\} +\#\big\{G_k\textup{-orbits }O\subseteq T\textup{ with }\prod_{\mathfrak{t}\in O}\gamma_{\mathfrak{t},L}(F)=-1\big\}}.
\end{equation}
Recall from \Cref{integer relative depths} that we have $\cup_{\mathfrak{t}\in T}\mathfrak{t}=\mathcal{R}\setminus \mathcal{R}\cap K^{\textup{nr}}$, and that each $\mathfrak{t}=\{r_{\mathfrak{t},1},r_{\mathfrak{t},2}\}$ is an intertia-orbit of roots of $f(x)$. In particular, we can factor $f(x)$ over $K$ as a product 
\[f(x)=f_{\textup{nr}}(x) \cdot \prod_{O\in T/G_k}f_O(x)\] where $f_{\textup{nr}}(x)\in K[x]$ splits over $K^{\textup{nr}}$ and where, for a $G_k$-orbit $O\subseteq T$, we have
\[f_O(x)=\prod_{\mathfrak{t}\in O}(x-r_{\mathfrak{t},1})(x-r_{\mathfrak{t},2})~~\in K[x].\]
In what follows, for a polynomial $g(x)$ we write $\Delta_{g}$ for its discriminant. 
From the above factorisation we find
\[(\Delta_C,L/K)=(\Delta_{f_{\textup{nr}}},L/K)\cdot \prod_{O\in T/G_k}(\Delta_{f_O},L/K).\]
This  follows from the fact that, for coprime polynomials $h_1(x)$, $h_2(x)\in K[x]$, we have $\Delta_{h_1h_2}=\Delta_{h_1}\Delta_{h_2}\textup{Res}(h_1,h_2)^2$ where $\textup{Res}(h_1,h_2)\in K^{\times}$ denotes the resultant of $h_1(x)$ and $h_2(x)$. 

Since $L/K$ is ramified whilst $f_\textup{nr}(x)$ splits over an unramified extension, we see that $(\Delta_{f_0},L/K)=1$ if and only if $\Delta_{f_{\textup{nr}}}$ is a square in $K$, which in turn happens if and only if the Frobenius element $F\in G_k$ acts as an even permutation of the roots of $f_\textup{nr}(x)$. Thus we have 
\begin{equation*}(\Delta_{f_\textup{nr}},L/K)=(-1)^{ \#\big\{\textup{even-sized }G_k\textup{-orbits on }  \mathcal{R}\cap K^{\textup{nr}} \big\}}.
\end{equation*}
To conclude, we claim that for each $G_k$-orbit $O\subseteq T$ we have $(\Delta_{f_O},L/K)=\prod_{\mathfrak{t}\in O}\gamma_{\mathfrak{t},L}(F)$. 
Indeed, from the definition of $\gamma_{\mathfrak{t},L}$ given in \Cref{gamma_t_notation}, we see that $\prod_{\mathfrak{t}\in O}\gamma_{\mathfrak{t},L}(F)$ is equal to $1$ if and only if the quantity 
 \[\prod_{\mathfrak{t}\in O} (r_{\mathfrak{t},1}-r_{\mathfrak{t},2})^2 (-\pi)^{-2d_\mathfrak{t}} ~~\in \mathcal{O}_{K}^{\times}\] 
 is a square in $K$. We thus have 
\[\prod_{\mathfrak{t}\in O}\gamma_{\mathfrak{t},L}(F)=\Big(\prod_{\mathfrak{t}\in O} (r_{\mathfrak{t},1}-r_{\mathfrak{t},2})^2 (-\pi)^{-2d_\mathfrak{t}} ,L/K\Big)=\Big(\prod_{\mathfrak{t}\in O} (r_{\mathfrak{t},1}-r_{\mathfrak{t},2})^2,L/K\Big),\]
where for the second equality we note that $-\pi$ is a norm from $L=K(\sqrt{\pi})$ (recall that, since $\mathfrak{t}\in T$, the quantity $2d_\mathfrak{t}$ is an odd integer). For $\mathfrak{t}\neq \mathfrak{t}'\in O$ write $R(\mathfrak{t},\mathfrak{t}')=\prod_{r\in \mathfrak{t}, r'\in \mathfrak{t}'}(r-r')$, noting that this quantity lies in $K^{\textup{nr}}$ and that $R(\mathfrak{t},\mathfrak{t}')=R(\mathfrak{t}',\mathfrak{t})$. Then we have 
\[\Delta_{f_O}=\prod_{\mathfrak{t}\in O}(r_{\mathfrak{t},1}-r_{\mathfrak{t},2})^2\cdot \prod_{\{\mathfrak{t},\mathfrak{t}'\}\subseteq O}R(\mathfrak{t},\mathfrak{t}')^2,\]
where the second product runs over all unordered pairs of distinct elements of $O$. The   product $\prod_{\{\mathfrak{t},\mathfrak{t}'\}\subseteq O}R(\mathfrak{t},\mathfrak{t}')$ is visibly fixed by $G_K$, hence lies  in $K$. We conclude that $\Delta_{f_O}$ and $\prod_{\mathfrak{t}\in O}(r_{\mathfrak{t},1}-r_{\mathfrak{t},2})^2$ are congruent modulo squares in $K^\times$, proving the claim.
 \end{proof}
 
\section{Residue characteristic 2}   \label{residue characteristic 2}

In this section we consider \Cref{Kramer Tunnell} when $K$ is a finite extension of $\mathbb{Q}_2$ and when the quadratic extension $L/K$ is ramified.  Let $C/K$ be a hyperelliptic curve with Jacobian $J$.
We suppose henceforth that $J/K$ has good ordinary reduction. Let $J(K)_1$ denote the kernel of reduction on $J(K)$, and define $J(L)_1$ similarly. We begin by considering the norm map from $J(L)_1$ to $J(K)_1$. 

\begin{lemma} \label{Lubin lemma}
We have
\[\left|J(K)_1/\N  J(L)_1 \right|=\left|J(K)_1[2]\right|.\]
\end{lemma}

\begin{proof}
Let $G=\text{Gal}(L/K)\cong \mathbb{Z}/2\mathbb{Z}$. Let $g$ be the genus of $C$ so that by \cite[Theorem 1]{MR0491735}, there is a   matrix $u\in \textup{Mat}_g(\mathbb{Z}_2)$  (the $\textit{twist matrix}$ associated to the formal group of $J$) such that 
\[J(K)_1/\N J(L)_1\cong G^g/(1-u)G^g.\]
Moreover, denoting by $T$ the completion  of $K^\text{nr}$ we  have (see \cite[Lemma]{MR0491735})  \[J(K)_1\cong \big\{\alpha \in \left(\mathcal{O}_T^\times\right)^g~:~F\alpha =u \alpha \big\},\]
where $F$ denotes the Frobenius automorphism of $T$. In particular, we have
\[J(K)_1[2]\cong \big\{\alpha \in \left\{\pm1\right\}^g~:~(1-u)\alpha =1\big\}.\]
Identifying the groups $G$ and $\{\pm1\}$ in the obvious way, $J(K)_1[2]$ is identified with the kernel of multiplication by $1-u$ on $G^g$. We now conclude by noting that the cokernel and kernel of an endomorphism of a finite abelian group have the same order.  
\end{proof}

\begin{lemma} \label{ordinary norm size}
Suppose that $K(J[2])/K$ has odd degree. Then we have
\[\textup{dim}  J(K)/\N J(L) \equiv 0 ~~\textup{(mod 2)}.\]
\end{lemma}  

\begin{proof}

  \Cref{odd degree extension} reduces to the case  $K(J[2])=K$. In this case, we claim that 
 \[\textup{dim} J(K)/\N J(L) =2g.\]
 To see this, consider the commutative diagram with exact rows
\[
\xymatrix{0\ar[r] & J(L)_1\ar[r]\ar[d]^{\N} & J(L)\ar[r]\ar[d]^{\N} & \tilde{J}(k)\ar[r]\ar[d]^{2} & 0\\
0\ar[r] & J(K)_1\ar[r] & J(K)\ar[r] & \tilde{J}(k)\ar[r] & 0,
}
\]
\[
\]
where $\tilde{J}/k$ denotes the special fibre of the N\'{e}ron model of $J$.
The assumption that all 2-torsion is defined over $K$ means that reduction to the special fibre is a surjection from $J(K)[2]$ to $\tilde{J}(k)[2]$. In particular, in the exact sequence arising from applying the snake lemma to the diagram above, the connecting homomorphism is trivial. Thus the sequence
\[0\longrightarrow J_1(K)/\N J_1(L)\longrightarrow J(K)/\N J(L)\longrightarrow \tilde{J}(k)/2\tilde{J}(k)\longrightarrow 0\]
is short exact. As $J$ is ordinary (and  all its  2-torsion is defined over $K$), we have
\[\big| \tilde{J}(k)/2\tilde{J}(k)\big|=\big|\tilde{J}(k)[2]\big|=2^g.\]
On the other hand, by \Cref{Lubin lemma} we have 
\[\big|J(K)/\N J(L)\big|=\big|J(K)_1[2]\big|=2^g\]
also, from which the result follows.
\end{proof}

\begin{cor} \label{good ordinary case}
Suppose that $K(J[2])/K$ has odd degree. Then \Cref{Kramer Tunnell} holds for $C/K$ and the extension $L/K$.  
\end{cor} 

\begin{proof}
Again by  \Cref{odd degree extension} we can assume that all the $2$-torsion of $J$ is defined over $K$. Under this assumption $f(x)$ splits over $K$, so $(\Delta_C,L/K)=1$. Similarly, both $C$ and $C^L$ have a $K$-rational Weierstrass point,  so $\epsilon(C/K)+\epsilon(C^L/K)=0$. By \Cref{ordinary norm size} we have $(-1)^{\text{ord}_2J(K)/\N J(L)}=1$, and e.g. by \cite[Proposition 3.23]{MR2534092} we have $w(J/L)=1$.
\end{proof}

For the purpose of giving examples we now describe how to construct hyperelliptic curves over $\mathbb{Q}$ whose Jacobians are good ordinary over $\mathbb{Q}_2$ and have all their $2$-torsion defined over an odd degree extension on $\mathbb{Q}_2$. Let $g\geq 2$ be an integer.

\begin{lemma} \label{constructing ordinary curves in char 2}
Let $f(x)\in\bar{\mathbb{F}}_2[x]$ be a monic separable polynomial of degree $g+1$ and let $h(x)\in \bar{\mathbb{F}}_2[x]$ be a polynomial of degree $\leq g$, coprime to $f(x)$. Then the hyperelliptic curve
\[C/\bar{\mathbb{F}}_2:y^2-f(x)y=h(x)f(x)\]
is ordinary.
\end{lemma}

\begin{proof}
One readily checks that the equation defining $C$ is smooth, hence defines a hyperelliptic curve over $\bar{\mathbb{F}}_2$ of genus $g$. Let $J$ be the Jacobian of $C$. As in the proof of \cite[Theorem 23]{MR3290950}  one sees that $\dim J(\bar{\mathbb{F}}_2)[2]=g$, hence $J$ is ordinary.
\end{proof}
 
\begin{lemma} \label{explicit ordinary cor}
Suppose $f(x)\in \mathbb{Z}[x]$ has odd leading coefficient and degree $g+1$, and suppose that  $f(x)$ $\textup{(mod}~ \textup{2)}$ is separable with each irreducible factor having odd degree. Further, let $h(x)\in \mathbb{Z}[x]$ have degree $\leq g$ be such that  $h(x)$ $\textup{(mod}~ \textup{2)}$ is coprime to $f(x)$ $\textup{(mod}~ \textup{2)}$. Then the Jacobian $J$ of the hyperelliptic curve
\[C: y^2=f(x)(f(x)+4h(x))\]
has good ordinary reduction over $\mathbb{Q}_2$. Moreover,  $\mathbb{Q}_2(J[2])/\mathbb{Q}_2$ has odd degree.
\end{lemma}

\begin{proof}
A change of variables over $\mathbb{Q}_2$ brings $C$ into the form $y^2-f(x)y=h(x)f(x)$, so $J$ has good ordinary reduction over $\mathbb{Q}_2$ by \Cref{constructing ordinary curves in char 2}. Moreover, both $f(x)$ and $f(x)+4h(x)$ reduce to separable polynomials over $\mathbb{F}_2$ whose irreducible factors have odd degree. It follows from Hensel's lemma that $f(x)(f(x)+4h(x))$ splits over an odd degree unramified extension of $\mathbb{Q}_2$, hence $\mathbb{Q}_2(J[2])/\mathbb{Q}_2$ has odd degree (and is unramified).
\end{proof}


\section{Proof of Theorems 1.1 and 1.8} \label{sec:main_thm_proofs}

We have now established enough cases of \Cref{Kramer Tunnell} to deduce \Cref{cases of the parity conjecture,thm:cases_of_kramer_tunnell}. For completeness, we explain these deductions below.

\begin{proof}[Proof of \Cref{thm:cases_of_kramer_tunnell}]
The case where $K$ is nonarchimedean is \Cref{kramer-tunnell for reals places}. The case where the quadratic extension is unramified is \Cref{prop:unram_cases_of_conjecture}. For the remaining cases combine Proposition \ref{prop:ram_cases_of_conjecture} (=\Cref{later_ramified_cases}) and Proposition \ref{good ordinary case}. These deal, respectively, with ramified extensions in odd residue characteristic, and ramified extensions in residue characteristic $2$.
\end{proof}

\begin{proof}[Proof of \Cref{cases of the parity conjecture}]
As explained in \Cref{subsec:red_to_local}, this follows by combining \Cref{selmer decomposition_intro} (=\Cref{selmer decomposition}) with \Cref{thm:cases_of_kramer_tunnell}.
\end{proof}

\end{document}